\documentclass[11pt,oneside]{amsart}

\usepackage{amsmath}

\usepackage{tikz-cd}

\usepackage{fix-cm}

\DeclareMathAlphabet{\mathcal}{OMS}{cmsy}{m}{n}
\usepackage[pagebackref,colorlinks,linkcolor=red,citecolor=blue,urlcolor=blue,hypertexnames=false]{hyperref}

\input xy
\xyoption{all}
\newdir{ >}{{}*!/-5pt/@{>}}

\usepackage[hmargin=2cm,vmargin=1.5cm]{geometry}

\usepackage{amsrefs}
\makeatletter
\newcommand{\makesynonym}[2]{%
  \AtEndDocument{\write\@auxout{\noexpand\@makesynonym{#1}{#2}}}%
}
\newcommand{\@makesynonym}[2]{%
  \global\expandafter\let\csname b@#1\expandafter\endcsname\csname b@#2\endcsname
}
\makeatother
\usepackage{anyfontsize}
\usepackage{mathrsfs}

\usepackage{bm}

\usepackage{bbm}

\usepackage{fancyhdr}
\usepackage{url}

\usepackage{tensor}

\usepackage{mathtools}
\usepackage{accents}

\usepackage{amscd}
\usepackage{amssymb}
\usepackage{latexsym}

\usepackage{graphicx}

\usepackage{hyperref}

\usepackage{leftindex}

\usepackage{enumitem}

\newcommand{\comment}[1]{}

\def\Z{\mathbb{Z}} 

\def\inc{\operatorname{inc}}

\def\cop{\operatorname{cop}}

\def\coco{\operatorname{co}}
\def\res{\operatorname{res}}
\def\ext{\operatorname{ext}}
\def\obj{\operatorname{obj}}

\def\bu{\bullet}

\def\topdf{\texorpdfstring}

\newcommand{\Alg}{\mathrm{Alg}}
\newcommand{\rings}{\mathrm{Rings}}
\newcommand{\aha}{{\Alg}}
\newcommand{\ahal}{\aha_\ell}
\newcommand{\lu}{\operatorname{lu}}

\def\tight{\operatorname{tight}}

\def\Grp{\mathfrak{Grp}}
\def\Add{\operatorname{Add}}
\def\triqui{\vartriangleleft}

\def\Mod{\operatorname{Mod}}

\numberwithin{equation}{section}

\theoremstyle{plain}
\newtheorem{thm}[equation]{Theorem}
\newtheorem{lem}[equation]{Lemma}
\newtheorem{coro}[equation]{Corollary} 
\newtheorem{prop}[equation]{Proposition}

\theoremstyle{definition}

\newtheorem{question}{Question}
\newtheorem{defi}[equation]{Definition} 

\newtheorem{ex}[equation]{Example}
\newtheorem{exas}[equation]{Examples}

\theoremstyle{remark} 

 \newtheorem{rem}[equation]{Remark}

\newtheorem{theorem}[equation]{Theorem}
\newtheorem*{theorem*}{Theorem}
\newtheorem{lemma}[equation]{Lemma}
\newtheorem*{lemma*}{Lemma}
\newtheorem{corollary}[equation]{Corollary}
\newtheorem{proposition}[equation]{Proposition}

\newtheorem{remark}[equation]{Remark}
\newtheorem{definition}[equation]{Definition}
\newtheorem*{definition*}{Definition}

\newtheorem*{question*}{Question}

\newtheorem{example}[equation]{Example}
\newtheorem{examples}[equation]{Examples}

\makeatletter
\def\revddots{\mathinner{\mkern1mu\raise\p@
\vbox{\kern7\p@\hbox{.}}\mkern2mu
\raise4\p@\hbox{.}\mkern2mu\raise7\p@\hbox{.}\mkern1mu}}
\makeatother 
\newcommand{\bgl}{\begin{equation}} 
\newcommand{\egl}{\end{equation}}
\newcommand{\bgloz}{\begin{equation*}} 
\newcommand{\egloz}{\end{equation*}}
\newcommand{\bgln}{\begin{eqnarray}} 
\newcommand{\egln}{\end{eqnarray}}
\newcommand{\bglnoz}{\begin{eqnarray*}} 
\newcommand{\eglnoz}{\end{eqnarray*}}
\newcommand{\btheo}{\begin{theorem}}
\newcommand{\etheo}{\end{theorem}}
\newcommand{\btheooz}{\begin{theorem*}}
\newcommand{\etheooz}{\end{theorem*}}

\newcommand{\blemma}{\begin{lemma}}
\newcommand{\elemma}{\end{lemma}}
\newcommand{\blemmaoz}{\begin{lemma*}}
\newcommand{\elemmaoz}{\end{lemma*}}
\newcommand{\bproof}{\begin{proof}}
\newcommand{\eproof}{\end{proof}}
\newcommand{\bbew}{\begin{beweis}}
\newcommand{\ebew}{\end{beweis}}
\newcommand{\bremark}{\begin{remark}\em}
\newcommand{\eremark}{\end{remark}}
\newcommand{\bdefin}{\begin{definition}}
\newcommand{\edefin}{\end{definition}}
\newcommand{\bdefinoz}{\begin{definition*}}
\newcommand{\edefinoz}{\end{definition*}}
\newcommand{\bex}{\begin{example}\em}
\newcommand{\eex}{\end{example}}
\newcommand{\bexs}{\begin{examples}}
\newcommand{\eexs}{\end{examples}}

\newcommand{\bprop}{\begin{proposition}}
\newcommand{\eprop}{\end{proposition}}
\newcommand{\bcor}{\begin{corollary}}
\newcommand{\ecor}{\end{corollary}}
\newcommand{\bfa}{\begin{cases}} 
\newcommand{\efa}{\end{cases}}

\newcommand{\bquestion}{\begin{question}}
\newcommand{\equestion}{\end{question}}
\newcommand{\bquestionoz}{\begin{question*}}
\newcommand{\equestionoz}{\end{question*}}
\newcommand{\cB}{\mathcal B}
\newcommand{\cC}{\mathcal C}
\newcommand{\cD}{\mathcal D}
\newcommand{\cE}{\mathcal E}
\newcommand{\cF}{\mathcal F}
\newcommand{\cG}{\mathcal G}
\newcommand{\cH}{\mathcal H}
\newcommand{\cI}{\mathcal I}

\newcommand{\cK}{\mathcal K}
\newcommand{\cL}{\mathcal L}

\newcommand{\cO}{\mathcal O}
\newcommand{\cP}{\mathcal P}

\newcommand{\cR}{\mathcal R}
\newcommand{\cS}{\mathcal S}
\newcommand{\cT}{\mathcal T}
\newcommand{\cU}{\mathcal U}

\newcommand{\cX}{\mathcal X}
\newcommand{\cY}{\mathcal Y}

\def\Iz{\mathbbm{I}}
\def\Kz{\mathbbm{K}}
\def\Mz{\mathbbm{M}}
\def\Nz{\mathbbm{N}}

\def\Rz{\mathbbm{R}}

\def\Zz{\mathbbm{Z}}
\def\Sz{\mathbbm{S}}

\newcommand{\fB}{\mathfrak B}
\newcommand{\fC}{\mathfrak C}

\newcommand{\fF}{\mathfrak F}

\newcommand{\fL}{\mathfrak L}

\newcommand{\fP}{\mathfrak P}

\newcommand{\mfc}{\mathfrak c}
\newcommand{\mfd}{\mathfrak d}

\newcommand{\mft}{\mathfrak t}

\newcommand{\mfw}{\mathfrak w}
\newcommand{\mfx}{\mathfrak x}
\newcommand{\mfy}{\mathfrak y}
\newcommand{\mfz}{\mathfrak z}
\newcommand{\an}[1]{``#1''} 
\newcommand{\ti}{\tilde}

\newcommand{\lori}{\longrightarrow}
\newcommand{\lra}{\lori}

\newcommand{\ma}{\mapsto} 
\newcommand{\onto}{\twoheadrightarrow} 
\newcommand{\into}{\hookrightarrow} 
\newcommand{\isom}{\xrightarrow{\raisebox{-1ex}[0ex][0ex]{$\sim$}}} 

\def\SEMI{\mbox{$\times\kern-2pt\vrule height5pt width.6pt \kern3pt $}}

\newcommand{\Hom}{{\rm Hom}\,}

\newcommand{\Sp}{{\rm Sp\,}}

\newcommand{\id}{{\rm id}}

\renewcommand{\ker}{{\rm ker}\,}

\newcommand{\reg}{^\times} 
\newcommand{\lspan}{{\rm span}} 
\newcommand{\defeq}{\mathrel{:=}} 

\newcommand{\dop}{\text{: }} 
\newcommand{\ilim}{\varinjlim} 
\newcommand{\plim}{\varprojlim} 
\newcommand{\supp}{{\rm supp}} 
\newcommand{\lge}{\left\{} 
\newcommand{\rge}{\right\}} 
\newcommand{\lru}{\left(} 
\newcommand{\rru}{\right)} 
\newcommand{\lsp}{\left\langle} 
\newcommand{\rsp}{\right\rangle} 
\newcommand{\rukl}[1]{\lru #1 \rru} 
\newcommand{\gekl}[1]{\lge #1 \rge} 
\newcommand{\spkl}[1]{\lsp #1 \rsp} 
\newcommand{\menge}[2]{\gekl{ #1 \dop #2 }} 
\def\bf1{\mathbf{1}}
\newcommand{\dom}{{\rm dom\,}}
\newcommand{\ran}{{\rm ran\,}}

\newcommand{\bmu}{\bm{u}}

\newcommand{\bmsigma}{\bm{\sigma}}
\newcommand{\bmtau}{\bm{\tau}}

\newcommand{\bmF}{\bm{F}}

\newcommand{\oset}[2]{%
  \mathop{#2}\limits^{\vbox to -1.66ex{%
  \kern -1.4ex\hbox{$#1$}\vss}}}

\newcommand{\C}{\mathscr{C}}

\newcommand{\crG}{\mathscr{G}}

\newcommand{\Gn}{{G^{(0)}}}
\newcommand{\Gnu}{{G^{(\nu)}}}
\newcommand{\Gnum}{{G^{(\nu - 1)}}}
\newcommand{\Gnup}{{G^{(\nu + 1)}}}
\newcommand{\Hn}{{H^{(0)}}}
\newcommand{\Hnu}{{H^{(\nu)}}}
\newcommand{\Hnum}{{H^{(\nu - 1)}}}
\newcommand{\Hnup}{{H^{(\nu + 1)}}}

\newcommand{\Phinu}{{\Phi^{(\nu)}}}
\newcommand{\Phinup}{{\Phi^{(\nu + 1)}}}

\newcommand{\rmr}{c} 
\newcommand{\rms}{d} 

\newcommand{\mor}{{\rm mor}\,}

\newcommand{\sfC}{\mathsf{C}}

\newcommand{\acts}{\curvearrowright}

\newcommand{\coli}{\operatorname{colim}}
\newcommand{\hocolim}{\operatorname{hocolim}}
\newcommand{\hocofi}{\operatorname{hocofiber}}

\newcommand{\Idem}{{\rm Idem}}

\newcommand{\Gu}{G^{(0)}}
\newcommand{\Hu}{H^{(0)}}

\newcommand{\CO}{\cC \cO}

\newcommand{\whL}{\widehat{L}}

\newcommand{\bbs}{\mathbbm{s}}
\newcommand{\bbt}{\mathbbm{t}}

\newcommand{\bbK}{\mathbbm{K}}
\newcommand{\bbKK}{\bbK\bbK}

\newcommand{\bmalpha}{\bm{\alpha}}

\newcommand{\bmfw}{\bm{\mfw}}
\newcommand{\bmfx}{\bm{\mfx}}
\newcommand{\bmfy}{\bm{\mfy}}
\newcommand{\bmfz}{\bm{\mfz}}

\DeclareFontFamily{U}{mathb}{\hyphenchar\font45}
\DeclareFontShape{U}{mathb}{m}{n}{
      <5> <6> <7> <8> <9> <10>
      <10.95> <12> <14.4> <17.28> <20.74> <24.88>
      mathb10
      }{}
\DeclareSymbolFont{mathb}{U}{mathb}{m}{n}

\DeclareMathSymbol{\sqbullet}{1}{mathb}{"0D}

\newcommand{\bigast}{\mathop{\scalebox{1.5}{\raisebox{-0.2ex}{$\ast$}}}}

\usepackage{newtxtext}
\usepackage{newtxmath}

\numberwithin{equation}{section}
\bigskip

\begin{document}

\title{Discretisation in algebraic $K$-theory of independent ample groupoids and their Steinberg algebras}

\thispagestyle{fancy}

\author{Guillermo Corti{\~n}as}

\author{Xin Li}

\address{Guillermo Corti{\~n}as, IMaS, Departamento de Matem\'atica, Facultad de Ciencias Exactas y Naturales, Universidad de Buenos Aires and CONICET. Ciudad Universitaria, Pabell\'on I. C1428EGA Buenos Aires, Argentina}
\email{gcorti@dm.uba.ar}

\address{Xin Li, School of Mathematics and Statistics, University of Glasgow, University Place, Glasgow G12 8QQ, United Kingdom}
\email{Xin.Li@glasgow.ac.uk}


\thanks{The first-named author was partially supported by CONICET through PIP 11220250100969CO, and by UBA through UBACyT 20020220300206BA. The second-named author has received funding from the European Research Council (ERC) under the European Union's Horizon 2020 research and innovation programme (grant agreement No. 817597).}

\begin{abstract}
 An independent ample groupoid $G$ comes equipped with a discrete groupoid $G_d$ and a proper \'etale correspondence $\Omega:G_d\to G$. Discretisation with respect to a homological invariant $\cH$ holds when $\cH(\Omega)$ is an equivalence. We show for example that discretisation with respect to the $K$-theory spectrum $\bbK \fB_G$ of the permutative category of bisections always holds, establish excision and continuity for this invariant, and use all this and explicit resolutions by independent groupoids to compute it for several families of ample groupoids. We also prove discretisation with respect to invariants of the Steinberg algebra $\ell[G]$ such as (homotopy) algebraic $K$-theory in several particular cases, and carry out computations. 
\end{abstract}

\maketitle

\tableofcontents

\setlength{\parindent}{0cm} \setlength{\parskip}{0.5cm}

\section{Introduction}

This paper is concerned with computing homological invariants of ample groupoids. Ample groupoids are groupoids (i.e., small categories with invertible morphisms) equipped with a topology such that all structure maps become local homeomorphisms and its space of objects, also called the unit space, is a totally disconnected, locally compact Hausdorff space. For example, discrete groupoids are ample. They are also Morita equivalent to disjoint unions of groups, and thus homological invariants of a discrete groupoid decompose into invariants of groups. Next in the level of complexity are the groupoids that we call independent; for example, the universal groupoid of an inverse semigroup is independent. If $G$ is independent then there are a discrete groupoid $G_d$ and a proper \'etale correspondence $\Omega:G_d\to G$. Roughly speaking, we say that $G$ satisfies discretisation with respect to a homological invariant $\cH$ if 
$\Omega$ induces an equivalence $\cH(G_d)\isom \cH(G)$; this implies that $\cH(G)$ decomposes into copies of $\cH$ applied to a family of groups. Moreover, any ample groupoid $G$ admits a resolution by independent groupoids; so establishing discretisation with respect to $\cH$ for all independent groupoids in a resolution of $G$ helps one compute $\cH(G)$. 

Discretisation has been first developed in the context of topological $K$-theory of $C^*$-algebras attached to independent groupoids (see \cites{CEL1,CEL2,Nor,Li_CMP,Mil25b,LM}). In that setting, discretisation is derived as a consequence of the Baum-Connes conjecture with coefficients. In this paper, we establish discretisation results without relying on isomorphism conjectures such as the Baum-Connes conjecture or its algebraic analogues.

We establish discretisation for an independent ample groupoid $G$ and an invariant $\cH$ in several examples. Below is an incomplete list of such examples; for details, see the indicated cross-references. We write $K$, $KH$, $HH$, $HC$ and $HP$ for the algebraic and homotopy algebraic $K$-theory of a ring and the Hochschild, cyclic and periodic cyclic homology of an algebra. 

\begin{itemize}
\item[i)] $G$ is any independent groupoid and $\cH(G)=\bbK \fB_G$ is the algebraic $K$-theory spectrum of the permutative category $\fB_G$ introduced in \cite{Li22} (see Corollary \ref{cor:Disc}). 
\item[ii)] $G$ is proper and is the universal groupoid of an inverse semigroup and $\cH(G)$ is either $K$, $KH$, $HH$ or $HC$ of the Steinberg algebra $\ell[G]$ over a commutative unital ring $\ell$ (Corollary \ref{coro:discprop} and Examples \ref{exas:discprop}).
\item[iii)]  $G$ is the independent groupoid associated to a self-similar action of a discrete groupoid on a higher directed graph and $\cH$ is either $KH(\ell[G])$ with $\ell$ as above, or $HP(\ell[G])$ with $\ell$ a field of characteristic zero (Theorem \ref{thm:discg} and Examples \ref{exas:discg}). 
\item[iv)] $G$ is the universal groupoid of the Artin inverse semigroup $S(\Gamma)$ associated to a finite undirected graph $\Gamma$ whose complementary graph is connected and $\cH=KH(\ell[G])$ (Corollary \ref{coro:discartin} and Example \ref{ex:discartin}).
\end{itemize}
For $\cH$ as in i) above we also establish excision (Corollaries \ref{cor:Excision} and \ref{cor:Exci2}) and continuity (Corollary \ref{cor:continuity}), and use all these properties to compute $\cH(G)$  for graph groupoids (see \S~\ref{sss:Graphs}), groupoids arising from self-similar groups (see \S~\ref{sss:SelfSimilarGps}), groupoids arising from piecewise affine transformations (see \S~\ref{sss:AffineTrafos}), and AF groupoids (see \S~\ref{sss:AF}).  
In case iv) above we also show that discretisation with respect to $K$ is obstructed by the possible nonvanishing of Bass' twisted nil $K$-groups (Corollary \ref{coro:kacegamma}). We also compute $K$ and $KH$ of the Steinberg algebra of the tight groupoid of $S(\Gamma)$ (Theorem \ref{thm:bquotient}, Proposition \ref{prop:kaceqgamma} and Corollary \ref{coro:kaceqgamma}).

The particular case of iii) for self-similar group actions on $1$-graphs was done in \cite{acm}; its analogue for self-similar actions of discrete groupoids on $1$-graphs and topological $K$-theory of the $C^*$-algebra and groupoid homology of the Toeplitz groupoid was obtained by A. Miller and B. Steinberg in \cite{steimil}. In his article \cite{fletcher}, J. Fletcher showed that the Toeplitz $C^*$-algebra of a higher graph is an iterated Cuntz-Pimsner $C^*$-algebra. Inspired by this, we show that for $G$ as in iii) above, the Steinberg algebra
$\ell[G]$ is an iterated Cuntz-Pimsner $\ell$-algebra in the sense of \cite{toke}. To prove iii) we then use results of T. Lescure, reviewed in Appendix \ref{app2:lescure}, on a Pimsner-Voiculescu sequence for Cuntz-Pimsner algebras. For $\Gamma$ as in iv), the calculation of the topological $K$-theory of the $C^*$-algebra of both the universal and tight groupoids of $S(\Gamma)$ is due to N. Ivanov \cite{iva}.

The rest of this paper is organized as follows. In Section \S~\ref{sec:prelis} we give some preliminaries on independent groupoids. Subsection 
\S~\ref{subsec:lfws} studies locally finite weak semilattices and the topology of their character spaces. Subsection \S~\ref{subsec:invact} considers 
inverse semigroup actions $S\acts L$ via poset isomorphisms between downsets, and their dual actions $S\acts \whL$ on their character spaces. A groupoid is independent if it is isomorphic to the transformation groupoid $G=S\ltimes \whL$ of such an action (see \ref{defi:indep}).  One may also consider the 
discrete groupoid $G_d=S\ltimes L^\times$ associated to the action of $S$ on the nonzero elements of $L$. Discretisation is introduced in Section 
\ref{sec:disc}; it is expressed first in terms of a canonical correspondence $\Omega:G_d\to G$
\S~\ref{ss:DiscFormul} and then in terms of a canonical homomorphism of algebras $\theta:\ell[G_d]\to M_{L^\times}(\ell[G])$ \S~\ref{subsec:detour}. The results of Section \S~\ref{sec:discalg} establish discretisation for the cases ii), iii) and iv) of the list above. Subsection \S~\ref{ss:ProperGPD} establishes discretisation for a class of proper independent 
groupoids, including proper universal groupoids of inverse semigroups (Theorem \ref{thm:discprop} and Corollary \ref{coro:discprop}). In Subsection \S~\ref{ss:sshrg} we 
consider the case of inverse semigroups attached to self-similar actions of discrete groupoids on higher rank graphs, and establish discretisation for the 
Steinberg algebras of their universal groupoids in bivariant algebraic $K$-theory, and as a consequence also in $KH$ and $HP$ (Theorem \ref{thm:discg} and Examples \ref{exas:discg}). In Subsection \S~\ref{ss:Artin} we consider the 
algebra $C(\Gamma)$ of the Artin inverse semigroup $S(\Gamma)$ constructed from a graph $\Gamma$, which is closely related to the right-angled Artin group of $\Gamma$. Assuming that the complementary graph $\Gamma^{\cop}$ is connected, we establish discretisation for $C(\Gamma)$ with respect to a class of homology theories of algebras which includes homotopy algebraic $K$-theory (see Corollary \ref{coro:discartin}). 
We also show how discretisation for algebraic $K$-theory of $C(\Gamma)$ depends on the vanishing of twisted Bass' nil-$K$-theory groups (Corollary \ref{coro:kacegamma}). We use these discretisation results and a canonical independent resolution to compute the algebraic and homotopy algebraic $K$-theory of the Steinberg algebra of the tight groupoid of $S(\Gamma)$, $C_Q(\Gamma)$  (Theorem \ref{thm:bquotient} and Proposition \ref{prop:kaceqgamma}).  In Section \ref{s:Disc_KBG} we establish discretisation for algebraic $K$-theory spectra of ample groupoids as constructed in \cite{Li22} (see Corollary \ref{cor:Disc}), obtaining result i) of the list above. Next we establish excision (see Corollary \ref{cor:Excision}) and continuity (Corollary \ref{cor:continuity}) for such algebraic $K$-theory spectra. Combined with the idea of independent resolutions (see \cite{LM}), this leads to descriptions of algebraic $K$-theory spectra of ample groupoids in terms of simpler building blocks. We demonstrate this in the case of several examples, for graph groupoids (see \S~\ref{sss:Graphs}), groupoids arising from self-similar groups (see \S~\ref{sss:SelfSimilarGps}), groupoids arising from piecewise affine transformations (see \S~\ref{sss:AffineTrafos}), and AF groupoids (see \S~\ref{sss:AF}). 

In Appendix \ref{app1:proper} we prove some results on the local structure of proper ample groupoids, including the local finiteness of their inverse semigroups of compact open bisections (Proposition \ref{prop:semilocafin}). These are used in Section \ref{sec:disc} to prove discretisation for a wide class of proper independent groupoids.

In Appendix \ref{app2:lescure} we review and reprove T. Lescure's results from \cite{lescure} about the $K$-theory of the Pimsner-Toeplitz rings, as these results are used in \S~\ref{ss:sshrg}.

Throughout this paper, $\ell$ denotes a nonzero commutative ring with unit.

\section{Preliminaries}\label{sec:prelis}

\subsection{Locally finite weak semilattices}\label{subsec:lfws}
Let $P$ be a poset and $x\in P$. Write 
$$
(x\downarrow)=\{y\in P\colon y\le x\},\, (x\uparrow)=\{y\in P\colon y\ge x\}.
$$
In case $P$ happens to have a (necessarily unique) minimum element, we shall denote it as $0$ and call it the zero element of $P$. 
\bdefin
A \emph{weak semilattice} $L$ is a poset with a zero element $0$,
with the property that for all $l_1, l_2 \in L$ the set $(l_1\downarrow l_2)$ of maximal elements of $(l_1\downarrow)\cap (l_2\downarrow)$ is finite and nonempty. 
A \emph{semilattice} is a weak semilattice such that $(l_1\downarrow l_2)$ has a unique element for all $l_1,l_2\in L$; in this case we write $l_1l_2$ for that unique element.

\setlength{\parindent}{0.5cm} \setlength{\parskip}{0cm}

A subset $\bar{F}$ of a weak semilattice is called $\downarrow$-closed if for all $x, y \in \bar{F}$, $x \downarrow y \subseteq \bar{F}$. We call a weak semilattice $L$ \emph{locally finite} if for any finite subset $F \subseteq L$ there exists a finite subset $\bar{F} \subseteq L$ which is $\downarrow$-closed and such that $F \subseteq \bar{F}$.
\edefin
\setlength{\parindent}{0cm} \setlength{\parskip}{0.5cm}

In the following, we will use the following notation: Given a poset $L$ with zero element $0$ and a subset $F \subseteq L$, we write $F\reg \coloneq F \setminus \gekl{0}$. A subset of $L$ is called a down-set if whenever $l$ is in the subset and $k \in L$ satisfies $k \leq l$, then $k$ has to be in the subset. 

\begin{ex}
Any semilattice $E$ is a locally finite weak semilattice. 
\end{ex}

\bdefin\label{def:chi}
Let $L$ be a locally finite weak semilattice. A character on $L$ is a map $\chi: \: L \to \gekl{0,1}$ with the properties that $\chi(0) = 0$, that $
\chi$ is not identically $0$, and that for all $l_1, l_2 \in L$, $\chi(l_1) = 1$ and $\chi(l_2) = 1$ hold simultaneously if and only if there exists $l \in l_1 \downarrow l_2$ with $\chi(l) = 1$. Let $\widehat{L}$ be the set of all characters on $L$. We equip $\widehat{L}$ with the topology of pointwise convergence. For $0 \neq l, l_1,\dots,l_n \in L$, we write 
$$\widehat{L}(l) \coloneq \menge{\chi \in \widehat{L}}{\chi(l) = 1},$$ and 
\begin{equation}\label{eq:Ll1ln}
    \widehat{L}(l;l_1,\dots,l_n)=\widehat{L}(l)\cap\bigcap_{i=1}^n\widehat{L}(l_i)^c.
\end{equation}
Here ${X}^c=\widehat{L}\setminus X$. 
\edefin
\begin{ex}\label{ex:ppal}
 Let $x\in L^\times=L\setminus\{0\}$. Consider the map 
 \[
 \cP_x=\chi_{(x\uparrow)}:L\to \{0,1\},\,\, \chi_{(x\uparrow)}(l) = 1  \iff l \in (x\uparrow).
 \]
 $\cP_x$ is a character, the \emph{principal character} associated to $x$. 
\end{ex}
\begin{lem}\label{lem:basieljá}
\item[i)] If $L$ is a weak semilattice, then $\cB=\{\widehat{L}(l;l_1,\dots,l_n)\colon n\ge 0,\, l,l_1,\dots,l_n\in L^\times\}$ is a family of compact open subsets of $\widehat{L}$ and basis for its topology.
\item[ii)] If in addition $L$ is locally finite, then every compact open subset of $\widehat{L}$ is a finite disjoint union of elements of 
$\cB$.
\end{lem}
\begin{proof}
\begin{itemize}
    \item[i)] For $l;l_1,\dots,l_n\in L$ put $X=\{0,1\}^L\supset U_{l;l_1,\dots,l_n}=\{x \in X \colon x_l=1,\, x_{l_i}=0\,\forall i\}$. We have $\widehat{L}(l;l_1,\dots,l_n)=\widehat{L}\cap U_{l;l_1,\dots,l_n}$; since the $U_{l;l_1,\dots,l_n}$ are a basis of $X$, $\cB$ is a basis of its subspace $\widehat{L}$. In particular $\widehat{L}(l_i)$ is open and thus $\widehat{L}(l;l_1,\dots,l_n)=\widehat{L}(l)\cap\bigcap_{i=1}^n \widehat{L}(l_i)^c$ is closed in $L(l)$, so to prove that $\widehat{L}(l;l_1,\dots,l_n)$ is compact, it suffices to show that $L(l)$ is. One checks that the restriction map
\[
\res:\widehat{L}(l)\to \widehat{(l\downarrow)},\, \chi\mapsto \chi \vert_{(l\downarrow)}
\]
is a homeomorphism with inverse 
\[
\res^{-1}(\psi)(x)=\left\{\begin{matrix}
 1 & \text{ if } (l\downarrow x)\cap\supp(\psi)\ne \emptyset\\
 0 & \text{ else.}
\end{matrix}\right.
\]
Hence we are reduced to showing that $K:=\widehat{(l\downarrow)}$ is compact, which will follow if we prove it is closed in $Y:=\{0,1\}^{(l\downarrow)}$, since the latter space is compact. Take a net $(\psi_{\lambda})$ of elements of $K$ converging to $\psi\in Y$. Then $\psi(l)=\psi_{\lambda}(l)=1$. It is also straightforward to show that $\psi$ satisfies the remaining conditions of Definition \ref{def:chi} that make it an element of $K$.

\item[ii)] By part i), every compact open of $\widehat{L}$ is a finite union $V=\bigcup_{i=1}^nV_i$ of elements of $\cB$. Let $[[n]]=\{1,\dots,n\}$ and $\cP=\{I\subset [[n]]\colon I\ne\emptyset\}$. For each $I\in\cP$, put $I'=[[n]]\setminus I$ and set $$V_I=\bigcap_{i\in I}V_i\cap\bigcap_{j\in I'}V^c_j.$$ Then $V=\sqcup_{I\in\cP}V_I$, so it is enough to check that each $V_I$ is a disjoint union of elements of $\cB$. Because $L$ is assumed to be locally finite, and the $V_i$ are finitely many and belong to $\cB$, there is a $\downarrow$-closed finite set $F\subset L$ such that for each $i$ there are $l,l_1,\dots,l_p\in F$, all of them depending on $i$, such that $V_i=\widehat{L}(l;l_1,\dots,l_p).$ Let 
$$
\cB_F=\{\widehat{L}(l;l_1,\dots,l_p)\in \cB\colon p\ge 0,\, l,l_1,\dots,l_p\in F\}
$$
Using the definition of character \ref{def:chi}, we see that the $V_I$ are finite unions of elements of $\cB_F$. For example
\begin{gather}
\widehat{L}(l;l_1,\dots,l_p)\cap \widehat{L}(m;m_1,\dots,m_q)=\bigcup_{x\in (l\downarrow m)}\widehat{L}(x;l_1,\dots,l_p,m_1,\dots,m_q)\label{eq:intL}\\
\widehat{L}(l;l_1,\dots,l_p)\cap \widehat{L}(m;m_1,\dots,m_q)^c=\bigcup_{i=1}^q\bigcup_{x\in (l\downarrow m_i)}\widehat{L}(x;l_1,\dots,l_p,m)\label{eq:diffL}
\end{gather}
and since $F$ is $\downarrow$-closed, all $x$ appearing above are in $F$.  Note, furthermore, that if $l$ and $m$ are comparable, then the union of \eqref{eq:intL} has a single term, and if $l$ and $m_i$ are comparable, then the $i$-th union of \eqref{eq:diffL} does. Otherwise the $x$ appearing in \eqref{eq:intL} are strictly smaller than both $l$ and $m$ and those in \eqref{eq:diffL} are strictly smaller than both $l$ and some $m_i$. So we can write each $V_I$ as finite union, disjointify it, write each of the new sets as finite unions of elements of $\cB_F$ and if some of the unions have more than one term, we disjointify again; the process ends because $\cB_F$ is finite and because in each step, the finite unions we obtain involve elements of the form \eqref{eq:Ll1ln} with at least one of the $l, l_i$ strictly smaller than in the previous finite union we had disjointified.\qedhere
\end{itemize}
\end{proof}

\begin{rem}\label{rem:basilejá}
    The same argument as in the proof of part i) of Lemma \ref{lem:basieljá} shows that if $l_1,\dots,l_n\in L^\times$, then restriction of functions defines a homeomorphism of compact spaces $\bigcup_{i=1}^n\widehat{L}(l_i)\cong \widehat{\bigcup_{i=1}^nL(l_i)}$.
\end{rem}

We will show in Proposition \ref{prop:coeljá} that if $L$ is a locally finite weak semilattice, then every compact open subset $U\subset \whL$ is itself the character space of some locally finite weak semilattice. In view of Lemma \ref{lem:basieljá}, every such $U$ is a disjoint union of subsets of the form $\whL(l; M)$. We shall presently introduce the quotient $L/M$, and show in Lemma \ref{lem:morebasis} that $\whL(l; M)$ is the character space of $L/M$. 

Let $L$ be a weak semilattice and $\emptyset\ne M\subset L$ a downset. Define an equivalence relation $l\sim_Ml'$ iff $l=l'$ or $l,l'\in M$. The \emph{quotient} $L/M$ is the set $L/\sim_M$ equipped with the partial order $[l_1]\le [l_2]$ if either $l_1\le l_2$ or $l_1\in M$. 

\begin{lem}\label{lem:quot}
Let $L$ be a weak semilattice and $\emptyset\ne M\subset L$ a downset; write $[l]$ for the class in $L/M$ of an element $l\in L$. 
\item[i)]
\[
([l_1]\downarrow [l_2])=\left\{\begin{matrix}
[0] & \text{ if } (l_1\downarrow l_2)\subset M\\
[(l_1\downarrow l_2)\setminus M]&\text{ otherwise.}
\end{matrix}\right.
\]
\item[ii)] $L/M$ is a weak semilattice, and is locally finite if $L$ is.
\item[iii)] Consider the closed subspace $\hat{L}\supset C=\{\phi\colon \phi(M)=\{0\}\}$. The projection $\pi:L\to L/M$ induces a homeomorphism $\widehat{L/M}\cong C$.
\end{lem}
\begin{proof}
  Straightforward.   
\end{proof}

Let $L$ be a locally finite weak semilattice, $l\in L$ and $M\subset L$ a downset. 
Put
\[
\widehat{L}(l;M)=\widehat{L}(l)\setminus (\bigcup_{m\in M}\widehat{L}(m)),\,\, L(l;M)=(l\downarrow)\setminus M.
\]
We abbreviate $L(l):=L(l;\emptyset)=(l\downarrow)$. We say that $M$ is \emph{finitely generated} if it can be expressed as a finite union $M=\bigcup_{i=1}^nL(l_i)$. 

\begin{lem}\label{lem:morebasis}
Let $L$ be a locally finite weak semilattice, $p\ge 1$, $l,m_1,\dots,m_p\in L$, and set $U \coloneq \widehat{L}(l;m_1,\dotsc,m_p)$.

    \item[i)] There exist $q\ge 0$ and $l_1,\dots, l_q\in L(l)$ such that
    $U=\widehat{L}(l;l_1,\dots,l_q)$.
    \item[ii)] Let $l,l_1,\dots,l_q$ be as in i) and let $M=\bigcup_{i=1}^qL(l_i)$. Then $U=\widehat{L}(l;M)$.
    \item[iii)] There is a natural homeomorphism $U\cong \widehat{L(l)/M}$.
\end{lem}
\begin{proof} If $\phi\in\widehat{L}(l)$ and $m\in L$, then it follows from the definition of character that $m\notin\supp(\phi)$ if and only if $\phi(l\downarrow m)=\{0\}$. This proves part i). If $x\le y\in L$, then $(x\downarrow y)=\{x\}$, so $\phi(y)=0\Rightarrow \phi(x)=0$. Hence $\phi(l_1)=\dots=\phi(l_q)=0\iff \phi(M)=0$, which proves ii). Part iii) follows by combining ii) of the current lemma with iii) of Lemma \ref{lem:quot}.
\end{proof}

To complete the preliminaries of Proposition \ref{prop:coeljá}, next we introduce coproducts of locally finite weak semilattices and show in Lemma \ref{lem:coprolfws} that a finite disjoint union of character spaces corresponds to the character space of the coproduct. 

The \emph{coproduct} of two locally finite weak semilattices $L_1,L_2$ is their coproduct as pointed sets, equipped with the partial order that makes $L_1$ and $L_2$ sub-semilattices, and any two nonzero elements $l_1\in L_1$ and $l_2\in L_2$ incomparable. 
\begin{lem}\label{lem:coprolfws}
 Let $n\ge 1$ and let $L_1,\dots,L_n$ be locally finite weak semilattices and let 
 $\bigvee_{i=1}^nL_i$ be their coproduct. Then the maps
 \[
 \ext_i:\widehat{L_i}\to\widehat{\bigvee_{1\le i\le n}L_i},\, \ext_i(\phi)(x)=\left\{\begin{matrix}
  \phi(x) & \text{  if } x\in L_i\\
  0& \text{ otherwise }
 \end{matrix}\right.
 \]
 induce a homeomorphism
 \[
 \ext:\bigsqcup_{i=1}^n\widehat{L}_i\to\widehat{\bigvee_{1\le i\le n}L_i}.
 \]
\end{lem}
\begin{proof}
Write $L=\bigvee_{i=1}^nL_i$. Because elements in distinct $L_i$ are incomparable,   
$\ext_i$ is well-defined for each $i$. If $l\in L_j$ and $m_1,\dots m_p\in L$ then there are $q\ge 0$ and $x_1,\dots,x_q\in L_j$ such that $U=\widehat{L}(l;m_1,\dots, m_p)=\widehat{L}(l; x_1,\dots,x_q)$, and we have $\ext_i^{-1}(U)=\widehat{L_j}(l; x_1,\dots,x_q)$. By Lemma \ref{lem:basieljá} i), this implies that $\ext_i$ is continuous. Thus $\ext$ is well-defined and continuous. Further, it is surjective, again by the incomparability between elements of distinct $L_i$, and open, since it maps $\hat{L_i}(l;x_1,\dots, x_q)$ to $\widehat{L}(l;x_1,\dots,x_q)$ for all $i$ and $l,x_1,\dots, x_q\in L_i$. Injectivity of $\ext$ is also clear.
\end{proof}

\begin{prop}\label{prop:coeljá}
Let $L$ be a locally finite weak semilattice and $\emptyset\ne U\subset \widehat{L}$ a compact open subset. Then there exists a locally finite weak semilattice  $L'$, isomorphic to a wedge of subquotients of $L$, such that $U\cong\widehat{L'}$ and such that $L'$ has a finite, nonzero number of maximal elements. 
\end{prop}
\begin{proof}
By Lemma \ref{lem:basieljá} ii), $U=\bigsqcup_{i=1}^nU_i$ with each $U_i$ a member of the basis $\cB$ of part i) of the same lemma. Thus by Lemma \ref{lem:morebasis}, $U_i=\widehat{L(l_i)/M_i)}$ for some $l_1,\dots,l_n\in L$ and some $M_i\subset (l_i\downarrow)$. By Lemma \ref{lem:coprolfws}, for  $L'=\bigvee_{i=1}^n(L(l_i)/M_i)$ we have $U\cong \widehat{L'}$. The proof is concluded by observing that the maximal elements of $L'$ are $[l_1],\dots,[l_n]$.
\end{proof}    
Put
 \begin{gather*}
\CO_d(L)=\{V\subset L^\times\colon \exists n\ge 1,\, l_1,\dots, l_n\in L^\times, 0\in M_i\subsetneq L(l_i)\\ \text{ finitely generated such that } V=\bigsqcup_{i=1}^nL(l_i;M_i)\}. 
\end{gather*}
\begin{lem}\label{lem:principal}
    Let $L$ be a locally finite weak semilattice, and let $\cP:L^\times\to \widehat{L}$, $x\mapsto\cP_x$ be as in Example \ref{ex:ppal}. 
    \item[i)] $\cP$ is injective and its image is dense.
    
    \item[ii)] The map $\CO(\widehat{L})\to \CO_d(L)$, $U\mapsto U_d:=\cP^{-1}(U)$ is an isomorphism of lattices. 
\end{lem}
\begin{proof}
  If $x_1,x_2\in L^\times$ and $\cP_{x_1}=\cP_{x_2}$, then $\cP_{x_i}(x_j)=\cP_{x_j}(x_j)$ imply that $x_1\le x_2\le x_1$, so $x_1=x_2$. Moreover, if  $U\subset \hat{L}$ is open and $\chi\in U$, then by Lemma \ref{lem:basieljá}, there are $l\in L^\times$ and $M\subset L(l)$ such $\chi\in \widehat{L}(l;M)$. Density follows from this since $\cP_l\in \widehat{L}(l,M)$. This proves i). Next observe that, if $M\subset L(l)$ is a downset,
  \[
  (\widehat{L}(l,M))_d=\{x\in L^\times\colon \chi_{x\uparrow}(l)=1,\chi_{x\uparrow}(M)=0\}=\{x\in L(l)\colon x\notin M\}=L(l;M). 
  \]
  It follows from this, part ii) of Lemma \ref{lem:basieljá}, and the fact that the inverse image map $\cP^{-1}$ is order and disjointness-preserving, that $\cP^{-1}$ maps $\CO(\widehat{L})\onto \CO_d(L)$. Moreover one checks, using i) of the present lemma, that $\cP^{-1}$ is a bijection with inverse $V\mapsto \widebar{V}$. 
\end{proof}
\begin{rem}\label{rem:Ud=Ltimes}
Let $U=\bigsqcup_{i=1}^n\widehat{L}(l_i;M_i)\in\CO(\widehat{L})$ and let $L_U=\bigvee_{i=1}^nL(l_i)/M_i$. Then the projection $U_d=\bigsqcup_{i=1}^nL(l_i;M_i)\onto L_U$ restricts to a bijection $U_d\isom L_U^\times$. Moreover, by Lemmas \ref{lem:morebasis} and \ref{lem:coprolfws} we have a homeomorphism $\widehat{L_U}\cong U$.      
\end{rem}

\subsection{Inverse semigroup actions}\label{subsec:invact}
\bdefin
Let $X$ be a topological space and let 
\[
\cI(X)=\{f:A\isom B\colon A,B\subset L \text{ open }\} 
\]
be the set of all homeomorphisms between open subsets of $X$. Remark that composition of functions makes $\cI(X)$ into an inverse semigroup, pointed by the empty function $0$. An \emph{action} of a pointed inverse semigroup $S$ on $X$ is a homomorphism of pointed semigroups 
$\rho:S\to\cI(X)$, such that for $\dom_X(s):=\dom(\rho(s))$ $(s\in S)$ we have $X=\bigcup_{s\in S}\dom_X(s)$. If $s\in S$ and $x\in\dom_X(s)$, we write
$s\cdot x=\rho(s)(x)$.
\edefin
\begin{ex}\label{ex:iposet}
Let $X$ be a poset; equip it with the downward-oriented Alexandroff topology, where the open sets are the downsets. Then $\cI(X)$ is the set of all poset isomorphisms between downsets of $X$ and an action of an inverse semigroup on $X$ in the sense above is what we call a \emph{poset action} of $S$ on $X$. If $X$ is just a set, then by a (set-theoretic) action of $S$ on $X$ we understand a poset action where $X$ is equipped with the discrete topology, or equivalently with the discrete preorder, so that any two distinct elements are incomparable.
\end{ex}
A poset action $S \acts L$ on a locally finite weak semilattice $L$ induces an action $S \acts \widehat{L}$ by homeomorphisms between open subsets, called the  \emph{dual action}. The latter is given by the partial homeomorphisms $\dom_{\widehat L}(s) \to \ran_{\widehat L}(s), \, \chi \ma s\cdot \chi$, where 
\[
 \dom_{\widehat L}(s) \coloneq \menge{\chi \in \widehat{L}}{\chi(l) = 1 \text{ for some } l \in \dom_L(s)},
\]
$\ran_{\widehat L}(s)$ is defined analogously, and $s\cdot \chi(l) = 1$ if and only if there exists $k \in \ran_L(s)$ with $k \leq l$ and $\chi(s^{-1}(k)) = 1$.

The \emph{transformation groupoid} or \emph{groupoid of germs} $S\ltimes X$ of an action  of an inverse semigroup $S$ by partial homeomorphisms on a locally compact, Hausdorff topological space $X$ is defined in \cite{exel}*{Section 4}. We have 
\begin{equation}\label{eq:gpdgrm}
 S \ltimes X \coloneq \{(s,x) \in S \times X\colon x \in \dom_{X}(s) \}/ { }_{\sim},    
\end{equation}
where $(s_1,x_1) \sim (s_2,x_2)$ if and only if $x_1 = x_2$ and there exists $s \leq s_1, s_2$ with $x_1 \in \supp_X(s)$. We write  $[s,x]$ for the class of $(s,x)$. 
The domain and codomain functions are given by $d[s,x]=x$, $c[s,x]=s\cdot x$ and composition by $[t,s.x][s,x] = [ts,x]$. The groupoid $S\ltimes X$ is equipped with the topology of \cite{exel}*{Proposition 4.14}, of which a basis is given by the open bisections 
\[
[s,U]=\{[s,x]\colon x\in U\}
\]
where $s\in S$ and $U\subset\dom_{X}(s)$ is open. The groupoid $S\ltimes X$ is always \'etale and is ample if and only if $X$ is totally disconnected, in which case the bisections $[s,U]$ as above with $U\in\CO(X)$ also form a basis. 
\begin{defi}\label{defi:indep}
An ample groupoid $G$ is called \emph{independent} if there exist an inverse semigroup $S$, 
a locally finite weak semilattice and a poset action $S\acts L$ such that $G\cong S\ltimes\whL$.
\end{defi}
\bremark\label{rem:actpat}
The most canonical action of an inverse semigroup $S$ on a locally finite weak semilattice is the action $S \acts E$ on the semilattice $E$ of idempotents in $S$. The action is given at $s \in S$ by $\dom_E(s) = \menge{e \in E}{e \leq s^{-1}s}$ and $s\cdot e \coloneq s e s^{-1}$. In this case, $S \ltimes \widehat{E}$ is (up to a small modification, see \cite{CELY}*{\S~5.5.4}) Paterson's universal groupoid of $S$ as in \cite{Pat}*{\S~4.3}.
\eremark

\begin{lem}\label{lem:doms}
Let $L$ be a locally finite weak semilattice and $S\acts L$ a poset inverse semigroup action. Let $s\in S$, $l\in L^\times$ and $M\subsetneq L(l)$ a finitely generated downset. Then the following are equivalent
\item[i)] $\dom_{\widehat{L}}(s)\supset \widehat{L}(l;M)$.
\item[ii)] $\dom_{L}(s)\supset L(l;M)$.
\item[iii)] $\dom_{L}(s)\supset L(l)$.
\item[iv)] $\dom_{\widehat{L}}(s)\supset \widehat{L}(l)$.
\end{lem}
\begin{proof}
If $x\in L(l;M)$ then $\cP_x\in\whL(l;M)$, from which it follows that i) implies ii). If ii) holds and $\chi\in \widehat{L}(l;M)$, then $l\in\supp(\chi)\cap \dom_L(s)$, so $\chi\in\dom_{\widehat{L}}(s)$. Thus i) holds $\iff$ ii) does. The particular case $M=0$ gives the equivalence of iii) and iv). Clearly iii)$\Rightarrow$ ii). Assume conversely that ii) holds. Then $l\in L(l;M)\subset\dom_L(s)$, which, since $\dom_L(s)$ is a downset, implies that $L(l)\subset\dom_L(s)$. 
\end{proof}

\smallskip

\begin{lem}\label{lem:tecombino}
 Let $X$ be a set and $S\acts X$ an inverse semigroup action. Also let $s\in S$, $Y, Y_1,\dots,Y_n\subset\dom_X(s)$ be subsets and $a_1,\dots,a_n\in\ell$. Then the following are equivalent.
 \item[i)] $\chi_{[s,Y]}=\sum_{i=1}^na_i\chi_{[s,Y_i]}$.
 \item[ii)] $\chi_Y=\sum_{i=1}^na_i\chi_{Y_i}$.
\end{lem}
\begin{proof}
  Put $G=S\ltimes X$. Let $\phi=\sum_{i=1}^na_i\chi_{[s,Y_i]}$, $\psi=\sum_{i=1}^na_i\chi_{Y_i}$ and $g\in G$. We have
  \[
  \phi(g)=\left\{\begin{matrix}\sum_{\{i\colon x\in Y_i\}}a_i=\psi(x)& \text{\, if } g=[s,x],\, x\in\bigcup_{i=1}^nY_i\\
                 0& \text{ otherwise}\end{matrix}\right.
  \]
  Hence $\chi_{[s,Y]}=\phi\iff \chi_Y=\psi$.
\end{proof}

\smallskip

\begin{lem}\label{lem:spanl} Let $S\acts L$ be an inverse semigroup action on a locally finite weak semilattice. Then 
\[
\ell[S\ltimes \whL]=\lspan_\ell\{\chi_{[s,\whL(l)]}\colon l\in L^\times\}.
\]    
\end{lem}
\begin{proof}
    Put $G=S\ltimes\whL$. By Lemma \ref{lem:basieljá}, every nonempty compact open subset of $\whL$ is a finite disjoint union of subsets of the form $\whL(l;l_1,\dots,l_n)$ for some $l\in L^\times $ and some $l_1,\dots,l_n\in L(l)$. It follows that $\ell[G]$ is spanned by the characteristic functions of all sets of the form $[s,\whL(l;l_1,\dots,l_n)]$ with 
    $$
    \whL(l;l_1,\dots,l_n)\subset\dom_{\whL}(s).
    $$ 
    By Lemma \ref{lem:doms} the latter condition is equivalent to 
    $$
    L(l)\subset\dom_{L}(s).
    $$
    Fix $s,l,l_1,\dots,l_n$ satisfying the above conditions. Let $F=\{l,l_1,\dots,l_n\}$ and let $\bar{F}$ be its $\downarrow$-closure. Remark that $\bar{F}\subset L(l)\subset\dom_L(s)$. Hence by Lemma \ref{lem:tecombino} it suffices to show that 
    \begin{equation}\label{eq:inspan}
    \chi_{\whL(l;l_1,\dots,l_n)}\in A:=\lspan_\ell\{\chi_{\whL(x)}\colon x\in \bar{F}\}.
    \end{equation} 
    Put $[n]=\{1,\dots,n\}$. We have
    \[
    \chi_{\whL(l;l_1,\dots,l_n)}=\prod_{i=1}^n(\chi_{\whL(l)}-\chi_{\whL(l_i)})=\sum_{I\subset [[n]]}(-1)^{|I|}\chi_{\bigcap_{i\in I}\whL(l_i)}
    \]
    Remark that if $n=1$, then this is just $\chi_{\whL(l)}-\chi_{\whL(l_1)}$ which is in the right hand side of \eqref{eq:inspan}. If $n>1$, proceed as in the proof of Lemma \ref{lem:basieljá} ii) to write each $\bigcap_{i\in I}\whL(l_i)$ as a disjoint union of sets of the form $\whL(x;x_1,\dots,x_m)$ for some $m$ and some $x,x_1,\dots,x_m\in \bar{F}\cap \bigcap_{i\in I}L(l_i)$. Iterating this, and using the fact that $\bar{F}$ is finite, we write 
    $\chi_{\whL(l;l_1,\dots,l_n)}$ as a linear combination of characteristic functions of sets of the form $\whL(x)$ and $\whL(x;y)$ with $x,y\in\bar{F}$ and $y<x$, and these functions are clearly in the $\ell$-module $A$ of \eqref{eq:inspan}.
\end{proof}

\section{Discretisation}\label{sec:disc}

As explained in the introduction, discretisation aims to compare invariants of independent groupoids with those of their associated discrete groupoids. Our motivation is that discrete groupoids reduce, up to Morita equivalence, to disjoint unions of groups, as shown by the following well-known and straightforward lemma. 

\begin{lem}\label{lem:gpd101}
 Let $\cG$ be a discrete groupoid and let
 $Z\subset \cG^{(0)}$ such that the projection induces a 
 bijection $Z\isom\cG^{(0)}/\cG$. For each $x\in \cG^{(0)}$ let $\cO_x=\cG^x\cdot x$ be its left orbit, $\{\xi_x\}=\cO_x\cap Z$ and choose $s_x\in\cG_{\xi_x}^x$ (if $x\in Z$, choose $s_z=z$). Then we have an isomorphism of groupoids 
 \[
 \cG\isom \coprod_{z\in Z}\cO_z\times\cO_z \times\cG_{z}^z,\, g\mapsto (c(g),d(g), s_{c(g)}gs_{d(g)}^{-1}).
 \]
 and an isomorphism of $\ell$-algebras
 \[
 \ell[\cG]\isom \bigoplus_{z\in Z}M_{\cO_z} \ell[\cG_{z}^z],\,\,\chi_g\mapsto \epsilon_{c(g),d(g)}s_{c(g)}gs_{d(g)}^{-1}.
 \]
\end{lem}

Let us now formulate discretisation. We offer two perspectives, one based on a canonical groupoid correspondence and another on a canonical homomorphism of algebras.

\subsection{Formulating discretisation using a canonical groupoid correspondence}
\label{ss:DiscFormul}

Given a locally finite weak semilattice $L$ equipped with an action of an inverse semigroup $S$, consider the restriction $S \acts L\reg$ of our action to $L\reg = L \setminus \gekl{0}$. Viewing $L\reg$ as a discrete topological space, we form the partial transformation groupoid  $S\ltimes L^\times$ as in \eqref{eq:gpdgrm}.

The idea of discretisation is to compare $S \ltimes L\reg$ with $S \ltimes \widehat{L}$ at the level of an invariant of interest. Invariants we have in mind are algebraic $K$-theory spectra of ample groupoids as in \cite{Li22} or (homotopy) algebraic $K$-theory for Steinberg algebras of $S \ltimes L\reg$ and $S \ltimes \widehat{L}$. To formulate discretisation more precisely, let us introduce a proper étale $S \ltimes L\reg$-$S \ltimes \widehat{L}$-correspondence $\Omega$. Recall (e.g. from \cite{ankom}) that, given \'etale groupoids $G$ and $H$, a $G$-$H$-bispace is a topological space $\Omega$ equipped with a left action of $G$ and a right action of $H$ that commute; that is, the left anchor $\rho \colon \Omega \to G^{(0)}$ is $H$-invariant, the right anchor $\sigma \colon \Omega \to H^{(0)}$ is $G$-invariant and $g . (\omega . h) = (g . \omega) . h$ for compatible $g \in G$, $\omega \in \Omega$ and $h \in H$. An \'etale $G$-$H$-correspondence $\Omega$ is a $G$-$H$-bispace $\Omega$ such that the right action $\Omega \curvearrowleft H$ is free, proper and \'etale.
The correspondence $\Omega$ is proper if the map $\Omega/H \to G^{(0)}$ induced by the left anchor is proper.

\bdefin
\label{def:DiscCorr}
Let $L$ be a locally finite weak semilattice equipped with the action of an inverse semigroup $S$. The \emph{discretisation correspondence} $\Omega=\Omega(S\acts L)\colon S \ltimes L\reg \to S \ltimes \widehat{L}$ is 
\begin{equation*}
\Omega = \coprod_{l \in L\reg} \menge{[s,\chi] \in S \ltimes \widehat{L}}{\chi \in \supp(s), \, s\cdot\chi \in \widehat{L}(l)}.    
\end{equation*}
We write $[l,s,\chi]$ for the element $[s,\chi]$, viewed as an element belonging to the part of $\Omega$ which is indexed by $l \in L\reg$. The left $S \ltimes L\reg$-action is given by $[s',l].[l,s,\chi] = [s'.l,s's,\chi]$, and the right $S \ltimes \widehat{L}$-action by $[l,s,\chi].[s',\chi'] = [l,ss',\chi']$.
\edefin
It is straightforward to see that $\Omega$ is {\'e}tale and proper.

Now discretisation can be formulated as the statement that, at the level of the invariant of interest, the correspondence $\Omega$ induces an equivalence between $S \ltimes L\reg$ and $S \ltimes \widehat{L}$. 

\begin{defi}
Let $\fC\subset\Grp$ be a subcategory of the category of ample groupoids and proper \'etale correspondences, $L$ a locally finite weak semilattice, $S$ an inverse semigroup, $S\acts L$ a poset action, $G=S\times \whL$ and $G_d=S\ltimes L^\times$. Assume that  $\Omega(S\acts L)\in\fC(G_d,G)$. We say that \emph{discretisation holds for $S\acts L$ with respect to $\cH$} if $\cH(\Omega(S\acts L))$ is an equivalence. 
\end{defi}

\begin{exas} As mentioned in  Remark~\ref{rem:Phi-ind-by-Omega}, the $K$-theory spectrum of the small permutative category $\fB G$ of an ample 
groupoid, introduced in \cite{Li22}, is defined on the discretisation correspondence (or rather on a modified correspondence which agrees with it up to Morita equivalence), so the 
discretisation statement makes sense for $G\mapsto \bbK \fB_G$ (see Section \S~\ref{s:Disc_KBG}). Other important invariants that 
factor through the functor from $\Grp$ to the category of small (nonunital) $\ell$-linear categories that 
sends $G\mapsto \ell[G]$ and a $(G,H)$ correspondence to the $\ell[G]-\ell[H]$ bimodule $\ell[Z]$ 
of locally constant, compactly supported functions \cite{ankom}. Composing the latter functor with 
the idempotent completion of the additivisation, we obtain a functor $G\mapsto 
\Idem(\ell[G]^{\oplus})$ landing in the category $\Idem\Add_\ell$ of idempotent complete $\ell$-
linear additive categories. If $\cH:\Idem{\Add}_\ell\to \Sp$ is a functor, discretisation makes 
sense for the composite $G\mapsto \cH(\Idem(\ell[G]^{\oplus}))$. 
This applies, for example, when $\cH$ is (homotopy) algebraic $K$-theory, Hochschild and cyclic homology or any of their variants.  
\end{exas}

\subsection{Formulating discretisation using a canonical algebra homomorphism}\label{subsec:detour}

As explained at the end of the previous subsection, oftentimes the functors from ample groupoids to spectra or other stable $\infty$-categories we are interested in factor through the functor $G\mapsto \ell[G]$ that sends an ample groupoid to its Steinberg algebra. It is thus useful to formulate discretisation in terms of such algebras, and this is precisely the subject of this subsection.

Let $S\acts L^\times $ and $S\acts \widehat{L}$ be as in the previous subsection, and set $G=G(S\ltimes \widehat{L})$, $G_d=G(S\ltimes L^\times)$, $A=\ell[G]$, $A_d=\ell[G_d]$. Remark that if $p\in S$ is idempotent and $l\in\dom_L(p)$ then 
$\widehat{L}(l)\subset\supp_{\whL} (p)$. It follows that we have a well-defined $\ell$-linear map
\begin{equation}\label{map:elteta}
\theta:A_d\to M_{L^\times} A,\,\, \theta(\chi_{[s,l]})=\epsilon_{s\cdot l,l}\chi_{[s,\widehat{L}(l)]}.
\end{equation}
Here $\epsilon_{s\cdot l,l}\chi_{[s,\widehat{L}(l)]}$ is the matrix whose $(s\cdot l,l)$-entry is the only non-zero entry, which equals $\chi_{[s,\widehat{L}(l)]}$. Furthermore, one checks that $\theta$ is a ring homomorphism. Hence any left $M_{L^\times}A$-module $M$ may be regarded as a left $A_d$-module using $\theta$; we write ${}_\theta M$ for this $A$-module. In particular this applies to $A^{(L^\times)}=\bigoplus_{l\in L^\times}\epsilon_lA$, and 
\begin{equation}\label{map:inclu}
\bar{\Omega}:=\bigoplus_{l\in L^\times}\chi_{[\widehat{L}(l)]}A\subset A^{(L^\times)}
\end{equation}
is an $(A_d,A)$ sub-bimodule. Moreover one checks that 
\begin{equation}\label{map:retra}
{}_\theta A^{(L^\times)}\to \bar{\Omega},\, \epsilon_l a\mapsto \epsilon_l\chi_{\widehat{L}(l)}
\end{equation}
is a homomorphism of $(A_d,A)$-bimodules, left inverse to \eqref{map:inclu}.

Next consider the $\ell$-submodule $\ell[\Omega]$ of the module of functions $\Omega\to \ell$ generated by the characteristic functions of the open compact Hausdorff subsets. The $(\cG_d,\cG)$-bispace structure of $\Omega$ induces an $(A_d,A)$-bimodule structure on $\ell[\Omega]$ (see \cite{ankom}*{Section 7}). 

\begin{lem}\label{lem:omegatheta} Let $A$, $A_d$, $\theta$, $\Omega$ and $\bar{\Omega}$ be as above.
\begin{itemize}
\item[i)] $\ell[\Omega]\cong \bar{\Omega}$ as $(A_d,A)$-bimodules.
\item[ii)] Let $N$ be a (unital) right $A_d$-module. Then tensoring with \eqref{map:inclu} induces a right $A$-bimodule isomorphism
$N\otimes_{A_d}\bar{\Omega}\isom N\otimes_{\theta}A^{(L^\times)}$.
\end{itemize}
\end{lem}
\begin{proof}
For $l\in L^\times$ let $\Omega_l=\rho^{-1}(\{l\})$. Then 
\[
\ell[\Omega_l]=\lspan_\ell\{\chi_{[s,U]}\colon U\subset\supp(s),\, s(U)\subset \widehat{L}(l)\}=\chi_{\widehat{L}(l)}A. 
\]
Hence $\ell[\Omega]=\bigoplus_{l\in L^\times}\chi_{\widehat{L}(l)}A=\bar{\Omega}$, and one checks that this identification is compatible with the respective bimodule structures, proving i). Because \eqref{map:inclu} has a left inverse (namely \eqref{map:retra}) which is a bimodule homomorphism, the map of ii) is injective. Next remark that any element of $N$ is a sum $x=\sum_lx_l$ with $x_l=x_l\chi_l$. Hence
$N\otimes_\theta A^{(L^\times)}$ is $\ell$-linearly spanned by elements of the form $\xi=x\otimes \epsilon_{m}a$ with $x=x\chi_l$, and
$\xi=x\otimes\theta(\chi_l)\epsilon_m a=\delta_{l,m}x\otimes\epsilon_l\chi_{\widehat{L}(l)}a$ is in the image of \eqref{map:inclu}. 
\end{proof}

Let $R$ be a ring with local units. A (right) $R$-module $M$ is called \emph{unital} if $M\cdot R=M$. In the following corollary and elsewhere, we write $\Mod R$ (resp. $R$-$\Mod$) for the category of unital right (resp. left) $R$-modules. 
\begin{coro}\label{coro:omegatheta}
There is a commutative diagram
\begin{equation}\label{diag:discrecompa}
\xymatrix{\Mod A_d\ar[rr]^{\otimes \ell[\Omega]}\ar[dr]_{\otimes_{\theta}M_{L^\times}A}&&\Mod A\\
&\Mod M_{L^{\times}}A\ar[ur]_{\otimes A^{(L^{\times})}}&}
\end{equation}
\end{coro}
\begin{lem}\label{lem:matrimori}
Let $X$ be a set, $x\in X$ and $R$ a ring with local units. Then the functor
$-\otimes_{M_XR}R^{(X)}:\Mod M_XR\to \Mod R$ is an equivalence, inverse to the scalar extension along the inclusion $\iota_x:R\to M_XR$, $a\mapsto \epsilon_{x,x}$ 
\end{lem}
\begin{proof}
One checks that the scalar extension along $\iota_x$ is the functor $\Mod R\to \Mod M_XR$ $N\mapsto N^{(X)}$, and that the latter is inverse to the functor of the lemma.
\end{proof}

\begin{defi}
Let $S$ be an inverse semigroup, $L$ a locally finite weak semilattice, $S\acts L$ a poset action, and $\theta:\ell[G_d]\to M_{L^\times}\ell[G]$ as in \eqref{map:elteta}. Let $\cH\colon\ahal\to\Sp$ be a functor. We say that $S\acts L$ satisfies \emph{discretisation} with respect to $\cH$ if $\cH(\theta)$ is an equivalence.     
\end{defi}
In the following lemma and elsewhere, we write 
$$\ahal\supset\ahal^{\lu}$$
for the category of $\ell$-algebras and algebra homomorphisms and the full subcategory of algebras with local units. 
\begin{lem}\label{rem:conjequi}
Let $\cH:\Idem\Add_\ell\to\Sp$ be a functor; put $\cH':\ahal^{\lu}\to \Sp$, $\cH'(R)=\cH'(\Idem(R^\oplus))$, $\cH":\Grp\to\Sp$, $G\mapsto \cH'(\ell[G])$.  Let $S\acts L$ be an inverse semigroup action on a locally finite weak semilattice. Then $S\acts L$ satisfies discretisation with respect to $\cH"$ if and only if it satisfies discretisation with respect to $\cH'$.
\end{lem}
\begin{proof}
 Recall that, by definition, $\cH"(\Omega)=\cH(-\otimes_{A_d}\ell[\Omega])$. It follows from Lemma \ref{lem:matrimori} that  $\cH'(\iota_l):H(A)\to H(M_{L^\times}(A))$ is an equivalence, so by \eqref{diag:discrecompa}, $\cH'(\theta)$ is an equivalence if and only if $\cH(-\otimes_{A_d}\ell[\Omega])$ is an equivalence. This concludes the proof.    
\end{proof}

\begin{ex}\label{ex:discfull}
Let $S$ be a pointed inverse semigroup and $E\subset S$ the lattice of idempotent elements. As mentioned in Remark \ref{rem:actpat}, $G=S\ltimes \widehat{E}$ is the full groupoid of $S$. Now we point out that $G_d=S\ltimes E^\times$ is the universal groupoid of the inverse semigroup $S_d$ with the same elements as $S$, and with product defined by
\[
s\cdot_d t=\left\{\begin{matrix}
st & \text{ if } s^*s=tt^*\\
0 & \text{ otherwise}
\end{matrix}\right.
\]
Therefore by \cite{steinappr}*{Theorem 6.3}, the maps $s\mapsto \chi_{[s,D_{ss^*}]}$ and $s\mapsto \chi_{[s,s^*s]}$ induce algebra isomorphisms
$\ell[S]\isom \ell[G]$ and $\ell[S_d]\isom \ell[G_d]$. Under these identifications, $\theta$ becomes the map
\[
\theta:\ell[S_d]\mapsto M_{L^\times}\ell[S],\, \theta(s)=\epsilon_{ss^*,s^*s}s.
\]
This is the map called $\rho_d$ \cite{acm}*{Section 7}; hence discretisation for inverse semigroups as formulated in loc. cit. is equivalent to discretisation of universal groupoids in our sense.
\end{ex}

\section{\texorpdfstring{Discretisation and Steinberg algebras}{Discretisation and Steinberg algebras}}\label{sec:discalg}

\subsection{Example 1: Proper groupoids}
\label{ss:ProperGPD}
\numberwithin{equation}{subsection}

Let $L$ be a locally finite weak semilattice and $x\in L$. In what follows we write
$$L^{\bullet}(x)=L(x)\setminus\{x\}.$$

\begin{lem}\label{lem:sunion}
Let $S\acts L$ be an inverse semigroup action on a locally finite weak semilattice, $s\in S$, $F\subset \dom_L(s)$ finite and $\bar{F}$ its $\downarrow$-closure. Then
\[
\chi_{[s,\bigcup_{x\in F}\whL(x)]}\in\lspan_\ell\{\chi_{[s,y]}\colon y\in\bar{F}\}.
\]
\end{lem}
\begin{proof}
Put $\sigma=[s,\bigcup_{x\in F}\whL(x)]$. Remark that $\sigma =\bigcup_{s\in F}[s,\whL(x)]$; thus its characteristic function is the supremum of $\{\chi_{[s,\whL(x)]}\colon x\in F\}$ in the idempotent order:
\begin{align*}
\chi_\sigma=&\bigvee_{x\in F}\chi_{[s,\whL(x)]}\\
=&\sum_{\emptyset\ne I\subset F}(-1)^{|I|+1}\chi_{\cap_{x\in I}\whL(x)}
\end{align*}
Put
\[
(\big\downarrow)_{x\in I} x=\max\{y\in L\colon y\le x\forall x\in I\}.
\]
We have
\[
\cap_{x\in I}\whL(x)=\bigcup_{y\in (\big\downarrow)_{x\in I} x}\whL(y).
\]
The above union is either equal to a single $\whL(x)$ with $x\in I$ or a union of $\whL(y)$ with $y\in \bar{F}$ and $y<x$ for all $x\in I$. Hence because $\tilde{F}$ is finite, if we keep writing every union that appears as a supremum, we eventually get to a linear combination of $\chi_{\whL(z)}$ with $z\in \bar{F}$.
\end{proof}
\begin{lem}\label{lem:goodF}
Let $S\acts L$ be an inverse semigroup action on a locally finite weak semilattice, $G=S\ltimes \whL$, and $G_d=S\ltimes L^\times$. Let $F\subset G_d$ be such that
\item[i)] $[s,x]\in F\Rightarrow [s^*,sx]\in F$.
\item[ii)] If $[s,x],[t,y]\in F$ then $\forall z\in t^*(x\downarrow t(y))$, $[st,z]\in F$.
Consider the subset 
\[
S^{\cop}(G)\supset F_c:=\{[s,\whL(x)]\colon [s,x]\in F\}.
\]
Let
\begin{gather*}A=\lspan_\ell\{\chi_\sigma\colon \sigma\in F_c\}\subset \ell[G],\\ 
A_d=\lspan_\ell \{\chi_x\colon x\in F\}\subset \ell[G_d]. \end{gather*}
Assume that $F$ is finite. Then $F\subset G_d$ is a subgroupoid and there is an isomorphism of algebras
$\psi:A\to A_d$ such that for all $f\in A$ and $[s,x]\in F$, we have
\begin{equation}\label{eq:premaphi}
\psi(f)([s,x])=f([s,\cP_x]).
\end{equation}
\end{lem}
\begin{proof}
Remark that conditions i) and ii) imply that $[s^*s,x]$, which we shall henceforth identify with $x$, is in $F$ for all 
$[s,x]\in F$, and it follows that $F\subset G_d$ is a subgroupoid with unit space $F^{(0)}=F\cap L$, and that $F^{(0)}$ is $\downarrow$-closed. We use this to show that the $\ell$-linear map $\psi$ is multiplicative. It follows from Lemma \ref{lem:sunion} and condition ii) that the $\chi_{[s,\whL(x)]}$ with $[s,x]\in F$ generate $A$ as an $\ell$-module. Hence it suffices to check multiplicativity on generators of the form $\chi_{[s,\whL(x)]}$ with $[s,x]\in F$. For $x\in F^{(0)}$ put $L_F(x)=L(x)\cap F^{(0)}$. Let $[s,x],[t,y]\in F$; using i) and ii) in the third line, we have
\begin{gather*}
\psi(\chi_{[s,\whL(x)]})\psi(\chi_{[t,\whL(y)]})=\chi_{[s,L_F(x)]}\chi_{[t,L_F(y)]}=\\
\sum_{z\in L_F(x),\, w\in L_F(y)}\chi_{[s,x]}\chi_{[t,w]}=\sum_{\{w\in L_F(y)\colon t(w)\in L_F(x)\}}\chi_{[st,w]}\\
=\chi_{[st,\bigcup_{u\in t^*(x\downarrow t(y))}L_F(u)]}\\
=\psi(\chi_{[st,\bigcup_{u\in t^*(x\downarrow t(y))}\whL(u)]})=\psi(\chi_{[s,\whL(x)]}\chi_{[t,\whL(y)]}).
\end{gather*}

Again using i) and ii), we obtain that if $[s,x]\in F$ and $y\in L_F^\bu(x):=L(x)^\bu\cap F$, then $[s,\whL(y)]=[s,\whL(x)]\whL(y)\in S(F)$, which by Lemma \ref{lem:sunion} implies that $\chi_{[s,\bigcup_{y\in L^{\bullet}_F(x)}\whL(y)]}\in A$. It follows that there is an 
$\ell$-linear map  $\phi:A_d\to A$ such that
\[
\phi(\chi_{[s,x]})=\chi_{[s,\whL(x)]}-\chi_{[s,\bigcup_{y\in L^{\bullet}_F(x)}\whL(y)]}=\chi_{[s,\whL(x;L^{\bullet}_F(x))]}. 
\]
Note that here we are using finiteness of $F$. 

Let us check that $\psi$ and $\phi$ are inverse maps. Let $[s,x],[t,z]\in F$; we have 
\begin{gather*}
\psi(\phi(\chi_{[s,x]}))=\psi(\chi_{[s,\whL(x;L_F^\bu(x))]})([t,z])=\chi_{[s,\whL(x)]}([t,\cP_z])-
\chi_{[s,\bigcup_{y\in L^{\bullet}_F(x)}\whL(y)]}([t,\cP_z])\\
=\chi_{[s,L_F(x)]}([t,z])-\chi_{[s,\bigcup_{y\in L^{\bullet}_F(x)}L(y)])}([t,z])=\chi_{[s,x]}([t,z]).    
\end{gather*}

Thus $\psi\circ\phi=\id_{A_d}$. To prove that $\phi\circ\psi=\id_A$, it suffices to show that $\phi(\psi(\chi_{[s,\whL(x)]}))=\chi_{[s,\whL(x)]}$ for all $[s,x]\in F$. Using ii) at the last step, we obtain
\begin{gather*}
\phi(\psi(\chi_{[s,\whL(x)]}))=\phi(\chi_{[s,L_F(x)]})=\\
=\sum_{y\in L_F(x)}\chi_{[s,\whL(y)]}-\chi_{[s,\bigcup_{z\in L_F^\bu(y)}\whL(z)]}\\
=\sum_{y\in L_F(x)}\chi_{[s,\whL(y;L_F^\bu(y))]}\\
=\chi_{[s,\sqcup_{y\in L_F(x)}\whL(y;L_F^\bu(y))]}=\chi_{[s,\whL(x)]}.
\end{gather*}
\end{proof}

 Let $S\acts L$, $G$ and $G_d$ be as in Lemma \ref{lem:goodF}. We say that the action of $S$ on $L$ is \emph{locally finite} if for every finite subset $\Lambda\subset G_d$ there exists a finite subset $\Lambda\subset F\subset G_d$ such that $F$ satisfies conditions  i) and ii) of the lemma. 

Let $A\in\ahal$, $X$ a set, $x\in X$ and $\iota_x:A\to M_XA$ the corner inclusion. Let $\fC$ be an $\infty$-category and $\cH:\ahal\to\fC$ a functor.
The functor $\cH$ is \emph{$M_X$-stable} if the natural transformation $\cH(\iota_x):\cH\to \cH M_X$ is an equivalence. If $\fC$ is additive, $\cH$
is called \emph{finitely additive} if for every finite collection $\{A_i\colon i\in F\}$, the map $\cH(\bigoplus_{i\in F}A_i)\to \bigoplus_{i\in F}\cH(A_i)$ induced by the projections $\pi_j:\bigoplus_{i\in F}A_i\to A_j$ is an equivalence. 

\begin{thm}\label{thm:discprop}
 Let $L$ be a locally finite weak semilattice and $S\acts L$ a locally finite inverse semigroup action and $G=S\ltimes \whL$. 
 Let $\cH:\ahal\to\Sp$ be an $M_{L^\times}$-stable, finitely additive functor commuting with filtering colimits. 
 If $G$ is proper, 
 then $S\acts L$ satisfies discretisation with respect to $\cH$.
\end{thm}
\begin{proof}
Let $A=\ell[G]$, $A_d=\ell[G_d]$. Let $\cF$ be the collection of all finite subsets of $G_d$ satisfying conditions i) and ii) of Lemma \ref{lem:goodF}. For $F\in\cF$, let
\begin{gather*}
A_d(F)=\lspan_\ell\{\chi_{[s,x]}:[s,x]\in F\},\\
A(F)=\lspan_\ell\{\chi_{[s,\whL(x)]}\colon [s,x]\in F\}.    
\end{gather*}
Put $F^{(0)}=d(F)$. By Lemma \ref{lem:spanl} and because we are assuming that $S\acts L$ is locally finite, we have $A=\bigcup_{F\in \cF}A(F)$ and $A_d=\bigcup_{F\in\cF}A_d(F)$, where both unions are filtering. Moreover, since we are assuming that $\cH$ commutes with filtering colimits and is $M_{L^\times}$-stable, and since we have 
$$
\theta(A_d(F))\subset M_{F^{0}}A(F)\subset M_{L^\times}A(F)
$$ 
for all $F\in\cF$, it suffices to show that $\cH(\theta):\cH(A_d(F))\to \cH(M_{F^{(0)}}A(F))$ is an equivalence.
Let $\phi:A_d(F)\leftrightarrow A(F):\psi$ be the inverse algebra isomorphisms of Lemma \ref{lem:goodF}. With notations as in the proof of that lemma, let 
\[
\iota:A_d(F)\to M_{F^{(0)}}A(F),\, \iota(\chi_{[s,x]})=\epsilon_{s(x),x}\chi_{[s,L_F(x)]}.
\]
We have a commutative diagram
\[
\xymatrix{A_d(F)\ar[r]^\theta\ar[dr]_{\iota}&M_{L^\times}A(F)\ar[d]_{M_{F^{(0)}}\psi}\\
& M_{L^\times}A_d(F)}
\]
Let $\iota_0,\iota_1:A_d(F)\to M_FA_d(F)$,
\begin{gather*}
\iota_0(\chi_{[s,x]})=\epsilon_{s(x),x}\chi_{[s,x]},\,\,
\iota_1(\chi_{[s,x]})=\epsilon_{s(x),x}\chi_{[s,L_F^\bu(x)]}.
\end{gather*}
One checks that $\iota_0$ and $\iota_1$ are orthogonal algebra homomorphisms with $\iota=\iota_0+\iota_1$.
Hence, because $\cH$ is finitely additive and $M_{L^\times}$-stable, we have 
$\cH(\iota)\sim\cH(\iota_0)+\cH(\iota_1)$. One checks that for $N=|F|$, the composite
\begin{equation}\label{eq:nilpo}
M_{F^{N-1}}(\iota_1)\circ\dots\circ M_F(\iota_1)\circ\iota_1:A_d(F)\to M_{F^N}A_d(F)
\end{equation}
is identically zero. Recall that $M_{F^{(0)}}$-stability of $\cH$ implies that for $x\in F^{(0)}$, the natural map $h=\cH(\epsilon_{x,x})$ is an equivalence and is independent of $x$ up to equivalence. It follows from \eqref{eq:nilpo} that the composite $\cH(\iota_1)\circ h$ is nilpotent. Hence to show that $\cH(\iota)$ is an equivalence, it suffices to show that $\cH(\iota_0)$ is one. Let $\cR\subset F^{(0)}$ be such that the projection $F^{(0)}\to F^{(0)}/F$ restricts to a bijection on $\cR$. For $x\in\cR$ and $[s,x]\in F_x^x$, we have
\[
\iota_0(\chi_{[s,x]})=\epsilon_{x,x}\chi_{[s,x]}.
\]
Hence we have a commutative diagram
\[
\xymatrix{A_d(F)\ar[r]^{\iota_0}&M_FA_d(F)\\
\bigoplus_{x\in\cR}\ell[F_x^x]\ar[u]_{\inc}\ar[r]^{\oplus_x\epsilon_{x,x}}&\bigoplus_{x\in\cR}M_F\ell[F_x^x]\ar[u]^{\inc}}
\]
It follows from Lemma \ref{lem:gpd101} and matricial stability that applying $\cH$ to either any of the vertical maps or the bottom horizontal map is an equivalence. Hence $\cH(\iota_0)$ is an equivalence. This finishes the proof. 
\end{proof}

\begin{coro}\label{coro:discprop} Let $S\acts L$ be an inverse semigroup action on a semilattice and as in Theorem \ref{thm:discprop} (in particular, we assume that $G = S \ltimes \whL$ is proper). Let $\cH$ be as in Theorem \ref{thm:discprop} as well. Then $S\acts L$ satisfies discretisation with respect to $\cH$. 
\end{coro}
\bproof
Since $L$ is a semilattice, the set $S^{\cop}_p(G):=\{[s,\whL(x)]\colon x\in\dom_L(s)\}\cup\{\emptyset\}\subset  S^{\cop}(G)$ is an inverse sub-semigroup. Hence, since $G$ is proper, $S^{\cop}_p(G)$ is locally finite by Proposition \ref{prop:semilocafin}. Moreover, note that a subset $F\subset G_d$ satisfies the conditions of Lemma~\ref{lem:goodF} if and only if the subset $F_c\cup\{\emptyset\}\subset S^{\cop}(G)$ is an inverse subsemigroup. Thus $S\acts L$ is locally finite. Now the corollary follows from Theorem~\ref{thm:discprop}.
\eproof

\begin{exas}\label{exas:discprop} The hypotheses on $\cH$ of Theorem \ref{thm:discprop} are satisfied by algebraic $K$-theory $K$, homotopy algebraic $K$-theory $KH$, and Hochschild and cyclic homology of $\ell$-algebras, $HH(-/\ell)$ and $HC(-/\ell)$. Neither negative cyclic $HN(-/\ell)$ nor periodic cyclic homology $HP(-/\ell)$ commute with filtering colimits, so the theorem does not apply to them.
\end{exas}

\subsection{Example 2: Self-similar higher rank graph algebras}
\label{ss:sshrg}

\subsubsection{Higher rank graphs}
Let $k\ge 1$; equip the monoid $\Nz_0^k$ with the product order and regard it as a category.  We write 
$n\vee m$ for the supremum of two elements $n,m\in\Nz_0^k$. A \emph{$k$-graph} is a small category 
$\Lambda$ equipped with a functor $|\,|:\Lambda\to \Nz_0^k$ such that if $\lambda\in\Lambda$ is such that 
$|\lambda|=n_1+n_2\in \Nz_0^k$, then there exist unique composable $\lambda_1,\lambda_2\in\Lambda$ with 
$|\lambda_i|=n_i$ such that $\lambda=\lambda_1\lambda_2$. For $n\in\Nz_0^k$, put 
$$
\Lambda^n=\{\lambda\in\Lambda\colon |\lambda|=n\}.
$$ 
The elements of $\Lambda^0$ are called \emph{vertices}; thus a vertex is the same as the identity of an object of $\Lambda$. If $\lambda\in\Lambda$ we write $r(\lambda)$ for 
its domain and $s(\lambda)$ for its codomain. If $v\in\Lambda^0$ and $n\in\Nz_0^k$, we write $v\Lambda=s^{-1}\{v\}$ and $v\Lambda^n$ for the subset of those $\lambda$ with $s(\lambda)=v$ and $|\lambda|=n$. Let $\cB=\{\epsilon_i\colon 1\le i\le k\}\subset \Nz_0^k$ be the canonical basis; we have $|\lambda|=\sum_{i=1}^k|\lambda|_i\epsilon_i$. An element $\lambda\in\Lambda$ is an \emph{$i$-edge} if $|\lambda|=\epsilon_i$. 
We say that $\Lambda$ is \emph{row-finite} if $v\Lambda^n$ is finite for every vertex $v$ and $n\in\Nz_0^k$.
Let $\lambda,\mu\in\Lambda$. Put
\[
\Lambda\times\Lambda\supset \Lambda(\lambda,\mu)=\{(\alpha,\beta)\colon s(\alpha)=r(\lambda),\, s(\beta)=r(\mu),\,\lambda\alpha=\mu\beta\}
\]
Remark that $s(\lambda)\ne s(\mu)\Rightarrow \Lambda(\lambda,\mu)=\emptyset$. We say that a $k$-graph $\Lambda$ is \emph{finitely aligned} 
if for every $\lambda,\mu\in\Lambda$ the set 
\[
\Lambda(\lambda,\mu)\supset\Lambda^{\min}(\lambda,\mu)=\{(\alpha,\beta)\colon |\lambda\alpha|=|\lambda|\vee|\mu|\}
\]
is finite. Every row-finite $k$-graph is finitely aligned, but the converse does not hold. We shall also consider the set of \emph{minimal common extensions} 
\[
MCE(\lambda,\mu)=\{\lambda\alpha\colon (\exists \beta)\, (\alpha,\beta)\in \Lambda^{\min}(\lambda,\mu)\}.
\]

\subsubsection{Self-similar higher rank graphs}
Let $k\ge 0$, $\Lambda$ a finitely aligned $k$-graph, and $\crG$ a discrete groupoid with unit space $\crG^{(0)}=\Lambda^0$, and domain and codomain maps $d,c \colon \crG \to \crG^{(0)}$. 
A \emph{self-similar} action of $\crG$ on $\Lambda$ consists of
\begin{itemize}
    \item  A map 
    $$
     \crG{}_d\times_s\Lambda\to \Lambda,\,\, (g,\lambda)\mapsto g(\lambda),
    $$
    satisfying $|g\cdot \lambda|=|\lambda|$ \
    $\forall (g,\lambda)\in  \crG{}_d\times_s\Lambda$, that makes $\Lambda$, equipped with 
 the discrete topology and the source map $s:\Lambda\to \Lambda^0= \crG^{(0)}$ as 
    anchor, into a left $ \crG$-correspondence.

    \medskip
    
    \item A map $\phi: \crG{}_d\times_s\Lambda\to  \crG$ such that

    \smallskip
    
    \begin{enumerate}
        
        \item $\phi(g,v)=g$ $\forall v\in \Lambda^0$ and $g\in  \crG^v=d^{-1}(\{v\})$.

        \smallskip
        
        \item $c(\phi(g,\lambda))=r(g(\lambda))$ and $d(\phi(g,\lambda))=r(\lambda)$ $\forall (g,\lambda)\in  \crG{}_d\times_s\Lambda$.

        \smallskip
        
        \item $g(\lambda\mu)=g(\lambda)\phi(g,\lambda)(\mu)$ $\forall  (g,\lambda,\mu)\in  \crG{}_d\times_s\Lambda_r\times_s\Lambda$.

        \smallskip
        
        \item $\phi(gh,\lambda)=\phi(g,h(\lambda))\phi(h,\lambda)$ $\forall (g,h,\lambda)\in  \crG{}_d\times_c \crG{}_d\times_s\Lambda$.
        
        \smallskip
        
        \item $\phi(g,\lambda\mu)=\phi(\phi(g,\lambda),\mu)$ $\forall (g,\lambda,\mu)\in  \crG{}_d\times_s\Lambda{}_r\times_s\Lambda$.
    \end{enumerate}
    \end{itemize}
\subsubsection{Cohn groupoids}    
Let $k\ge 1$ and let $( \crG,\Lambda,\phi)$ be a self-similar $k$-graph. We say that two elements $(\lambda,g,\mu)$ and $(\xi,h,\eta)$ of $\Lambda{}_r\times_c \crG{}_d\times{}_r\Lambda$ are \emph{orthogonal}, and write $(\lambda,g,\mu)\perp (\xi,h,\eta)$, if $MCE(\lambda,\xi)=MCE(\eta,\mu)=\emptyset$. Call a subset $F\subset \Lambda{}_r\times_c \crG{}_d\times{}_r\Lambda$ \emph{orthogonal} if any two distinct elements of $F$ are orthogonal. 
 Let 
 \[
 S( \crG,\Lambda)=\{F\subset \Lambda{}_r\times_c \crG{}_d\times{}_r\Lambda\colon F \text{ finite and orthogonal}\}.
 \]
 We shall equip $S(\crG,\Lambda)$ with a multiplication. We start with the case of singletons; for $\{(\lambda,g,\mu)\},\{(\xi,h,\eta)\}\in S( \crG,\Lambda)$, we set 
 \begin{equation}\label{eq:multisg}
 \{(\lambda,g,\mu)\}\cdot\{(\xi,h,\eta)\}=\{(\lambda g(\alpha),\phi(g,\alpha)\phi(h,h^{-1}(\beta)),\eta h^{-1}(\beta))\colon (\alpha,\beta)\in \Lambda^{\min}(\mu,\xi)\}
 \end{equation}
 For arbitrary $F,F'\in S( \crG,\Lambda)$, put
 \[
 F\cdot F'=\bigcup_{\overset{(\lambda,g,\mu)\in F}{(\xi,h,\eta)\in F'}} \{(\lambda,g,\mu)\}\cdot\{(\xi,h,\eta)\}
 \]

 \begin{lem}\label{lem:larki}
 \item[i)] The multiplication defined above makes $S( \crG,\Lambda)$ into an inverse semigroup, pointed by the empty set. 
 \item[ii)] If $F\in S( \crG,\Lambda)$ its inverse is 
 \[
 F^*=\{(\mu,g^{-1},\lambda)\colon (\lambda,g,\mu)\in F\}.
 \]
 \item[iii)] $F^2=F$ if and only if $$(\lambda,g,\mu)\in F\Rightarrow
 \lambda=\mu,\, g=r(\lambda).$$
 \end{lem}
 \begin{proof}
 For the case when $ \crG$ is a discrete group, this was proven by Hossein Larki in \cite{larki}*{Proposition 3.6 and Corollary 3.5}. The same arguments apply verbatim for arbitrary discrete groupoids. 
 \end{proof}
 
 In what follows we fix $k\ge 1$ and a self-similar, finitely aligned $k$-graph $( \crG,\Lambda,\phi)$ over a discrete groupoid $ \crG$.
 
 Recall that any poset $P$ may be regarded as a category where 
 $\hom_P(p,q)$ is empty if $p\not\le q$ and has the single element $(p,q)$ 
 if $p\le q$. Let $k\ge 1$. Write $\Omega_{k,\infty}$ for the $k$-graph 
 consisting of the category associated to the poset $\Nz_0^k$, 
 equipped with the grading $|(p,q)|=q-p$. For $m\in\Nz_0^k$, let 
 $\Omega_{k,m}\subset\Omega_{k,\infty}$ be the full $k$-subgraph with vertices $\{n\in\Nz_0^k\colon n\le m\}$.  
 By \emph{path of length} $m$ in $\Lambda$ we understand a degree-preserving functor $x:\Omega_{k,m}\to \Lambda$. 
 Let $X_{\Lambda}$ be the space of (both finite and infinite) paths in $\Lambda$ as defined in \cite{farthing}*{Definition 5.1}. As per \cite{farthing}*{Remark 5.3}, we may identify $\Lambda$ with the subset of $X_{\Lambda}$ consisting of all finite paths. Remark that an infinite path is essentially a countable transfinite composition of elements of $\Lambda$ together with all its factorisations. 
 If $\lambda\in\Lambda$ and $y\in X_{\Lambda}$ are such that $r(\lambda)=s(y)$ then $\lambda y$ is a well-defined element of $X_{\Lambda}$. If $\lambda\in\Lambda$ and $x\in X_{\Lambda}$ we write $\lambda\ge x$ if there exists $y\in X_{\Lambda}$ such that $x=\lambda y$. 
 
For $F\in S( \crG,\Lambda)$, let
\[
X_{\Lambda}\supset D_F=\{x\colon (\exists (\lambda,g,\mu)\in F)\, \mu\ge x\}
\]
Remark that $D_F=D_{F^*F}$. If $F=\{(\lambda,g,\mu)\}$, we abbreviate 
$D_{(\lambda,g,\mu)}:=D_F$ and if $v\in \Lambda^0$, we put $D_v=D_{(v,v,v)}$. 
For $x\in D_v$ and $g\in  \crG^v$, set
\begin{gather*}
g(x):\Omega_{k,|x|}\to \Lambda, \\
g(x)(p,q)=\phi(g,x(0,p))(x(p,q))
\end{gather*}
One checks that $g(x)\in D_{c(g)}$. Thus we have a map
\begin{gather*}
  \crG{}_d\times_sX_{\Lambda}\to X_{\Lambda}\\
 (g,x)\mapsto g(x). 
\end{gather*}
Let $F\in S(\crG,\Lambda,\phi)$ and $x\in D_F$. Remark that there are unique elements $(\lambda,g,\mu)\in F$ and $y\in D_{r(\mu)}$ such that $x=\mu y$. Set
\[
F\cdot x=\lambda g(y).
\]
This defines a map
\[
F\cdot: D_F\to D_{F^*}.
\]
 Let $\cI(\Lambda)$ and $\cI(\Lambda_*)$ be as in Example \ref{ex:iposet}. If $U,V\subset \Lambda$ are downsets, and $f:U\to V$ is a poset isomorphism, we write $U_*=U\sqcup\{0\}$, $V_*=V\sqcup\{0\}$, and $f_*:U_*\to V_*$ for the pointed extension of $f$. We have a semigroup homomorphism
\[
\cI(\Lambda)\to \cI(\Lambda_*), \, f\mapsto f_*
\]

 \begin{lem}\label{lem:kgraphlfws}
 \item[i)] $\Lambda_*$ is a locally finite weak semilattice.
 \item[ii)] The map $\phi:X_{\Lambda}\to \widehat{\Lambda_*}$, 
$\phi(x)=\chi_{(x\uparrow)}$ is a homeomorphism. 
 \end{lem}
 \begin{proof}
i) is straightforward. For ii), we first prove that $\phi$ is bijective. Injectivity follows from the fact that a path $x$ is 
completely determined by its factorisations, which are in turn 
determined by $\chi_{(x\uparrow)}=\phi(x)$. To prove surjectivity, let 
$\chi_{\cF}\in \widehat{\Lambda_*}$. Then $\Lambda\supset\cF\ne\emptyset$ and $\cF\owns x,y$ $\iff$ $(x\downarrow y)\subset \cF$. Thus if $\cF$ is finite, it must have a minimum element $\alpha$ and one checks that $\cF=\phi(\alpha)$. Next assume that $\cF$ is infinite. If $\lambda,\mu\in\cF$ then 
$\emptyset\ne MCE(\lambda,\mu)$ and if $|\lambda|=|\mu|$, we must have $\lambda=\mu$ by uniqueness of the factorisation in $\Lambda$. Since $\cF$ is an infinite upset, this implies that for any $n\in\Nz_0^k$ there is a unique $x(0,n)\in\cF$ with $|x(0,n)|=n$, and if $n\le m$, a unique  $x(n,m)\in\Lambda$ such that $x(0,m)=x(0,n)x(n,m)$. This defines a functor $x:\Omega_{k,\infty}\to\Lambda$ and one checks that $\phi(x)=\chi_{\cF}$. Thus $\phi$ is bijective. One checks that if $r(\lambda)=s(\nu_i)$ for $i=1,\dots, n$ are elements of $\Lambda$, then 
\[
\phi(D_\lambda\cap \bigcap_{i=1}^nD_{\lambda\nu_i}^c)=
\widehat{\Lambda_*}(\lambda)\cap \bigcap_{i=1}^n\widehat{\Lambda_*}(\lambda\nu_i)^c.
\]
Thus $\phi$ is a homeomorphism. 
 \end{proof}

\begin{lem}\label{lem:sgacts}
 \item[i)] The map $F\mapsto (F\cdot D_F\to D_{F^*})$ is a semigroup homomorphism $S( \crG,\Lambda,\phi)\to\cI(X_{\Lambda})$. This defines a partial action $S( \crG,\Lambda)\acts X_{\Lambda}$.
 \item[ii)] For all $F\in S( \crG,\Lambda,\phi)$, $D_F\cap \Lambda$ is a downset, $F\cdot (D_F\cap\Lambda)=(D_{F^*}\cap\Lambda)$, and $F\cdot \colon D_F\cap\Lambda\to D_{F^*}\cap\Lambda$ is a poset isomorphism. Hence we have a semigroup homomorphism $S( \crG,\Lambda,\phi)\to \cI(\Lambda_*)$, which defines an action $S( \crG,\Lambda,\phi)\acts \Lambda_*$.
 \item[iii)] Under the homeomorphism $\phi$ of Lemma \ref{lem:kgraphlfws}, the action of part i) corresponds to the action on $\widehat{\Lambda_*}$ that is dual to that of part ii).
\end{lem}
\begin{proof}
\begin{itemize}
    \item[i)] We have to prove that if $F_1,F_2\in S( \crG,\Lambda,\phi)$
    and $x\in F_2^*(D_{F_1})$, then $F_1\cdot (F_2\cdot x)=(F_1F_2)\cdot x$.  
    By hypothesis there exist (necessarily unique) $(\lambda_2,g_2,
    \mu_2)\in F_2$ and $y\in D_{r(\mu_2)}$ such that $x=\mu_2y$ and 
    $(\lambda_1,g_1,\mu_1)\in F_1$ and $x\in D_{r(\mu_1)}$ such that 
    $\mu_1\ge\lambda_2g_2(y)$. Then 
    $F_2\cdot x=\lambda_2g_2(y)$ and  for $(\alpha,\beta)\in 
    \Lambda^{\min}(\mu_1,\lambda_2)$, we have
    \begin{gather*}
    \lambda_2g_2(y)=\lambda_2\beta g_2(y)(|\mu_1|\vee|\lambda_2|,\infty)\\
    =\lambda_2g_2(g_2^{-1}(\beta)y(|\mu_1|\vee|\lambda_2|,\infty)),\infty)=\lambda_2\beta\phi(g_2,g_2^{-1}(\beta))(y(|\mu_1|\vee|\lambda_2|,\infty)\\
    =\mu_1\alpha \phi(g_2,g_2^{-1}(\beta))(y(|\mu_1|\vee|\lambda_2|,\infty)).
    \end{gather*}
    Thus 
    \begin{gather*}
    F_1\cdot(F_2\cdot x)=\{\lambda_1g_1(\alpha)\phi(g_1,\alpha)\phi(g_2,g_2^{-1}(\beta))(y(|\mu_1|\vee|\lambda_2|,\infty)\colon (\alpha,\beta)\in\Lambda^{\min}(\mu_1,\lambda_2)\}\\
    =\{(\lambda_1g_1(\alpha),\phi(g_1,\alpha)\phi(g_2,g_2^{-1}(\beta),\mu_2g_2^{-1}(\beta))\}\cdot \mu_2y\\
    =(F_1F_2)\cdot x.
    \end{gather*}

\item[ii)] Let $F\in S(\crG,\Lambda,\phi)$ and let $\pi_3:\Lambda{}_r\times_c \crG{}_d\times_r\Lambda\to \Lambda$ be the projection. 
It follows from the definitions and the fact that $F$ is orthogonal that 
$$
D_F\cap \Lambda=\sqcup_{\mu\in \pi_3(F)}(\mu\downarrow).
$$
Hence $D_F\cap\Lambda$ is a downset. If $\lambda\in \pi_3(F)$, then there is a unique element in $F$ of the form $(\lambda,g,\mu)$ and $F$ acts on $(\mu\downarrow)$ via $\mu x\mapsto \lambda g(x)$, which is a poset isomorphism onto $D_{F^*}\cap \Lambda$.

\item[iii)] Let $F\in S(\crG,\Lambda,\phi)$. Then $\dom(F)=\phi(D_F\cap \Lambda)$ and thus
\begin{gather*}
\supp(F)=\{\phi(x)\colon \exists \mu\in \pi_3(F),\, \mu\ge x\}\\
=\phi(D_F).
\end{gather*}
Besides, if $x\in D_F$ then there are unique $(\lambda,g,\mu)\in F$ and $y\in D_{r(\mu)}$ such that $x=\mu y$ and $F\cdot \phi(x)=F\cdot\chi_{(\mu y\uparrow)}$ is the characteristic function of 
\begin{gather*}
 \Lambda\supset\{\xi\colon \exists \eta\in \bigsqcup_{\lambda\in \pi_1(F)}(\lambda\downarrow)\cap(\xi\downarrow)\quad (F^*\cdot \eta)\ge x\}\\
 =\{\xi\exists \lambda\in \pi_1(F),\, \zeta\le r(\lambda)\eta-\lambda\zeta,\, \mu\zeta\ge x\}\\
 =(\lambda\zeta\uparrow)=(F\cdot x\uparrow).
 \end{gather*}
 This shows that $F\cdot \phi(x)=\phi(F\cdot x)$.\qedhere
\end{itemize}
\end{proof}

The \emph{Cohn groupoid} of $( \crG,\Lambda,\phi)$ is the transformation groupoid $G_c=G_c( \crG,\Lambda,\phi)=S(\crG,\Lambda,\phi)\ltimes X_{\Lambda}$ as defined in \eqref{eq:gpdgrm}. Thus $G_c( \crG,\Lambda,\phi)$ consists of equivalence classes $[F,x]$ of pairs $(F,x)\in S( \crG,\Lambda,\phi)\times X_{\Lambda}$ with $x\in D_F$, where $[F_1,x]=[F_2,y]$ if and only if $x=y$ and there exists an idempotent $E\in S(\crG,\Lambda,\phi)$ such that $x\in D_E$ and 
$F_1E=F_2E$. In what follows, if $F$ has a single element $(\lambda,g,\mu)$, we shall often abbreviate $[F,x]=[(\lambda,g,\mu),x]$.

\begin{lem}\label{lem:tgpd}
\item[i)] For every $t\in G_c( \crG,\Lambda,\phi)$ there exist $x\in X_{\Lambda}$ and $\{(\lambda,g,\mu)\}\in S( \crG,\Lambda,\phi)$ with $\mu\ge x$ such that $t=[(\lambda,g,\mu),x]$.
\item[ii)] $\nabla=[(\lambda,g,\mu),x]=[(\xi,h,\eta),y]$ if and only if $x=y$ and there exist $\alpha\in r(\mu)X_{\Lambda}$ and $\beta\in r(\eta)X_{\Lambda}$ such that the following conditions hold simultaneously
\begin{enumerate}
      \item $\mu\alpha=\eta\beta\ge x$;
    \item $\lambda g(\alpha)=\xi h(\beta)$;
    \item $\phi(g,\alpha)=\phi(h,\beta)$;
\end{enumerate}
and in this case we have $\nabla=[(\lambda g(\alpha),\phi(g,\alpha),\mu\alpha),x]$.

\item[iii)] The assignment $[(\lambda,g,\mu),x]\mapsto |\lambda|-|\mu|$ gives a well-defined, continuous groupoid homomorphism from $G_c( \crG,\Lambda,\phi)$ to the discrete groupoid $\Z^k$.
\end{lem}
\begin{proof}
Put $T=(\crG,\Lambda,\phi)$, $S=S(T)$, $G_c=G_c(T)$. If $[F,x]\in G_c$ then there is a unique $(\lambda,g,\mu)\in F$ with $\mu\ge x$, and $[F,x]=[F\{\mu,r(\mu),\mu\},x]=[(\lambda,g,\mu),x]$. This proves i). Next let $s=(\lambda,g,\mu)$, $t=(\xi,h,\eta)$ and assume that $[s,x]=[t,y]$. Then by definition of the equivalence relation of germs we must have $x=y$, and $\mu\ge x\le \eta$ implies that there exists $(\gamma_1,\gamma_2)\in\Lambda^{\min}(\mu,\eta)\ne\emptyset$.
Then $\gamma=\mu\gamma_1=\eta\gamma_2\ge x$ and for  $u=(\gamma,r(\gamma),\gamma)$, we have $\{s'\}:=\{s\}\{u\}=\{(\lambda g(\gamma_1),\phi(g,\gamma_1),\gamma)\}$, $\{t'\}:=\{t\}\{u\}=\{(\xi h(\gamma_2),\phi(h,\gamma_2),\gamma)\}$ and 
\begin{gather*}
[s',x][(\lambda g(\gamma_1),\phi(g,\gamma_1),\gamma),x]=[s,x]\\
=[t,x]=[(\xi h(\gamma_2),\phi(h,\gamma_2),\gamma),x]=[t',x].
\end{gather*} 
Again by the definition of germ, there exists an idempotent $E\in S( \crG,\Lambda,\phi)$ such that 
\begin{equation}\label{condi:germs}
 x\in D_E \text{ and } \{s'\}E=\{t'\}E.
\end{equation}  
By Lemma \ref{lem:larki} (iii), $E$ is a finite orthogonal set of elements of the form $e(\zeta)=(\zeta,r(\zeta),\zeta)$. The first condition of \eqref{condi:germs} implies that there is a unique such element with $\zeta\ge x$. Hence $x\in D_{e(\zeta)}$, $E e(\zeta)=e(\zeta)$ and right-multiplying both sides of the identity in \eqref{condi:germs} by $\{e(z)\}$, and using properties  (3) and (5) of $\phi$, we obtain that there are $(\omega_1,\rho_1),(\omega_2,\rho_2)\in \Lambda^{\min}(\gamma,\zeta)$ such that 
\begin{gather*}
 (\lambda g(\gamma_1\omega_1),\phi(g,\gamma_1\omega_1),\gamma\rho_1)=(\xi h(\gamma_2\omega_2), \phi(h,\gamma_2\omega_2),\gamma\rho_2).
\end{gather*}
Hence $\rho:=\rho_1=\rho_2$ and for $\alpha=\gamma_1\omega_1$ and $\beta=\gamma_2\omega_2$ we have $\lambda g(\alpha)=\xi h(\beta)$,\, $\phi(g,\alpha)=\phi(h,\beta)$
and 
\begin{gather*}
\mu\alpha=\mu\gamma_1\omega_1=\gamma\omega_1=\zeta\rho=\gamma\omega_2=\eta\gamma_2\omega_2=\eta\beta\ge x.   
\end{gather*}
This concludes the proof of ii). The fact that the assignment $[(\lambda,g,\mu),x]\mapsto |\lambda|-|\mu|$ of part iii) is well-defined is straightforward from ii). Continuity follows from the fact that if $[(\lambda,g,\mu),x]\mapsto n\in\Z^k$, then every element of $D_{(\lambda,g,\mu)}$ is also mapped to $n$. This concludes the proof.
\end{proof}

In what follows, we use the following abbreviated notation for elements of $S(\crG,\Lambda,\phi)$; we write $\lambda=(\lambda,r(\lambda),r(\lambda))$, $\lambda^*=(r(\lambda),r(\lambda),\lambda)$, $g=(c(g),g,d(g))$ $(
\lambda\in\Lambda,\, g\in \crG)$. Thus with our conventions in place, an element $(\lambda,g,\mu)\in\Lambda{}_r\times_c \crG{}_d\times_r\Lambda$, can be written as the following product in $S(\crG,\Lambda,\phi)$
\[
(\lambda,g,\mu)=\lambda g\mu^*.
\]
Let $\lambda_1g_1\mu^*_1,\,\lambda_2g_2\mu^*_2\in \Lambda{}_r\times_cG{}_d\times_r\Lambda$; in the next lemma we consider the set
\begin{multline*}
\Lambda^{\min}(\lambda_1g_1\mu_1^*,\,\lambda_2g_2\mu^*_2)=\\
\{(\alpha_1,\alpha_2)\colon \lambda_1g_1(\alpha_1)=\lambda_2g_2(\alpha_2),\, \mu_1\alpha_1=\mu_2\alpha_2,\, 
|\alpha_i|=(|\lambda_1|\vee|\lambda_2|-|\lambda_i|)\vee(|\mu_1|\vee|\mu_2|-|\mu_i|)\}
\end{multline*}
and write $\pi_1:\Lambda^{\min}(\lambda_1g_1\mu^*_1,\,\lambda_2g_2\mu^*_2)\to \Lambda$ for the projection onto the first coordinate. Finally, if $G$ is any ample groupoid, we write $G^{\cop}$ for the inverse semigroup of all compact open bisections of $G$.
\begin{lem}\label{lem:intergerm} 
Let $\Psi:S(\crG,\Lambda,\phi)\to G_c^{\cop}(\crG,\Lambda,\phi)$, $\Psi(F)=\bigsqcup_{\lambda g\mu^*\in F}[\lambda g\mu^*,D_{\mu^*}]$. Write $A^c$ for the complement of a subset $A\subset G_c(\crG,\Lambda,\phi)$
\item[i)] $\Psi$ is a homomorphism of pointed inverse semigroups.
\item[ii)]
\[
\Psi(\lambda_1 g_1\mu^*_1)\cap\Psi(\lambda_2 g_2\mu_2^*)=\bigsqcup_{\alpha\in\pi_1(\Lambda^{\min}(\lambda_1 g_1\mu_1^*,\,\lambda_2 g_2\mu_2^*))}\Psi((\lambda_1g_1(\alpha)\phi(g_1,\alpha)(\mu_1\alpha)^*)
\]
\item[iii)]
\[
\Psi(\lambda_1 g_1\mu^*_1)\cap\Psi(\lambda_2 g_2\mu_2^*)^c=\Psi(\lambda_1g_1\mu_1^*)\cap\bigcap_{\alpha\in\pi_1(\Lambda^{\min}(\lambda_1 g_1\mu_1^*,\,\lambda_2 g_2\mu_2^*))}\Psi((\lambda_1g_1(\alpha)\phi(g_1,\alpha)(\mu_1\alpha)^*)^c
\]
\item[iv)] If $\lambda g\mu^*\in S(\crG,\Lambda,\phi)$ and $\alpha_1,\dots,\alpha_m\in D_{r(\mu)}$, then 
\begin{equation}\label{eq:basisg}
\Psi(\lambda g\mu^*)\cap\bigcap_{i=1}^m\Psi(\lambda g(\alpha_i)\phi(g,\alpha_i)(\mu\alpha_i)^*)^c
\end{equation}
is a compact open subset of $G=G(\crG,\Lambda,\phi)$, and the family $\cB$ of all sets of the form \eqref{eq:basisg} is a basis for the topology of $G$.
\end{lem}
\begin{proof}
\begin{itemize}
    \item[i)] Because the product of bisections distributes over disjoint unions, it suffices to consider the case of singleton elements of $S(\crG,\Lambda,\phi)$. We have
    \begin{align*}
        &\Psi(\lambda_1g_1\mu_1^*)\Psi(\lambda_2g_2\mu^*_2)\\
        = \ &[\lambda_1g_1\mu_1,D_{\mu_1^*}][\lambda_1g_1\mu_1,D_{\mu_1^*}]\\
        = \ &\{[(\lambda_1g_1\mu_1)(\lambda_2g_2\mu_2),\mu_2x]\colon \mu_1\ge \lambda_2g_2(x)\colon x\in D_{r(\mu_2)}\}\\
        = \ &\bigsqcup_{(\alpha,\beta)\in\Lambda^{\min}(\mu_1,\lambda_2)}\{[\lambda_1g_1(\alpha)\phi(g_1,\alpha)\phi(g_2,g_2^{-1}(\beta))(\mu_2g_2^{-1}(\beta))^*,\mu_2g_2^{-1}(\beta)y]\colon y\in D_{r(g_2^{-1}(\beta))} \}\\
        = \ &\Psi(\bigsqcup_{(\alpha,\beta)\in\Lambda^{\min}(\mu_1,\lambda_2)}\{\lambda_1g_1(\alpha)\phi(g_1,\alpha)\phi(g_2,g_2^{-1}(\beta))(\mu_2g_2^{-1}(\beta))^*)\})\\
        = \ &\Psi((\lambda_1g_1\mu_1^*)(\lambda_2g_2\mu^*_2)).
    \end{align*}

    \item[ii)] A groupoid element $\nabla$ is in $\Psi(\lambda_1 g_1\mu^*_1)\cap\Psi(\lambda_2 g_2\mu_2^*)$ if and only if $\nabla=[\lambda_1g_1\mu_1^*,x]=[\lambda_2g_2\mu_2^*,x]$ for some $x$. By Lemma \ref{lem:tgpd} ii), this happens if and only if $\nabla$ is in the disjoint union of the statement.

    \item[iii)] Immediate from ii).
    \item[iv)] We have
    \begin{gather*}
    \Psi(\lambda g\mu^*)\cap\bigcap_{i=1}^m\Psi(\lambda g(\alpha_i)\phi(g,\alpha_i)(\mu\alpha_i)^*)^c\\
    =[\lambda g\mu^*,D_{\mu\mu^*}\cap\bigsqcup_{i=1}^mD_{(\mu\alpha_i)(\mu\alpha_i)^*}^c],
    \end{gather*}
    which is homeomorphic to $U_{\mu;\alpha_1,\dots,\alpha_m}:=D_{\mu\mu^*}\cap\bigsqcup_{i=1}^mD_{(\mu\alpha_i)(\mu\alpha_i)^*}^c$, hence Hausdorff, since $X_{\Lambda}$ is Hausdorff. By \cite{farthing}*{Remark 5.5}, the family of all sets of the form $U_{\mu;\alpha_1,\dots,\alpha_m}$ is a basis of $X_{\Lambda}$. It follows from this and the definition of the topology of the transformation groupoid \cite{exel} that the family of all sets of the form \eqref{eq:basisg} is a basis of open subsets of $G$. In particular, \eqref{eq:basisg} is open. To see that it is compact, it suffices to show that $U_{\mu;\alpha_1,\dots,\alpha_m}$ is. Since the latter set is closed in $U_{\mu}$, it suffices to show that $U_{\mu}$ is compact, and this follows from \cite{farthing}*{Corollary 5.8}.
\end{itemize}
\end{proof}
\begin{coro}\label{coro:intergerm} 
If 
\[
\Psi(\lambda_1g_1\mu_1^*)\cap\bigcap_{i=2}^m\Psi(\lambda_ig_i\mu_i^*)^c\ne\emptyset
\]
then it contains $[\lambda_1g_1\mu_1^*,\mu_1]$.
\end{coro}
\begin{proof}
Applying part iii) of Lemma \ref{lem:intergerm}, we obtain 
\begin{equation}\label{eq:intersection}
   \Psi(\lambda_1g_1\mu_1^*)\cap\bigcap_{i=2}^m\Psi(\lambda_ig_i\mu_i^*)^c= \Psi(\lambda_1g_1\mu_1^*)\cap\bigcap_{i=2}^m\cap_{\alpha_i\in\pi_1(\Lambda^{\min}(\lambda_1 g_1\mu_1^*,\,\lambda_i g_i \mu_i^*))}\Psi((\lambda_1g_1(\alpha_i)\phi(g_1,\alpha_i)(\mu_1\alpha_i)^*)^c
\end{equation}
If for some $2\le i\le m$ $\pi_1(\Lambda^{\min}(\lambda_1 g_1\mu_1^*,\,\lambda_i g_i \mu_i^*))$ contains a vertex, then the above intersection is empty. Hence if it is nonempty, it must contain $[\lambda_1g_1\mu_1^*,\mu_1]$. 
\end{proof}

\subsubsection{Cohn algebras}
Let $k\ge 1$, $\crG$ a discrete groupoid, and $T=(\crG,\Lambda,\phi)$ a finitely aligned self-similar $k$-graph.
The \emph{Cohn algebra} of $T$ over a unital commutative ring $\ell$ is the Steinberg $\ell$-algebra $C(T)=C_\ell(T)$ of the Cohn groupoid $G_c(T)$. 

\begin{lem}\label{lem:stct}
The map $S(T)\to C(T)$, $F\mapsto \chi_{\Psi(F)}$ is a homomorphism of pointed semigroups, and induces a surjective algebra homomorphism $\psi:\ell[S(T)]\to C(T)$.    
\end{lem}
\begin{proof}
That the first map is a homomorphism of pointed semigroups follows from  Lemma \ref{lem:intergerm} and the fact that for any ample groupoid $G$, $\chi:G^{\cop}\to \ell[G]$, $U\mapsto\chi_U$ is a pointed semigroup homomorphism. It follows from  \ref{lem:intergerm} that finite intersections of elements of $\cB$ are disjoint unions of elements of $\cB$. Hence by \cite{acm}*{Lemma 2.3.5 (i)}, the set $\{\chi_U\colon U\in\cB\}$ generates $\ell[G(T)]$ as an $\ell$-module. Hence to prove that $\psi$ is surjective it suffices to show that $\chi_U$ is in the image of $\psi$ for every $U\in\cB$. We have
\begin{gather*}
\chi_{\Psi(\lambda g\mu^*)\cap \bigcap_{i=1}^n \Psi(\lambda g(\alpha_i)\phi(g,\alpha_i)(\mu\alpha_i)^*)^c}=
\chi_{\Psi(\lambda g)}\left(\prod_{i=1}^n\chi_{\Psi(r(\mu))}-\chi_{\Psi(\alpha_i\alpha_i^*)})\right)\chi_{\Psi(\mu^*)}. 
\end{gather*}
\end{proof}

Even though the algebra homomorphism of Lemma \ref{lem:stct} need not be injective, we shall often abuse notation and write 
$\lambda g\mu^*$ for the image of the latter element in $C(T)$, that is, for $\chi_{\Psi(\lambda g\mu^*)}$.
Remark also that, as $ \crG$ is discrete, its Steinberg $\ell$-algebra is isomorphic to the algebra of the semigroup $S( \crG)= \crG\sqcup\{0\}$, pointed by $0$, and where the product of any two elements of $ \crG$ is their composition whenever the latter is defined, and is $0$ otherwise; the isomorphism $\ell[ \crG]\to \ell[S( \crG)]$ maps $\chi_g\mapsto g$. For this reason we shall abuse notation and identify $\ell[ \crG]$ with $\ell[S( \crG)]$ and $g$ with $\chi_{g}$.

Next we shall see that $C(T)$ is an iterated Toeplitz-Cuntz-Pimsner algebra. Notations and vocabulary are as in Appendix \ref{app2:lescure}.

For $0\le i\le k$, consider the subcategory
$$
\Lambda\supset\Lambda_{\le i}=\{\lambda\colon |\lambda|_j=0\,\forall i<j\le k\}.
$$ 
Thus $\Lambda_{\le 0}=\Lambda^0$ and $\Lambda_{\le k}=\Lambda$. For $i<k$, consider the following $\ell$-submodules of $C(T)$
\begin{gather*}
R_i=\lspan_\ell\{\lambda g\mu^*\colon \lambda,\mu\in\Lambda_{\le i}\}\\
P_{i}=\lspan_\ell\Lambda^{\epsilon_{i+1}}R_i,\, Q_i=P_i^*=R_i\lspan_\ell(\Lambda^{\epsilon_{i+1})^*}\\
\end{gather*}

\begin{lem}\label{lem:gxi}
Consider the $\ell[\Lambda^0]$-bimodules $M_i=\ell[\Lambda^{\epsilon_i}]$, $N_i=\ell[(\Lambda^{\epsilon_i})^*]$.
 \item[i)] $R_i$ is an $\ell$-subalgebra of $C(T)$.
 \item[ii)] The canonical map $\ell[\crG]\to R_0$ is an isomorphism of $\ell$-algebras.
 \item[iii)] The canonical maps $M_i\otimes_{\ell[\Lambda^0]}R_i\to P_i$ and $R_i\otimes_{\ell[\Lambda^0]}N_i\to Q_i$ are isomorphisms of right and of left $R_i$-modules, respectively. In particular $M_i\otimes_{\ell[\Lambda^0]}R_i$ and $R_i\otimes_{\ell[\Lambda^0]}N_i$ are $R_i$-bimodules.
 \item[iv)] For each $n\ge 0$, the product of $C(T)$ induces isomorphisms of $R_i$-bimodules 
 $\ell[\Lambda ^{n\epsilon_{i+1}}]\otimes_{\ell[\Lambda^0]}R_i\to P_i^{\otimes_{R_i} n}$ 
 and $R_i\otimes_{\ell[\Lambda^0]}\ell[(\Lambda^{n\epsilon_{i+1}})^*]\to Q_i^{\otimes_{R_i}n}$. 
\end{lem}
\begin{proof}
\begin{itemize}
    \item[i)] This is immediate from Lemma \ref{lem:stct}.
    \item[ii)] Remark that $S(\crG)\subset S$ is a pointed subsemigroup and that the canonical map is the corestriction to $R_0$ of the composite of $\ell[\crG]=\ell[S(\crG)]\subset\ell[S]$ with the homomorphism $\psi$ of Lemma \ref{lem:stct}. In particular the map of the current lemma, which is clearly surjective, is an algebra homomorphism. To prove that it is injective is equivalent to showing that the set $\psi(\crG)=\{\psi(g)\colon g\in \crG\}$ is $\ell$-linearly independent. Observe that by Lemma \ref{lem:tgpd}, if $g,h\in \crG$, then $[g,d(g)]=[h,d(h)]$ if and only if $g=h$. So $\psi(g)([h,d(h)])=\delta_{g,h}$, and thus $\psi(\crG)$ is linearly independent. 
    \item[iii)] Both maps are clearly surjective. Let $e_1,\dots,e_m$ be distinct elements of $\Lambda^{\epsilon_i}$ and let $x_1,\dots,x_m\in R_i$. If $\sum_{i=1}^me_ix_i=0$ in $P_i$ then multiplying on the left by $e_j^*$ we obtain that $r(e_j)x_j=0$, and so $e_j\otimes x_j=e_jr(e_j)\otimes x_j=e_j\otimes r(e_j)x_j=0$ in $M_i\otimes_{\ell[\Lambda^0]}R_i$. This proves that the first map is an isomorphism; a similar argument works for the second map.
    \item[iv)] This is immediate once one shows that $R_i\otimes_{R_i}P_i\to P_i$ and $Q_i\otimes_{R_i}R_i\to Q_i$ are bijective. Observe that $\Lambda^0\subset R_i$ is a system of local units; hence both maps are injective by Remark \ref{rem:unitality}. Surjectivity is immediate from the fact that $s(e)\cdot ex=ex$ $\forall e\in \Lambda^{\epsilon_{i+1}}$ and $x\in R_i$. 
\end{itemize}
\end{proof}
\begin{lem}\label{lem:nonzero}
Let $0\le i<k$ and let $f\in R_i$. If $f\ne 0$, then there exist $\xi,\eta\in\Lambda_{\le i}$ and $h\in G_{r(\xi)}^{r(\eta)}$ such that $f([\xi h\eta,\eta])\ne 0$.
\end{lem}
\begin{proof}
Because $f\in R_i$, we may write it as $f=\sum_{j=1}^mc_j\lambda_jg_j\mu_j^*$ for some $\lambda_j,\mu_j\in \Lambda_{\le i}$. Hence for $U_j=[\lambda_jg_j\mu_j^*, D_{\mu_j^*}]$, 
\begin{gather*}
 \supp(f)\subset U:=\bigcup_{j=1}^mU_j=\bigsqcup_{I\subset\{1,\dots,m\}}\bigcap_{p\in I}U_p\cap\bigcap_{q\notin I}U_q^c   
\end{gather*}
Remark that $\Lambda_{\le i}$ is closed under $\Lambda^{\min}$. Hence using Lemma \ref{lem:intergerm} ii) and \eqref{eq:intersection} we may rewrite $U$ as a disjoint union of sets of the form
\[
\Psi(\xi_1h_1\eta_1^*)\cap\bigcap_{j=2}^s\cap_{\alpha_i\in\pi_1(\Lambda^{\min}(\xi_1 h_1\eta_1^*,\,\xi_i h_j \eta_j^*))}\Psi((\xi_1h_1(\alpha_i)\phi(h_1,\alpha_i)(\eta_1\alpha_i)^*)^c
\]
where all the $\xi_j,\eta_j,\alpha_j\in\Lambda_{\le i}$. Now rewrite $f$ as an $\ell$-linear combination of characteristic functions of the sets above, as $f\ne 0$, one of the above sets must be non-empty and carry a nonzero coefficient $c\in\ell$. Hence by  Corollary \ref{coro:intergerm}, there are $\xi,\eta\in\Lambda_{\le i}$ and $h\in G_{r(\xi)}^{r(\eta)}$ such that $f([\xi h\eta,\eta])=c\ne 0$. 
\end{proof}

Remark that the product of $C(T)$ induces a balanced $R_i$-bilinear map 
\[
\cdot:  Q_i\times P_i\to R_i
\]
Hence we have an $R_i$-system 
\begin{equation}\label{eq:system}
\cX_i=(P_i,Q_i,\cdot)    
\end{equation}

\begin{lem}\label{lem:xiskosher}
    
The $R_i$-system $\cX_i$ satisfies the hypothesis of Theorem \ref{thm:toeplitz}.
\end{lem}
\begin{proof}
Put $R=R_i$, $P=P_i$, and $Q=Q_i$. We have already shown, in the course of the proof of part iv) of Lemma \ref{lem:gxi}, that the maps $R\otimes_RP\to P$ 
and $Q\otimes_RR\to Q$ are isomorphisms. By part ii) of the same Lemma, 
$\cK(\cX_i)=\ell[E_i^1]\otimes_{\ell[\Lambda^0]}R\otimes_{\ell[\Lambda^0]}\ell[(E^1_i)^*]$. If $e_1,\dots,e_n\in E_i^1$ and $x_1,\dots,x_n\in R$, then 
$\Theta=\sum_{j=1}^ne_j\otimes r(e_j)\otimes e_j^*\in\cK(\cX_i)$ satisfies $\Theta\cdot e_j\otimes x_j=e_j\otimes x_j$ and 
$x_j\otimes e_j^*\cdot\Theta=x_j\otimes e_j^*$ for all $1\le j\le n$. Hence $\cX_i$ satisfies FS. Let $f_P:P\to R^{(E_i^1)}$, 
$f_P(e\otimes x)=\epsilon_{e}x$ and $f_Q:Q\to R^{(E_i^1)}$, $f_Q(x\otimes e^*)=x\epsilon_{e^*}$. It is clear that $f_P$ and $f_Q$ are respectively right and 
left $R$-linear. Furthermore we have
\begin{gather*}
 \langle f_Q(x\otimes e^*),f_P(f\otimes y)\rangle=\langle x\epsilon_e, \epsilon_fy\rangle\\
 =\delta_{e,f}xy=(xe^*)\cdot (fy)=\langle x\otimes e^*,f\otimes y\rangle.
\end{gather*}
Hence $(f_P,f_Q)$ is a weak homomorphism of $R$-systems $\cX_i\to R^{(E_i^1)}$. It follows that $\cX_i$ satisfies BL. This finishes the proof. 
\end{proof}

\begin{lem}\label{lem:rigraded}
Let $0\le i<k$. The algebra $R_{i+1}$ admits a $\Z$-grading where for $n\in \Z$, the homogeneous component of degree $n$ is $$(R_{i+1})_n=\lspan_\ell\{\lambda g\mu^*\colon |\lambda|_{i+1}-|\mu|_{i+1}=n,\, |\lambda|_j=|\mu|_j=0\, \forall j>i+1\}$$.
\end{lem}
\begin{proof}
By Lemma \ref{lem:tgpd}, $R_k=C(\Lambda)$ is  $\Zz^k$-graded; composing with the projection onto the $i+1$-th coordinate we obtain a $\Zz$-grading, for which the elements $\lambda g\mu^*$ are homogeneous of degree $|\lambda|_{i+1}-|\mu|_{i+1}$. This $\Zz$-grading restricts to $R_{i+1}$ because the latter is $\ell$-linearly generated by homogeneous elements.
\end{proof}

\begin{thm}\label{thm:ctistoep}
The canonical map $\cT(\cX_i)\to R_{i+1}$ is an isomorphism of\, $\Zz$-graded $\ell$-algebras. 
\end{thm}
\begin{proof}
By definition, $R_{i+1}$ is $\ell$-linearly spanned by the elements $\lambda g\mu^*$ with $|\lambda|_j=|\mu|_j=0$ for $j>i+1$. By the factorisation property of $\Lambda$, we may write
$\lambda=\lambda_1\lambda_2$, $\mu=\mu_1\mu_2$ with $\lambda_1,\mu_1\in\Nz_0\epsilon_{i+1}$ and 
$|\lambda_2|_{i+1}=|\mu_2|_{i+1}=0$. Hence $x=\lambda_2g\mu_2^*\in R_i$ and 
$\lambda g\mu^*=\lambda_1 x\mu_1^*$ is in the image of the map of the theorem. In particular the
latter map is surjective. Since it is homogeneous with respect to the canonical $\Zz$-gradings, 
it suffices to check injectivity on each homogeneous component. By Lemma \ref{lem:gxi} iv), the 
homogeneous component of degree $n$ of $\cT=\cT(\cX_i)$ is spanned by all the elements of the 
form $\lambda x\mu^*$ with $x\in R_i$ and $|\lambda|,|\mu|\in \Nz_0\epsilon_{i+1}$ with 
$|\lambda|_{i+1}-|\mu|_{i+1}=n$. Let $y\in\cT_n$; if $n>0$ we may write
$y=\sum_{j=1}^m \lambda_j y_j$ with all the $\lambda_j$ distinct, $|\lambda_j|=n\epsilon_{i+1}$ 
and $y_j=r(\lambda_j)y_j\in\cT_0$. If $y\mapsto 0$ in $R_{i+1}$ then left-multiplying by 
$\lambda_j^*$ we get that $y_j\mapsto 0$. A similar argument works for $n<0$, so it suffices to 
prove that the zero degree component $f:\cT_0\to (R_{i+1})_0$ of the map of the theorem is a monomorphism. Assume otherwise that there exists $0\ne y\in \ker(f)$. Write $y=\sum_p \lambda_px_p\mu_p^*$ as a sum of nonzero terms, with 
$x_p\in R_i$, $|\lambda_p|=|\mu_p|$, $|\lambda_p|_j=0$ $\forall p,\, j\ne i+1$, and all $(\lambda_p,\mu_p)$ distinct.
Choose $q$ such that $|\lambda_q|$ is minimal; then 
\begin{gather}\label{eq:zerosum}
0\ne \lambda_q^*y\mu_q=x_q+\sum_p \alpha_px_p\beta^*_p\in\ker(f)
\end{gather}
where the sum runs over all $p\ne q$ such that there exist $\alpha_p\in \pi_1(\Lambda^{\min}(\lambda_q,\lambda_p)$ and $\beta_p\in\Lambda^{\min}(\mu_q,\mu_p)$. By Lemma \ref{lem:nonzero}, there exist $\xi,\eta\in \Lambda_{\le i}$ and $g\in G_{r(\xi)}^{r(\eta)}$ such that $\gamma=[\xi g\eta^*,\eta]\in \supp(x_q)$. Hence $\gamma\in\supp(\sum_p\alpha_px_p\beta^*_p)$ and therefore $\gamma\in \supp(\alpha_px_p\beta^*_p)$ for some $p$, and thus $\eta\le \beta_p$. This is a contradiction, since $|\eta|_{i+1}|=0$ and $|\beta_p|_{i+1}>0$. 
\end{proof}
\begin{coro}\label{coro:ctistoep1}
 There is a canonical isomorphism $C(\crG,\Lambda_{\le i},\phi)\to R_i$ $\forall\, 0\le i\le k$.
\end{coro}
\begin{proof}
 Fix a discrete groupoid $\crG$. We prove by induction on $i$ that if $k\ge 0$, $\Lambda$ any self-similar $k$-graph over $\crG$, and $i\le k$ then for the subalgebra $R_{i}(\Lambda)\subset C(\crG,\Lambda,\phi)$ of Lemma \ref{lem:gxi} we have an isomorphism $C(\crG,\Lambda_{\le i},\phi)\cong R_i(\Lambda)$. The base case $i=0$ is part ii) of Lemma \ref{lem:gxi}. Assume $j\ge 1$ and that the assertion we want to prove is true for all $i<j$ and all self-similar $k$-graphs over $\crG$ with $k\ge i$. Let $k\ge j$ and let $\Lambda$ be a $k$-graph. Then $R_{j-1}(\Lambda)\cong C(\crG,\Lambda_{\le j-1},\phi)\cong R_{j-1}(\Lambda_{\le j})$. It follows from this and from part iii) of Lemma \ref{lem:gxi} that the $R_{j-1}(\Lambda_{\le j})$-system $\cX_j$ and the $R_{j-1}(\Lambda)$-system of the same name, both correspond to the $C_{j-1}=C(\crG,\Lambda_{\le j-1},\phi)$-system 
 $\cY_{j-1}=(M_j\otimes_{\ell[\Lambda^0]}C_{j-1},C_{j-1}\otimes_{\ell[\Lambda^0]}N_j, \langle\,,\,\rangle)$-where $\langle\,,\,\rangle$ is the $C_{j-1}$-balanced and bilinear form determined by $\langle r(f)\otimes f^*,e\otimes r(e)\rangle=\delta_{e,f}r(e)$. By Theorem \ref{thm:ctistoep}, $R_j(\Lambda)\cong \cT(\cY_j)\cong R_j(\Lambda_{\le j})=C_j$; this finishes the proof.
\end{proof}

In the following corollary and elsewhere, we call an algebra homomorphism \emph{a $kk$-equivalence} if the  canonical functor $j:\ahal\to kk$ to the bivariant $K$-theory triangulated category constructed in \cite{ct} is an isomorphism. Further properties and vocabulary associated with $j$ and its universal property are recalled in Appendix \ref{app2:lescure}.

\begin{coro}\label{coro:ctistoep2}
    The inclusion $\ell[\crG]\to C(\crG,\Lambda,\phi)$ is a $kk$-equivalence. 
\end{coro}
\begin{proof}
 The map of the corollary is the composite of the inclusions $\ell[\crG]=R_0\to\cdots\to R_k=C(\crG,\Lambda,\phi)$. By Theorem \ref{thm:toeplitz} and Lemma \ref{lem:xiskosher}, each of these inclusions is a $kk$-equivalence.
\end{proof}

\subsubsection{Discretisation for self-similar higher rank graphs}
\begin{thm}\label{thm:discg}
Let $\crG$ be a discrete groupoid, $k\ge 1$, $\Lambda$ a finitely aligned self-similar $k$-graph over $\crG$, $\Lambda_*$ the associated locally finite weak semilattice and $S(\crG,\Lambda,\phi)$ the inverse semigroup of Lemma \ref{lem:larki}. Put $G=S(\crG,\Lambda,\phi)\ltimes \widehat{\Lambda_*}$, $G_d=S(\crG,\Lambda,\phi)\ltimes \Lambda$. Let $j:\ahal\to kk$ be the universal excisive, homotopy invariant and $\Lambda$-stable functor of \cite{ct}. Then the homomorphism of $\ell$-algebras 
$\theta:\ell[G_d]\to M_{\Lambda}\ell[G]$ of \eqref{map:elteta} is a $kk$-equivalence. 
\end{thm}
\begin{proof}
Let $\inc:\ell[\crG]\to C(\crG,\Lambda,\phi)$ and $\inc_d:\ell[\crG]\to \ell[G_d]$ be the inclusion maps and let 
$\iota:\ell[\crG]\to M_{\Lambda}\ell[\crG]$, $\iota(g)=\epsilon_{c(g),d(g)}g$. Use Lemma \ref{lem:sgacts} to identify $C(\crG,\Lambda,\phi)=\ell[G]$. Then the following diagram commutes
\[
\xymatrix{M_{\Lambda}\ell[\crG]\ar[r]^{M_{\Lambda}\inc}&M_{\Lambda}C(\crG,\Lambda,\phi)\\
\ell[\crG]\ar[u]^{\iota}\ar[r]^{\inc_d}& \ell[G_d]\ar[u]^\theta}
\]
In this diagram, the horizontal arrow at the bottom is induced by the inclusion $\crG \subset G_d$. One checks that the latter induces a bijection $\crG^{(0)}/\crG=\Lambda^0/\crG\isom \Lambda/G_d=G_d^{(0)}/G_d$; hence $\inc_d$ is a $kk$-equivalence by Lemma 
\ref{lem:gpd101}. Next, the upward arrow on the left of the commutative diagram above is a $kk$-equivalence by \cite{cm1}*{Proposition 3.12}, while the top horizontal arrow is a $kk$-equivalence by Corollary \ref{coro:ctistoep2}. It follows that $\theta$ is a $kk$-equivalence. 
\end{proof}

\begin{exas}\label{exas:discg}
 In \cite{eugema} Ellis and Rodr\'\i guez Cirone constructed a stable infinity category $kk_\infty$ whose homotopy category is triangle equivalent to $kk$. In particular for every pair of algebras $A,B$, there is a mapping spectrum 
 $\bbKK(A,B)$ and if an algebra homomorphism $f:B_1\to B_2$ is a $kk$-equivalence, then $\bbKK(A,f):\bbKK(A,B_1)\to \bbKK(A,B_2)$ is an equivalence of spectra. Hence it follows from Theorem \ref{thm:discg} that the action $S( \crG,\Lambda,\phi)\acts \Lambda_*$ of part ii) of Lemma \ref{lem:sgacts}  satisfies discretisation with respect to $\bbKK(A,-)$, for all $A$. In particular it satisfies discretisation with respect to $KH(-)\sim\bbKK(\ell,-)$. Similarly, the same action satisfies the contravariant version of discretisation with respect to $\bbKK(-,B)$, for all $B\in\ahal$. If $\ell$ is a field of characteristic zero, then the functor that assigns to each $A\in\ahal$ the Cuntz-Quillen $\Z/2\Z$ pro-complex $X^\infty(A)$ satisfies excision, homotopy invariance and matricial stability \cite{cq3}, so by universal property of $kk$, it sends $kk$-equivalences to equivalences. It follows that the action above also satisfies discretisation with respect to periodic cyclic homology and cohomology $HP(A,-), HP(-,B)$ for all $A,B\in\ahal$. In particular this holds for $HP(-)=HP(\ell,-)$.
\end{exas}

\subsection{Example 3: Groupoids associated to right-angled Artin groups}
\label{ss:Artin}

An \emph{undirected graph} $\Gamma$ consists of a nonempty set $X=X(\Gamma)$ and a collection $E=E(\Gamma)$ of $2$-element subsets of $X$. The elements of $X$ and $E$ are the \emph{vertices} and \emph{edges} of $\Gamma$. Two vertices $x,y\in X(\Gamma)$ are \emph{adjacent} if $\{x,y\}\in E(\Gamma)$. Let $\sim$ be the transitive hull of the adjacency relation; we say that $\Gamma$ is \emph{connected} if $x\sim y$ $\forall x,y\in X(\Gamma)$. The \emph{complementary} graph 
$\Gamma^{\coco}$ is the graph with the same vertices whose edges are the two-element subsets of $X$ that are not in $E$. To an undirected graph $\Gamma$ we associate the universal pointed monoid $S(\Gamma)$ on generators $\{V_x,V_x^*\colon x\in X(\Gamma)\}$ subject to the following relations
\begin{itemize}
    \item[SA1] $V_x^*V_x=1$ $\forall x\in X$.
    \item [SA2] $V_xV_y=V_yV_x$, $V_y^*V_x=V_xV_y^*$ if $\{x,y\}\in E$.
    \item[SA3] $V_x^*V_y=0$ if $x\ne y$ and $\{x,y\}\notin E$.    
\end{itemize} 

The \emph{Artin group} of $\Gamma$ is the universal group $A(\Gamma)$ on the same generators subject to conditions SA1 and SA2. We write $A_+(\Gamma)\subset A(\Gamma)$ for the submonoid generated by the $\{V_x\colon x\in X(\Gamma)\}$. For $\alpha,\beta\in A(\Gamma)$ we write $\alpha\ge\beta$ if $\alpha^{-1}\beta\in A_+(\Gamma)$. Thus $\ge$ is the reverse of the partial order relation considered, for example, in \cite{crisplaca1}*{Theorem 10}. In particular, $\ge$ is a left-invariant partial order with respect to which $A_+(\Gamma)$ is a semilattice in the sense of the present article.
\begin{lem}\label{lem:sartin}
\item[i)] The canonical monoid homomorphism $\iota:A_+(\Gamma)\to S(\Gamma)$, $\iota(V_x)=V_x$ is injective.
\item[ii)] $\forall s\in S(\Gamma)\setminus\{0\}$ $\exists !\, a,b\in A_+(\Gamma)$ such that $s=ab^*$.
\item[iii)] $S(\Gamma)$ is a pointed inverse monoid. In particular, 
$\Idem(S(\Gamma))=\{p\in S(\Gamma)\colon p^2=p\}$ is a semilattice. 
\item[iv)] The map $A_+(\Gamma)\to \Idem(S(\Gamma))\setminus\{0\}$, $a\mapsto aa^*$ is a semilattice isomorphism.
\end{lem}
\begin{proof} For $a\in A_+(\Gamma)$, let 
\[
L_a:A_+(\Gamma)\to A_+(\Gamma),\, b\mapsto ab
\]
We have a homomorphism of monoids $L:A_+(\Gamma)\to\cI(A_+(\Gamma))$, $a\mapsto L_a$. Because $\cI(A_+(\Gamma))$ is a pointed inverse semigroup, we also have a homomorphism
of pointed monoids $\widehat{L}:S(\Gamma)\to \cI(A_+(\Gamma))$ such that $\widehat{L}\circ\iota=L$.
By \cite{iva}*{Lemma 2.1}, every element of $S(\Gamma)$ can be written as $a=V_{x_1}\cdots V_{x_m}V_{y_n}^*\dots V_{y_1}^*$. For $\alpha=V_{x_1}\cdots V_{x_m}$ and 
$\beta=V_{y_1}\cdots V_{y_n}$ we have 
\[
L_a:\dom(L_a)=(\beta\downarrow)\to (\alpha\downarrow), \, L_a(\beta b)=\alpha b.
\]
Hence if $a'=V_{x'_1}\cdots V_{x'_{m'}}V_{y'_{n'}}^*\dots V_{y'_1}^*$ is such that $L_a=L_a'$, then for $\alpha'=V_{x'_1}\cdots V_{x'_{m'}}$ and $\beta'=V_{y'_1}\cdots V_{y'_{n'}}$ we must have $(\alpha'\downarrow)=(\alpha\downarrow)$ and $(\beta'\downarrow)=(\beta\downarrow)$ which implies that $\alpha=\alpha'$ and $\beta=\beta'$. All assertions of the lemma are straightforward from this. 
\end{proof}
We shall henceforth fix a finite undirected graph $\Gamma$  with at least two vertices such that $\Gamma^{\coco}$ is connected; we write $n+1=|X(\Gamma)$. Remark that if $n=1$ then $\Gamma^{\coco}$ is connected if and only if $E(\Gamma)=\emptyset$. Assume $n\ge 1$; choose a subgraph $\Gamma'$ with $|X(\Gamma')|=n$ such that $(\Gamma')^{\coco}$ is connected. Number the vertices of $\Gamma$ $1,\dots, n+1$ so that $X(\Gamma')=\{1,\dots, n\}$, and so the vertices adjacent to $n+1$, if any, are the first $k$. Thus for $V=V_{n+1}$ and $1\le i\le n$, we have $VV_i=V_iV$ if and only if $i\le k$. Remark that, as we are assuming that $\Gamma^{\coco}$ is connected, we must have $k<n$. 
Let $\ell$ be a commutative unital ring; let 
\[
C(\Gamma)=C_\ell(\Gamma):=\ell[S(\Gamma)]
\]
By definition, $C(\Gamma)$ is the universal $\ell$-algebra on generators $V_i,V,V_i^*,V^*$, $1\le i\le n$, subject to conditions SA$1$--$3$. Set
\[
\Omega=\{1\}\sqcup\{a\in \iota(A_+(\Gamma'))\colon aV\ne Va\}.
\]
\begin{lem}\label{lem:iva}
Let $m\ge 1$, and $\omega,\eta\in\Omega^m$. Then
\[
V^*\eta_m^*\cdots V^*\eta_1^*\omega_1V\cdots\omega_mV=\delta_{\omega,\eta}.
\]
\end{lem}
\begin{proof}
An induction argument reduces the lemma to the case $m=1$. So for $\omega,\eta\in\Omega$ we have to prove that 
 \begin{equation}\label{eq:iva}
 V^*\omega\eta^*V=\delta_{\omega,\eta}    
 \end{equation} 
 The case when $\omega=\eta=1$ is immediate from SA1. The cases $\omega=1\ne \eta$ and $\omega\ne 1=\eta$ follow from SA2 and SA3. The argument explained in \cite{iva}*{bottom of page 6006--first two paragraphs of page 6007} to establish the identity \eqref{eq:iva} upon projection to a certain quotient $C_Q(\Gamma)$ of $C(\Gamma)$ defined in loc. cit. also works to prove it in $S(\Gamma)$.
\end{proof}

Set $B_0=T_0=C(\Gamma')$; for $m\ge 1$ put
\begin{gather}
T_{m}=\lspan_\ell\{\omega_1V\cdots \omega_mV xV^*\eta^*_m\cdots V^*\eta^*_1\colon
x\in C(\Gamma'),\, \omega,\eta\in\Omega^m\},\, B_m=\sum_{i=0}^mT_i
\end{gather}

\begin{prop}\label{prop:cgamma0}
 \item[i)] The map
 \[
 M_{\Omega^m}T_0\to T_m,\, \epsilon_{\omega,\eta}x\mapsto \omega_1V\cdots \omega_mV xV^*\eta^*_m\cdots V^*\eta^*_1
\]
is an isomorphism of algebras.
\item[ii)] $T_i T_j=T_{\max\{i,j\}}$ for all $i,j\ge 0$. In particular, for all $m\ge 0$, $B_m$ is a subalgebra of $C(\Gamma)$.
\item[iii)] For each $m\ge 0$, the inclusion $B_m\subset B_{m+1}$ has a left inverse $\pi_m:B_{m+1}\to B_m$ with $\ker(\pi_m)=T_m$. In particular we have a split exact sequence of algebras
\[
0\to T_{m+1}\lra B_{m+1}\overset{\pi_m}{\lra} B_m\to 0. 
\]
\end{prop}
\begin{proof} 
The first two assertions of the proposition are straightforward from  Lemma \ref{lem:iva}. To prove iii) it suffices to show that 
\begin{equation}\label{eq:bmitm+1}
B_m\cap T_{m+1}=0    
\end{equation} for each $m\ge 0$. To abbreviate notation, for $a\in T_0$ and tuples $\omega,\eta\in\Omega^i$ we shall identify $\epsilon_{\omega,\eta}a$ with its image under the isomorphism of part i). To prove \eqref{eq:bmitm+1} we must show that if
\begin{gather}\label{eq:intzero}
\sum_{i=0}^m\sum_{(\omega,\eta)\in\Omega^i}\epsilon_{\omega,\eta}a_{\omega,\eta}=
\sum_{\lambda,\eta\in \Omega^{m+1}}\epsilon_{\lambda,\mu}b_{\lambda,\mu}    
\end{gather}
then all $b_{\lambda,\mu}$ must be zero. So assume $b_{\lambda,\mu}\ne 0$ for some $\lambda,\mu\in \Omega^{m+1}$. Multiplying both sides of the identity \eqref{eq:intzero} by $(\lambda_1V\cdots\lambda_{m+1}V)^*$ on the left and by $\mu_1V\cdots\mu_{m+1}V$ on the right, we get
\[
\sum_{i=0}^mV^*\lambda_{m+1}^*\cdots V^*\lambda_{i+1}^*a_{\lambda_{\le i},\mu_{\le i}}\mu_{i+1}V\cdots\mu_{m+1}V=b_{\lambda,\mu}.
\]
It follows from $SA1-3$ that if the left hand side of the identity above is nonzero, then it must be in the semigroup algebra of the inverse semigroup $S_k\subset \Gamma'$ generated by $V_1,\dots,V_k$. But then $b_{\lambda,\mu}$ commutes with $V$ and $V^*$, and so 
\[
\epsilon_{\lambda,\mu}b_{\lambda,\mu}=(\lambda_1V\cdots\lambda_mV)b_{\lambda,\mu}(\mu_1V\cdots\mu_mV)^*.
\]
Hence multiplying both sides of the identity \eqref{eq:intzero} by $(\lambda_1V\cdots\lambda_{m}V)^*$ on the left and by $\mu_1V\cdots\mu_{m}V$ on the right, we get
\[
a_{\lambda_{\le m},\mu_{\le m}}+\sum_{i=0}^{m-1}V^*\lambda_{m}^*\cdots V^*\lambda_{i+1}^*a_{\lambda_{\le i},\mu_{\le i}}\mu_{i+1}V\cdots\mu_{m}V=b_{\lambda,\mu}VV^*.
\]
By definition, the first term of the left hand side is in $C(\Gamma')$ and as before, the summation is in $\ell[S_k]$, so we get that $b_{\lambda,\mu}VV^*\in C(\Gamma')$. Since the elements of $\Gamma$ are $\ell$-linearly independent in $C(\Gamma)$, $b_{\lambda,\mu}\ne 0$ implies that there are $\xi,\zeta\in S_k$ and $\alpha,\beta\in\Gamma'$ such that 
\[
\alpha\beta^*=\xi\zeta^*VV^*=\xi V(V\zeta)^*.
\]
By Lemma \ref{lem:sartin} ii), this implies that $\alpha=\xi V$ so $V=\xi^*\alpha\in\Gamma'$, a contradiction. 
\end{proof}

\begin{thm}\label{thm:discartin}
Let $\Gamma$ be a finite undirected graph such that $\Gamma^{\coco}$ is connected. Let $\ell$ be a commutative unital ring, $C(\Gamma)$ the Cohn $\ell$-algebra and $\iota:\ell\to C(\Gamma)$ the inclusion. Let $\cT$ be a cocomplete stable $\infty$-category and $\cH:\ahal\to\cT$ a functor. Assume that $H$ is excisive, homotopy invariant and matricially-stable, and that it preserves inductive colimits. Then $\cH(\iota)$ is an equivalence. 
\end{thm}
\begin{proof}
Let $n+1=|X(\Gamma)|$, $n\ge 0$; we proceed by induction on $n$. If $n=0$, $\Gamma$ has only one vertex, and $C(\Gamma)$ is the Toeplitz $\ell$-algebra of \cite{ct}*{Section 4.10}, and the theorem holds more generally without the colimit assumption, by \cite{ct}*{Lemma 7.3.2}. Next assume that $n\ge 1$, and let $\Gamma'\subset \Gamma$ and $V$ be as above. The relations $SA1--3$ that define $S(\Gamma)$ are homogeneous with respect to the weights $|V|=1$, $|V^*|=-1$, $|V_i|=0$ for $1\le i\le n$. Thus we have a $\Z$-grading $C(\Gamma)=\bigoplus_{m\in\Z}C(\Gamma)_m$ with $V\in C(\Gamma)_1$, $V^*\in C(\Gamma)_{-1}$ and $V^*V=1$. Hence $\phi:C(\Gamma)_0\to C(\Gamma)_0$, $\phi(x)=VxV^*$ is a monomorphism of algebras, and by \cite{fracskewmon}*{Lemma 2.4}, $C(\Gamma)$ is isomorphic to the fractional skew monoid ring
$C(\Gamma)_0[V,V^*;\phi]$ defined in loc. cit. Now the argument of \cite{abc}*{Theorem 8.4}, formulated for $KH$, applies more generally to show that for any functor $\cH$ as in the current theorem, for $B=C(\Gamma)_0$ and the inclusion map $\inc:B\to C(\Gamma)$ we have a triangle
\begin{equation}\label{seq:fracskewmon}
\xymatrix{\cH(B)\ar[r]^{1-\phi}&\cH(B)\ar[r]^{\inc}& \cH(C(\Gamma))}    
\end{equation}
Next observe that $B=\bigcup_{m\ge 0}B_m$, the union of the subalgebras of Proposition \ref{prop:cgamma0}. Hence using the definition of $B$, and the colimit preservation, excision, and matricial stability hypotheses on $H$, we get
\[
H(B)=H(\coli_mB_m)\isom \hocolim_mH(B_m)\isom \hocolim_m \bigoplus_{i=0}^mH(T_i)\isom\hocolim_mH(T_0)^m\isom\bigoplus_{m=0}^\infty H(C(\Gamma'))
\]
Under the equivalence above, $\phi$ becomes the shift map $s$ that sends the 
$i$-th copy of $H(C(\Gamma'))$ identically onto the $i+1$-th copy, and thus by \eqref{seq:fracskewmon}, we have that
$$H(C(\Gamma))\isom \hocolim_m(H(C(\Gamma'),s))\isom HC(\Gamma')$$
and that $H$ maps the inclusion $C(\Gamma')=B_0\subset B\to C(\Gamma)$ to an equivalence. By the inductive hypothesis $H$ also maps $\ell\to C(\Gamma')$ to an equivalence. Therefore $H(\iota)$ is an equivalence; this concludes the proof. 
\end{proof}

\begin{coro}\label{coro:discartin}
Let $\cH$ be as in Theorem \ref{thm:discartin}. Then the canonical action $S(\Gamma)\acts \widehat{\Idem S(\Gamma)}$ satisfies discretisation with respect to $\cH$.
\end{coro}
\begin{proof}
Recall that by definition $\ell[S(\Gamma)]=C(\Gamma)$; using Lemma \ref{lem:sartin} to identify $\Idem(S(\Gamma))=A_+$, we have a commutative diagram
\begin{equation}\label{diag:discartin}
\xymatrix{\ell[S(\Gamma)_d]\ar[r]^{\theta}& M_{A_+}C(\Gamma)\\
\ell\ar[u]\ar[r]^{\iota}& C(\Gamma)\ar[u]_{\epsilon_{1,1}}}
\end{equation}
By Theorem \ref{thm:discartin}, $H(\iota)$ is an equivalence. Next use matricial stability, and, in the case of the first map, also Lemma \ref{lem:gpd101}, to show that $H$ applied to the two vertical maps is an equivalence. It follows that $H(\theta)$ is an equivalence, concluding the proof. 
\end{proof}

\begin{ex}\label{ex:discartin}
 Homotopy algebraic $K$-theory $KH$ satisfies the hypothesis of Theorem \ref{thm:discartin} (see \cite{kh}). Hence Theorem \ref{thm:discartin} and the Corollary \ref{coro:discartin} applies to $\cH=KH$. 
\end{ex}

\subsection{\topdf{$K$}{K}-theory of \topdf{$C(\Gamma)$}{C(G)}}
Let $\Gamma$ be an undirected graph on $n+1$ vertices, $n\ge 0$, such that $\Gamma^{\coco}$ is connected. Choose a filtration $\Gamma_0\subset\dots\subset\Gamma_n=\Gamma$ by full subgraphs so that $\Gamma_i$ has $i+1$ vertices and $(\Gamma_i)^{\coco}$ is connected. Number the vertices of $\Gamma$ so that $X(\Gamma^i)=\{1,\dots,i+1\}$ and let $V_1,\dots,V_{n+1},V_1^*,\dots,V_{n+1}^*$ be the corresponding generators of $S(\Gamma)$. We have a filtration by subalgebras $\ell\subsetneq C(\Gamma_0)\subsetneq C(\Gamma_1)\subsetneq\dots\subsetneq C(\Gamma_n)=C(\Gamma)$. Equip $C(\Gamma_i)$ with the $\Zz$-grading $|\,|_i$ where for $1\le j\le i+1$, $|V_j|_i=\delta_{i+1,j}=-|V_j^*|_i$. Let $B^i=C(\Gamma_i)_0$ be the zero degree component with respect to $|\,|_i$, and $\phi_i:B^i\to B^i$, $\phi_i(b)=V_{i+1}bV_{i+1}^*$. Put 
\[
\tilde{B}^i=B^i[\phi_i^{-1}]=\coli (B^i\overset{\phi_i}{\lra}B^i\overset{\phi_i}{\lra}B^i\overset{\phi_i}{\lra}\dots)
\]
Let $\tilde{\phi}_i:\tilde{B}^i\to\tilde{B}^i$ be the algebra automorphism induced by $\phi_i$, $\tilde{B}^i[t,t^{-1};\tilde{\phi}_i]=\tilde{B}^i\rtimes_{\phi_i}\Z$, the twisted Laurent polynomial ring, and $\tilde{B}[t;\tilde{\phi}_i]\subset \tilde{B}^i[t,t^{-1};\tilde{\phi}_i]\supset \tilde{B}[t^{-1};\tilde{\phi}_i]$ the twisted polynomial ring.
Put 
\begin{gather*}
NK(\tilde{B}^i,\tilde{\phi}_i)_+=\hocofi(K(\tilde{B}^i)\to 
K(\tilde{B}^i[t,\tilde{\phi}]\\
NK(\tilde{B}^i,\tilde{\phi}_i)_-=\hocofi(K(\tilde{B}^i)\to 
K(\tilde{B}^i[t^{-1},\tilde{\phi}]
\end{gather*}

\begin{prop}\label{prop:kacegamma}
$K(C(\Gamma))\isom K(\ell)\oplus \bigoplus_{i=0}^nNK(\tilde{B}^i,\tilde{\phi}_i)_+\oplus NK(\tilde{B}^i,\tilde{\phi}_i)_-$.    
\end{prop}
\begin{proof}
It suffices to show that for $\Gamma'=\Gamma^{n-1}$, $B=B^n$, $V=V_{n+1}$ and $\phi=\phi_n$, we have 
$$
K(C(\Gamma))=K(C(\Gamma'))\oplus NK(\tilde{B},\tilde{\phi})_+\oplus NK(\tilde{B},\tilde{\phi})_-.
$$ 
As explained in the proof of Theorem \ref{thm:discartin}, $C(\Gamma)\cong C(\Gamma')[t_+,t_-,\phi]$ is the fractional skew monoid ring of \cite{fracskewmon}. Hence by \cite{abc}*{Theorem 3.6}, we have
\[
K(C(\Gamma))=\hocofi(K(B)\overset{1-\phi}{\lra} K(B))\oplus NK(\tilde{B},\tilde{\phi})_+\oplus NK(\tilde{B},\tilde{\phi})_-.
\]
So we have to prove that the canonical map
\[
K(C(\Gamma'))\to \hocofi(K(B)\overset{1-\phi}{\lra} K(B))
\]
induced by the inclusion $C(\Gamma')\subset B$ is an equivalence. This follows as in the proof of the analogous statement in Theorem \ref{thm:discartin}.
\end{proof}
\begin{coro}\label{coro:kacegamma}
 The universal action $S(\Gamma)\acts \Idem(S(\Gamma))$ satisfies discretisation with respect to algebraic $K$-theory if and only if $NK(\tilde{B}^i,\tilde{\phi}_i)_+$ is contractible for all $i=0,\dots, n$. 
\end{coro}
\begin{proof}
Applying $K$ to the arrows in diagram \eqref{diag:discartin}, we obtain that the two vertical maps are equivalences, so $K(\theta)$ is an equivalence if and only if $K(\iota)$ is. By Proposition \ref{prop:kacegamma}, $\hocofi K(\iota)$ is the direct sum of all $NK(\tilde{B}^i,\tilde{\phi}_i)_+$ and $NK(\tilde{B}^i,\tilde{\phi}_i)_-$. The proof is concluded by observing that  $\tilde{B}_i[t,\tilde{\phi}_i]\cong\tilde{B}^i[t^{-1},\tilde{\phi}_i]$, and therefore $NK(\tilde{B}^i,\tilde{\phi}_i)_-\cong NK(\tilde{B}^i,\tilde{\phi}_i)_+$.
\end{proof}
\subsection{The \topdf{$K$}{K}-theory of \topdf{$\ell[G_{\tight}(\Gamma)]$}{l[Gtight]}}

Let $\Gamma$ be an undirected graph and $S\subset X(\Gamma)$ a subset. The \emph{full subgraph} on $S$ is the graph $\Gamma_S$ with $X(\Gamma_S)=S$ and
$E(\Gamma_S)=\{\sigma\in E(\Gamma)\colon \sigma\subset S\}$; $\Gamma_S$ is \emph{complete} or \emph{clique} if $\{s,t\}\in E(\Gamma_S)$ for all $s\ne t\in S$. For brevity, we say that a subset $S\subset X(\Gamma)$ is complete to mean that $\Gamma_S$ is complete. Let $m=|X(\Gamma)|$. The \emph{Euler characteristic} of $\Gamma$ is
\[
\chi(\Gamma)=\sum_{i=0}^m (-1)^i|\{S\subset X(S) \text{ complete }\}|
\]
\begin{lem}\label{lem:sumidempo}
Let $R$ be a unital ring and $p_1,\dots,p_n\in R$ commuting idempotents. For $S\subset I_n=\{1,\dots,n\}$, put $p_S=\prod_{i\in S}$. Then
\[
[\prod_{i\in I_n}(1-p_i)]=\sum_{S\subset I}(-1)^{|S|}[p_S]\in K_0(R)
\]
\end{lem}
\begin{proof}
 It is straightforward to check that 
 \[
 \prod_{i\in I_n}(1-p_i)=\sum_{S\subset I_n}(-1)^{|S|}p_S
 \]
 as elements of $R$. What we have to check is that 
 \[
 [\sum_{S\subset I_n}(-1)^{|S|}p_S]=\sum_{S\subset I_n}(-1)^{|S|}[p_S]
 \]
 as elements of $K_0(R)$. It suffices to prove this for the universal commutative ring 
 $R=\Z[p_1,\dots,p_n\,|\, p_i^2=p_i,\, i=1,\dots,n]$. We shall use the fact that, as $R$ is commutative, $K_0(R)$ is a ring, and that if $p,q\in R$ are idempotent, we have $[p]\cdot [q]=[pq]$. We proceed by induction on $n$. If $n=1$, the assertion is clear, since $1\ge p_1$ in the order of idempotents. Let $n\ge 1$ and assume the lemma proven for $n$ commuting idempotents. Then
 \begin{gather*}
 [\sum_{S\subset I_{n+1}}(-1)^{|S|}p_S]=[\sum_{S\subset I_n}(-1)^{|S|}p_S-p_{n+1}\sum_{S\subset I_{n}}(-1)^{|S|}p_S]\\
 =[\sum_{S\subset I_n}(-1)^{|S|}p_S]-[p_{n+1}][\sum_{S\subset I_n}(-1)^{|S|}p_S]=\sum_{S\subset I_n}(-1)^{|S|}[p_S]-\sum_{S\subset I_n}(-1)^{|S|}[p_{S\cup\{n+1\}}]\\
 =\sum_{S\subset I_{n+1}}(-1)^{|S|}[p_S].
 \end{gather*}
\end{proof}

\begin{lem}\label{lem:euler}
Let $\Gamma$ and $C(\Gamma)$ be as in Theorem \ref{thm:discartin}. Then
$[\prod_{i=1}^{n+1}(1-V_iV_i^*)]=\chi(\Gamma)[1]\in K_0(C(\Gamma))$.
\end{lem}
\begin{proof}
Remark that $\prod_{i\in S}V_iV_i^*=0$ if $S$ is not complete, while if $S$ is complete is equal to $\prod_{i\in S}V_i\prod_{i\in S}V_i^*$, which is Murray-von Neumann equivalent to $1$. Now apply Lemma \ref{lem:sumidempo}. 
\end{proof}
Put
\[
C_Q(\Gamma)=C(\Gamma)/\langle\prod_{i=1}^{n+1}(1-V_iV_i^*)\rangle 
\]     

Lemma \ref{lem:tight} below says that $C_Q(\Gamma)$ is the Steinberg algebra of Exel's tight groupoid 
\[
G_{\tight}(\Gamma)=G_{\tight}(S(\Gamma))
\]
as defined in (\cite{exel}*{page 196}).
\begin{lem}\label{lem:tight}
There is a natural isomorphism $C_Q(\Gamma)\cong \ell[G_{\tight}(\Gamma)]$.    
\end{lem}
\begin{proof}
Let $S=S(\Gamma)$, $E=\Idem(S)$,  and $A_*=A_+(\Gamma)\sqcup \{\emptyset\}$. Identify $\ell[E]=\ell[\widehat{E}]$ via the canonical isomorphism $p\mapsto \chi_{\widehat{E}(p)}$. Let $\chi\in \widehat{E}$. It follows from \cite{exel}*{Proposition 11.8} that $\chi$ is tight if and only if for all $a\in A_*$ the element
\[
q_a=(\prod_{i=1}^{n+1}(aa^*-(ax_i)(ax_i)^*))
\]
satisfies $q_a(\chi)=0$. Hence the subset of tight characters is closed in $Z=\widehat{E}$ and thus coincides with the tight spectrum 
$Z_t=\widehat{E}_{\tight}$. Let $U=Z\setminus Z_t$; then  $\ell[U]\triqui \ell[E]$ is the ideal of all elements supported in $U$. From what we have just 
seen, $q_a\in I(U)$ for all $a\in A_*$. We shall show that $I(U)$ is in fact linearly spanned by the $q_a$. Let $\chi\in Z$ and $\cF=\supp(\chi)$; then 
$\cF$ is an upset, it does not contain $0$, and is closed under products.  It follows that if $\cF$ is finite, then it has a minimum element $p=aa^*$; thus 
$\chi=\cP_p$ and $q_a(\chi)\ne 0$, so $\chi\in U$. We claim that $\chi$ is tight if  $\cF$ is infinite. Let $p=aa^*\in E$; if $p\notin \cF$, then the fact 
that $\cF$ is an upset implies that $q_a(\chi)=0$. If $p\in \cF$, then because $\cF$ is infinite, we may choose a word $b$ with length $|b|>|a|$ such that 
$bb^*\in \cF$. Then $pbb^*=cc^*$ for some word $c$ with $|c|>|a|$. Because $\cF$ is an upset this implies that there exists $i$ such that $px_ix_i^*\in\cF$, 
and thus $q_a(\chi)=0$. Thus the claim is proved. It follows that $U$ consists of the principal characters, and that $q_a$ corresponds to 
$\chi_{\cP_{aa^*}}$, the characteristic function of the singleton on the principal character of $aa^*$. As $q_a$ is continuous, it follows that $\{\cP_{aa^*}\}$ is open. Hence $U$ is discrete, and thus the $q_a$ span $\ell[U]$, as we wanted 
to show. Next put $G=G_u(S)$, $G_{t}=G_{\tight}(S)$; by definition $G_{t}=(G_u)_{|Z}$ and  $\ell[G_t]=\ell[G]/\ell[G_{|U}]$. Moreover it is straightforward that $\ell[G_{|U}]$ is the 
ideal $\ell[G]$ generated by $\ell[U]$. By what we have just seen, and since $q_a=aa^*q_\emptyset$, $q_\emptyset=\prod_{i=1}^{n+1}(1-V_iV_i^*)$ generates $\ell[G_{|U}]$ as a two-sided ideal of $\ell[G]$. 
This concludes the proof. 
\end{proof}

 \begin{thm}\label{thm:bquotient}
 Let $\Gamma$ and $\cH$ be as in Theorem \ref{thm:discartin}. There is a bifibration sequence
 \[
 \cH(\ell)\overset{\chi(\Gamma)}{\lra} \cH(\ell)\to \cH(C_Q(\Gamma))
 \]
 \end{thm}
 \begin{proof}
 Let $q=\prod_{i=1}^{n+1}(1-V_iV_i^*)$, $I=\langle q\rangle\triqui C(\Gamma)$. For $1\le j\le n+1$ we have $V_j^*q=V_j^*(1-V_jV_j^*)\prod_{i\ne j}(1-V_iV_i^*)=0=qV_j$. It follows from this that the $\ell$-linear map $M_{A_+}\to I$, $\epsilon_{a,b}\mapsto aqb^*$ is an isomorphism of algebras. Hence
 the algebra homomorphism $q:\ell\to I$, $x\mapsto xq$ induces an equivalence $\cH(\ell)\to \cH(I)$. Let $\inc:I\to C(\Gamma)$ be the inclusion. We have a commutative diagram with vertical equivalences, where $\iota$ is as in Theorem \ref{thm:discartin}, and whose bottom line is a bifibration sequence
 \[
 \xymatrix{\cH(\ell)\ar[d]_{\cH(q)}^{\wr}\ar[r]^f& \cH(\ell)\ar[d]_{\cH(\iota)}^{\wr}&\\
 \cH(I)\ar[r]^(.4){\cH(\inc)}& \cH(C(\Gamma))\ar[r]& \cH(C_Q(\Gamma)}
 \]
 By the universal property of the canonical functor $j:\ahal\to kk_\infty$, we have $\cH=\bar{\cH}\circ j$ for some exact functor $\bar{\cH}$. Hence the homotopy class of $\cH(\inc\circ q)$ is the image under $\bar{\cH}$ of the class in $kk(\ell,C(\Gamma))=KH_0(\C(\Gamma))$ of $\inc\circ q$, and the latter is the image of $[q]\in K_0(C(\Gamma))$ under the comparison map $K_0(C(\Gamma))\to KH_0(C(\Gamma))$. Hence $\cH(\inc\circ q)\sim\cH(\iota)\circ (\chi(\Gamma)\id_{\cH(\ell)})$ by Lemma \ref{lem:euler}. This completes the proof. 
 \end{proof}

 \begin{prop}\label{prop:kaceqgamma} With notations as in Proposition \ref{prop:kacegamma}, we have an equivalence of (nonconnective) spectra
 \begin{gather*}
 K(C_Q(\Gamma))\isom \hocofi(K(\ell)\overset{\chi(\Gamma)}{\lra} K(\ell))\oplus \bigoplus_{i=0}^nNK(\tilde{B}^i,\tilde{\phi}_i)_+\oplus NK(\tilde{B}^i,\tilde{\phi}_i)_-
 \end{gather*}
 
 \end{prop}
 \begin{proof}
We adopt the notations of the proof of Theorem \ref{thm:bquotient}. By Lemma \ref{lem:euler}, $K(\inc\circ q):K(\ell)\to K(C(\Gamma))$ factors as the cup product $-\cup\chi(\Gamma):K(\ell)\to K(\ell)$ followed by $K(\iota)$. By Proposition \ref{prop:kacegamma}, $K(\iota)$ is a split inclusion with cofiber $\bigoplus_{i=0}^n NK(\tilde{B}^i,\tilde{\phi}_i)_+\oplus NK(\tilde{B}^i,\tilde{\phi}_i)_-$. Hence
\begin{gather*}
K(C_Q(\Gamma))\isom \hocofi(K(\inc))\isom \hocofi(K(\inc\circ q))\isom \hocofi(K(\iota)\circ(-\cup\chi(\Gamma)))\\
\isom \hocofi(K(\ell)\overset{\chi(\Gamma)}{\lra} K(\ell))\oplus  \bigoplus_{i=0}^nNK(\tilde{B}^i,\tilde{\phi}_i)_+\oplus NK(\tilde{B}^i,\tilde{\phi}_i)_-.
\end{gather*}
 \end{proof}
 \begin{coro}\label{coro:kaceqgamma}
 The element $[1]\in K_0(C_Q(\Gamma))$ has order $|\chi(\Gamma)|$.
 \end{coro}
 \begin{proof}
 Let $1_{C_Q}$, $1_{C}$ and $1_\ell$ be the unit elements of $C_Q(\Gamma)$, $C(\Gamma)$ and $\ell$ and let $\pi:C(\Gamma)\to C_Q(\Gamma)$ be the projection. Then $[1_{C_Q}]=K_0(\pi[1_C])=K_0(\pi\circ\iota [1_\ell])$, and $K_0(\iota)$ is a split monomorphism, so $m[1_{C_Q}]=0$ iff $m[1_\ell]$ is a multiple of $\chi(\Gamma)[1_\ell]$. This implies that $m\in\chi(\Gamma)\Zz$, because as $\ell$ is commutative, the natural map $\Zz=K_0(\Zz)\to K_0(\ell)$, $1\mapsto [1_\ell]$ is a split monomorphism.
 \end{proof}
 
\section{\texorpdfstring{Discretisation for algebraic $K$-theory spectra of ample groupoids}{Discretisation for algebraic K-theory spectra of ample groupoids}}\label{s:Disc_KBG}
\numberwithin{equation}{section}
In this section, we establish discretisation for algebraic $K$-theory spectra of ample groupoids as constructed in \cite{Li22}.

First of all, recall the setting and constructions from \cite{Li22}. We will be rather brief and refer to \cite{Li22} for details.

Let $G$ be an ample groupoid with locally compact Hausdorff unit space $\Gn$, codomain map $\rmr$ and domain map $\rms$ (these are also called range and source maps). Let $\cR$ be the full equivalence relation on $\Nz = \gekl{1, 2, 3, \dotsc}$, i.e., $\cR = \Nz \times \Nz$, $\cR^{(0)} = \Nz$ and we view an element $(j,i) \in \cR$ as a morphism from $i$ to $j$. Equip $\cR$ with the discrete topology. Equivalently, $\cR = \bigcup_N \cR_N$, where $\cR_N$ is the full equivalence relation on $\gekl{1, \dotsc, N}$. We set out to define a small permutative category $\fB_G$. In short, objects in $\fB_G$ are just compact open subspaces $u \subseteq \Nz \times \Gn$ and morphisms in $\fB_G$ are nothing else but compact open bisections of $\cR \times G$, the direct product of the groupoids $\cR$ and $G$. Let us spell this out: The objects of $\fB_G$ are given by 
$$
 \menge{\coprod_{i=1}^m (i,U_i) \subseteq \Nz \times \Gn}{U_i \in \cC \cO}.
$$
Here $\cC \cO$ stands for the set of compact open subspaces of $\Gn$. For $m=0$, the disjoint union becomes the empty set $\emptyset$. Now given objects $u = \coprod_{i=1}^m (i,U_i)$ and $v = \coprod_{j=1}^n (j,V_j)$, a morphism with target $v$ and domain $u$ is of the form $\coprod_{j,i} (\bbs_{j,i}, \sigma_{j,i})$, where $1 \leq i \leq m, 1 \leq j \leq n$, $\bbs_{j,i}$ is the map only defined on $\gekl{i}$ which sends $i$ to $j$, $\sigma_{j,i}$ are compact open bisections of $G$ and $U_{j,i} \in \cC \cO$ such that $\rms(\sigma_{j,i}) = U_{j,i}$. Moreover, we require that for all $i$, $U_i = \coprod_j U_{j,i}$, and that for all $j$, $V_j = \coprod_i \rmr(\sigma_{j,i})$. We denote the set of objects of $\fB_G$ by $\obj \fB_G$, the set of morphisms of $\fB_G$ by $\mor \fB_G$ and the set of morphisms with target $v$ and domain $u$ by $\fB_G(v,u)$. Given $\sigma \in \fB_G(v,u)$, we set $\mft(\sigma) \defeq v$ and $\mfd(\sigma) \defeq u$.

The composition of two morphisms $\tau = \coprod_{k,j} (\bbs_{k,j}, \sigma_{k,j})$ and $\sigma = \coprod_{j,i} (\bbs_{j,i}, \sigma_{j,i})$ with $\mfd(\tau) = \mft(\sigma)$ is given by
$$
 \tau \sigma = \coprod_{k,i} (\bbs_{k,i}, \coprod_j \tau_{k,j} \sigma_{j,i}).
$$
Here the product of two bisections is by definition given by
$$
 \tau_{k,j} \sigma_{j,i} \defeq \menge{hg}{h \in \tau, \, g \in \sigma, \, \rms(h) = \rmr(g)}.
$$
Note that if we view $\tau$ and $\sigma$ as compact open bisections of $\cR \times G$, then composition is just multiplication of bisections in $\cR \times G$.

We define the functor $\oplus$. To do so, we introduce the notation that for $m \in \Nz$, $\bbt_m$ denotes the map $\Nz \to \Nz$ given by addition with $m$. Given $u = \coprod_{i=1}^m (i,U_i)$, $u' = \coprod_{i'=1}^{m'} (i',U'_{i'})$, define
$$
 u \oplus u' \defeq \coprod_{i=1}^m (i,U_i) \amalg \coprod_{i'=1}^{m'} (\bbt_m(i'),U'_{i'}).
$$
Furthermore, define $\pi_{u,u'} \in \fB_G(u' \oplus u, u \oplus u')$ by setting
$$
 \pi_{u,u'} \defeq \coprod_{i=1}^m (\bbs_{m'+i,i},U_i) \amalg \coprod_{i'=1}^{m'} (\bbs_{i',m+i'},U'_{i'}).
$$
Moreover, given $\sigma = \coprod_{j,i} (\bbs_{j,i}, \sigma_{j,i}) \in \fB_G(v,u)$ and $\sigma' = \coprod_{j',i'} (\bbs_{j',i'}, \sigma_{j',i'}) \in \fB_G(v',u')$, define
$$
 \sigma \oplus \sigma' \defeq \coprod_{j,i} (\bbs_{j,i}, \sigma_{j,i}) \amalg \coprod_{j',i'} (\bbs_{\bbt_n(j'),\bbt_m(i')}, \sigma_{j',i'}).
$$
It is now straightforward to check that $\oplus$ indeed defines a functor and that $\gekl{\pi_{u,u'}}$ are natural isomorphisms such that $\fB_G$ becomes a small permutative category, with unit $\emptyset$.

Every element $U \in \cC \cO$ will be viewed as an element of $\obj \fB_G$ by identifying $U$ with $(1,U) \in \obj \fB_G$. With this convention, it is clear that $\coprod_{i=1}^m (i,U_i) = U_1 \oplus \dotso \oplus U_m$.

Now let us suppose that $L$ is a locally finite weak semilattice equipped with an action of an inverse semigroup $S$. Our goal is to construct a permutative functor $\fB_{S \ltimes L\reg} \to \fB_{S \ltimes \widehat{L}}$ which induces a weak equivalence $\Kz \fB_{S \ltimes L\reg} \sim \Kz \fB_{S \ltimes \widehat{L}}$. To achieve this goal, let us first describe the maps on homology induced by permutative functors. Assume that $G$ and $H$ are ample groupoids and $\Phi: \: \fB_G \to \fB_H$ is a permutative functor. Consider the permutative categories $\fB_{G \acts \Gnu}$ and $\fB_{H \acts \Hnu}$ from \cite{Li22}*{\S~3}. Assume that there are permutative functors $\Phinu: \: \fB_{G \acts \Gnu} \to \fB_{H \acts \Hnu}$ for all $\nu$. Consider the permutative functors $\Delta_{\nu + 1}^{\mu}: \: \fB_{G \acts \Gnup} \to \fB_{G \acts \Gnu}$ and $\Delta_{\nu + 1}^{\mu}: \: \fB_{H \acts \Hnup} \to \fB_{H \acts \Hnu}$ from \cite{Li22}*{\S~4.3}. Assume that we have, for all $\nu$ and $\mu$, the following commutative diagram
\begin{equation}
\label{e:CDnu}
 \xymatrix{
 \fB_{G \acts \Gnup} \ar[r]^{\Delta_{\nu + 1}^{\mu}} \ar[d]_{\Phinup} & \fB_{G \acts \Gnu} \ar[d]^{\Phinu} \\
 \fB_{H \acts \Hnup} \ar[r]_{\Delta_{\nu + 1}^{\mu}} & \fB_{H \acts \Hnu} \\
 }
\end{equation}

Now let $\sfC$ be an abelian group, which will be our coefficients. Let $\Phinu_*$ be the map uniquely determined by the commutative diagram
\[
 \xymatrix{
 \ti{H}_0(\Kz \fB_{G \acts \Gnu}) \ar[r]^{\cong} \ar[d]_{\ti{H}_0(\Kz \Phinu)} & \pi_0(\Kz \fB_{G \acts \Gnu}) \ar[r]^{\cong} & \spkl{\pi_0(B \fB_{G \acts \Gnu})}_{\oplus} \ar[r]^{\cong} & \C(\Gnu,\sfC) \ar[d]^{\Phinu_*} \\
 \ti{H}_0(\Kz \fB_{H \acts \Hnu}) \ar[r]^{\cong} & \pi_0(\Kz \fB_{H \acts \Hnu}) \ar[r]^{\cong} & \spkl{\pi_0(B \fB_{H \acts \Hnu})}_{\oplus} \ar[r]^{\cong} & \C(\Hnu,\sfC)
 }
\]
Here the first horizontal maps in each row are given by inverses of Hurewicz maps, and $\spkl{\pi_0(B \fB_{G \acts \Gnu})}_{\oplus}$ is the Grothendieck group of $\pi_0(B \fB_{G \acts \Gnu})$ with respect to $\oplus$.

Now the collection of these maps $(\Phinu_*)_{\nu}$ defines a chain map $(\C(\Gnu,\sfC), \ti{\partial}_{\nu})_{\nu} \to (\C(\Hnu,\sfC), \ti{\partial}_{\nu})_{\nu}$ and hence a map $\Phi_*: \: H_*(G) \to H_*(H)$ on groupoid homology, which by definition is given by the homology groups of $(\C(\Gnu,\sfC), \ti{\partial}_{\nu})_{\nu}$ and $(\C(\Hnu,\sfC), \ti{\partial}_{\nu})_{\nu}$ (see for instance \cite{Li22}*{\S~2.3} for basics of groupoid homology). 

\bprop
\label{prop:Phi_*}
In the situation above, we have the following commutative diagram
\[
 \xymatrix{
 \ti{H}_*(\Kz \fB_G) \ar[r]^{\cong} \ar[d]_{\ti{H}_*(\Kz \Phi)} & H_*(G) \ar[d]^{\Phi_*} \\
 \ti{H}_*(\Kz \fB_H) \ar[r]^{\cong} & H_*(H)
 }
\]
where the horizontal isomorphisms are the ones constructed in \cite{Li22}*{Theorem~4.18}.
\eprop
\bproof
Just go through the proof of \cite{Li22}*{Theorem~4.18} and observe that every step is natural with respect to permutative functors as our $\Phi$, $\Phinu$.
\eproof

Now, the question is how do permutative functors as above arise? In general, given any permutative functor $\Upsilon: \: \fB_G \to \fB_H$, view $\Upsilon(u) = \bigoplus_i v_i$ as a compact open subspace of $\Nz \times \Hn = (\cR \times H)^{(0)}$ for any $u \subseteq \Gn$. This way we obtain another permutative functor $\Phi: \: \fB_G \to \fB_{\cR \times H}$ with the property that $\Phi(u) \subseteq (\cR \times H)^{(0)}$ for all $u \subseteq \Gn$. Because of this additional property, we obtain a map $\Phinu$ from $\obj \fB_{G \acts \Gnu}$ to $\obj \fB_{\cR \times H \acts (\cR \times H)^{(\nu)}}$ by sending $\sigma_1 \tensor[_{\rms}]{\times}{_{\rmr}} \dotso \tensor[_{\rms}]{\times}{_{\rmr}} \sigma_{\nu}$ to $\Phi(\sigma_1) \tensor[_{\rms}]{\times}{_{\rmr}} \dotso \tensor[_{\rms}]{\times}{_{\rmr}} \Phi(\sigma_{\nu})$. 

For these maps to give rise to functors, we need the following condition:
\begin{enumerate}
\item[(F)] Whenever there exists a morphism in $\fB_{G \acts \Gnu}$ from $\sigma_1 \tensor[_{\rms}]{\times}{_{\rmr}} \dotso \tensor[_{\rms}]{\times}{_{\rmr}} \sigma_{\nu}$ to $\tau_1 \tensor[_{\rms}]{\times}{_{\rmr}} \dotso \tensor[_{\rms}]{\times}{_{\rmr}} \tau_{\nu}$, there exists a morphism in $\fB_{\cR \times H \acts (\cR \times H)^{(\nu)}}$ from $\Phi(\sigma_1) \tensor[_{\rms}]{\times}{_{\rmr}} \dotso \tensor[_{\rms}]{\times}{_{\rmr}} \Phi(\sigma_{\nu})$ to $\Phi(\tau_1) \tensor[_{\rms}]{\times}{_{\rmr}} \dotso \tensor[_{\rms}]{\times}{_{\rmr}} \Phi(\tau_{\nu})$.
\end{enumerate}
Note that morphisms, if they exist in $\fB_{\cR \times H \acts (\cR \times H)^{(\nu)}}$, for $\nu \geq 1$, are uniquely determined by domain and target.

If (F) holds, then we obtain permutative functors $\Phinu$, and the diagrams in \eqref{e:CDnu} always commute for $\nu \geq 1$. However, we need another condition to make sure that the diagram in \eqref{e:CDnu} commutes for $\nu = 0$. Assume that (F) holds, and consider the following condition: Assume that $\sigma$ is an object in $\fB_{G \acts G}$ and $\tau$ is a compact open bisection in $\cR \times G$ such that $\rms(\tau) = \rmr(\sigma)$, so that $\tau$ gives rise to a morphism in $\fB_{G \acts G}(\tau \sigma, \sigma)$. Hence, condition (F) implies that there exists a morphism $\bm{\tau}' \in \fB_{\cR \times H \acts \cR \times H}(\Phi(\tau \sigma), \Phi(\sigma))$. At the same time, $\tau$ also gives rise to a morphism $\bm{\tau} \in \fB_G(\rmr(\tau), \rmr(\sigma))$. 
\begin{enumerate}
\item[(CD)$_0$] For any $\sigma$ and $\tau$ as above, we have $\Delta_1^0(\bm{\tau}') = \Phi(\bm{\tau})$.
\end{enumerate}

Once we have (F) and (CD)$_0$, we are in the setting of Proposition~\ref{prop:Phi_*} and thus obtain a description of the map induced by $\Phi$ and hence $\Upsilon$ on homology.

\bex
\label{ex:PhiG0-H0}
Sometimes (as we will see below), we are given a permutative functor $\Phi: \: \fB_G \to \fB_H$ which already has the property that $\Phi(u) \subseteq \Hn$ for all $u \subseteq \Gn$. In that case, once we check (the analogues of) conditions (F) and (CD)$_0$, Proposition~\ref{prop:Phi_*} provides a description of the map induced by $\Phi$ on homology. Note that in such a setting, the functors $\Phinu$ will restrict to functors $\fB_{\Gn \acts \Gnum} \to \fB_{\Hn \acts \Hnum}$, where we view $\fB_{\Gn \acts \Gnum}$ and $\fB_{\Hn \acts \Hnum}$ as subcategories of $\fB_{G \acts \Gnu}$ and $\fB_{H \acts \Hnu}$ following \cite{Li22}*{\S~4.2}. For instance, we are in a setting as described for permutative functors arising from open embeddings or fibrewise bijective proper surjections between groupoids as in \cite{Li22}*{\S~3.1}. In these cases it is easy to check that (the analogues of) (F) and (CD)$_0$ are satisfied, and the induced maps on groupoid homology provided by Proposition~\ref{prop:Phi_*} coincide with the usual ones.
\eex

Now suppose that $L$ is a locally finite weak semilattice equipped with an action of an inverse semigroup $S$. Write $\fB_d \defeq \fB_{S \ltimes L\reg}$ and $\fB \defeq \fB_{S \ltimes \widehat{L}}$. Construct the permutative functor $\Phi: \: \fB_d \to \fB$ sending $l \in L\reg$ to $\widehat{L}(l) \defeq \menge{\chi \in \widehat{L}}{\chi(l) = 1}$. More generally, $\Phi$ sends $(l_1, \dotsc, l_j)$ to $\bigoplus_i \widehat{L}(l_i)$. On morphisms, define $\Phi([s,l]) = [s, \widehat{L}(l)] \coloneq \menge{[s,\chi]}{\chi(l) = 1}$, and extend $\Phi$ analogously as we did for objects. 

Moreover, write $\fB_d^{(\nu)} \defeq \fB_{S \ltimes L\reg \acts (S \ltimes L\reg)^{(\nu)}}$ and $\fB^{(\nu)} \defeq \fB_{S \ltimes \widehat{L} \acts (S \ltimes \widehat{L})^{(\nu)}}$ and consider $\Phinu: \: \obj \fB_d^{(\nu)} \to \fB^{(\nu)}$ sending $\bigoplus_i ([s_i^{\mu},l_i^{\mu}])_{\mu}$ to $\bigoplus_i ([s_i^{\mu},\widehat{L}(l_i^\mu)])_{\mu}$. It is now straightforward to check that (F) holds and that the diagrams in \eqref{e:CDnu} commute. Moreover, it is straightforward to see that the maps $\Phinu_*$ as introduced before Proposition~\ref{prop:Phi_*} coincide with the ones appearing in \cite{Mil}*{Proposition~8.4}. Hence, combining Proposition~\ref{prop:Phi_*} with \cite{Mil}*{Proposition~8.4}, we deduce that $\Kz \Phi$ induces isomorphisms in homology. The following now follows immediately.

\bcor
\label{cor:Disc}
$\Phi$ induces a weak equivalence $\Kz \Phi: \: \Kz \fB_d \to \Kz \fB$. 
\ecor

\bremark
\label{rem:Phi-ind-by-Omega}
Let us explain how Corollary~\ref{cor:Disc} fits in with the formulation of discretisation in \S~\ref{sec:disc}. Let $\cR_{L\reg}$ denote the full equivalence relation on $L\reg$, i.e., $\cR_{L\reg} \coloneq L\reg \times L\reg$. Let $H$ be the open subgroupoid 
\[
 H \coloneq \menge{((l'',l'),[s,\chi]) \in \cR_{L\reg} \times (S \ltimes \whL)}{l' \in \dom_L(s), \, l'' = s \cdot l, \, \chi(l') = 1}
\]
of $\cR_{L\reg} \times (S \ltimes \whL)$. Consider the following translation action $S \ltimes L\reg \to \cR_{L\reg} \times (S \ltimes \whL)$ in the sense of \cites{BEM,MB,Mil24}, given by the anchor map $\cR_{L\reg} \times (S \ltimes \whL) \to L\reg, \, ((s \cdot l', l'),[s,\chi]) \ma s \cdot l'$ and the left action $S \ltimes L\reg \acts \cR_{L\reg} \times (S \ltimes \whL)$ defined by $(t,l) \bullet ((s \cdot l', l'),[s,\chi]) \coloneq ((ts \cdot l', l'), [ts,\chi])$. Note that this left action commutes with the canonical right action $\cR_{L\reg} \times (S \ltimes \whL) \curvearrowleft \cR_{L\reg} \times (S \ltimes \whL)$, as required for a translation action. Also note that the restriction of the anchor map to the unit space of $\cR_{L\reg} \times (S \ltimes \whL)$ is proper. Let $\Theta$ be the composition $S \ltimes L\reg \to H \to \cR_{L\reg} \times (S \ltimes \whL)$, where the first arrow is given by the translation action just defined and the second arrow is the canonical inclusion. Here we view all these arrows as proper groupoid correspondences, and form composition of groupoid correspondences. Then $\Theta$ agrees with $\Omega$ from \S~\ref{ss:DiscFormul} up to the canonical Morita equivalence $S \ltimes \whL \sim_M \cR_{L\reg} \times (S \ltimes \whL)$. Also note that $\Theta$ induces the map $\theta$ from \S~\ref{subsec:detour} (see \eqref{map:elteta}).
Moreover, the assignment $G \ma \fB_G$ from \cite{Li22} is functorial with respect to translation actions as above (with the property that the restriction of the anchor map to the unit space is proper). $G \ma \fB_G$ is also functorial with respect to open embeddings. Hence $\Theta$, being a composition of two such morphisms, induces a permutative functor $\fB_{\Theta} \colon \fB_{S \ltimes L\reg} \to \fB_{\cR_{L\reg} \times (S \ltimes \whL)}$. Now $\fB_{\Theta}$ agrees with $\Phi$ up to an equivalence of permutative categories $fB_{S \ltimes \whL} \sim \fB_{\cR_{L\reg} \times (S \ltimes \whL)}$. In this sense, $\Phi$ is induced by $\Omega$ from \S~\ref{ss:DiscFormul}.
\eremark

For $l \in L\reg$, define $S_l \defeq \menge{g \in S \ltimes L\reg}{\rmr(g) = l = \rms(g)}$. In other words, $S_l$ is the stabilizer group at $l$.
\bcor
\label{cor:Disc-B_G}
We have an equivalence of permutative categories $\bigoplus_{[l] \in S \backslash L\reg} \fB_{S_l} \sim \fB_d$. 
\ecor

\bcor
\label{cor:Disc-KB_G}
At the level of algebraic K-theory spectra, we obtain weak equivalences
\[
 \Kz \fB \sim \Kz \fB_d \sim \bigvee_{[l] \in S \backslash L\reg} \Kz \fB_{S_l} \sim \bigvee_{[l] \in S \backslash L\reg} \Sigma^{\infty} (B S_l)_+.
\]
\ecor
The last equivalence is provided by \cite{Tho82}*{Lemma~2.5}.

\section{\texorpdfstring{Describing algebraic $K$-theory spectra of ample groupoids}{Describing algebraic K-theory spectra of ample groupoids}}
\numberwithin{equation}{subsection}

\subsection{A contractibility criterion}

We will need the following contractibility criterion to establish excision for algebraic $K$-theory spectra of ample groupoids. In the following, given a small category $\C$, we write $\mft(\bmalpha)$ for the codomain and $\mfd(\bmalpha)$ for the domain of a morphism $\bmalpha \in \mor(\C)$, and $\C(\bmfy,\bmfx)$ denotes the set of morphisms $\menge{\bmalpha \in \mor(\C)}{\mft(\bmalpha) = \bmfy, \, \mfd(\bmalpha) = \bmfx}$. Moreover, given $\bmfx \in \obj(\C)$, we also denote the identity morphism $\bmfx \to \bmfx$ in $\mor(\C)$ by $\bmfx$.

\blemma
\label{Lem:ContractCrit}
Let $\C$ be a small category. Assume that the following hold:
\begin{enumerate}[label=(\roman*)]
\item $\C$ is not empty.
\item $\C$ is connected in the sense that the equivalence relation on $\obj(\C)$ generated by $\bmfx \sim \bmfy$ if $\C(\bmfy,\bmfx) \neq \emptyset$ is trivial (i.e., there is only one equivalence class).
\item For all finite subsets $D_0 \subseteq \obj(\C)$ and $D_1 \subseteq \mor(\C)$ with $\mft(D_1), \mfd(D_1) \subseteq D_0$, such that $D_0$ is $D_1$-connected in the sense that the equivalence relation on $D_0$ generated by $\bmfx \sim \bmfy$ if $D_1 \cap \C(\bmfy,\bmfx) \neq \emptyset$ is trivial, the following are true:
\begin{itemize}
 \item There exist maps $D_0 \to \obj(\C), \, \bmfx \ma \bmfx'$, and $D_1 \to \mor(\C), \, \bmalpha \ma \bmalpha'$, such that, for every $\bmfx \in \obj(\C)$ which is contained in $D_1$ (here we view $\bmfx$ as the identity morphism $\bmfx \to \bmfx$ in $\mor(\C)$), we have $\bmfx' \in \obj(\C)$, and we have $(\bmalpha_1 \bmalpha_2)' = \bmalpha_1' \bmalpha_2'$ for any $\bmalpha_1, \bmalpha_2 \in D_1$ such that $\bmalpha_1 \bmalpha_2 \in D_1$, and there exist morphisms $\theta_{\bmfx} \colon \bmfx' \to \bmfx$ in $\C$ for all $\bmfx \in D_0$ such that the diagram
\[
 \xymatrix{
 \bmfx' \ar[d]_{\bmalpha'} \ar[r]^{\theta_{\bmfx}} & \bmfx \ar[d]^{\bmalpha} \\
 \bmfy' \ar[r]_{\theta_{\bmfy}} & \bmfy
 }
\]
commutes for all $\bmalpha \in D_1$.
\item Moreover, there exists $\bm{\ast} \in \obj(\C)$ together with morphisms $\omega_{\bmfx} \colon \bmfx' \to \bm{\ast}$ in $\C$ such that the diagram
\[
 \xymatrix{
 \bmfx' \ar[d]_{\bmalpha'} \ar[r]^{\omega_{\bmfx}} & \bm{\ast} \\
 \bmfy' \ar[ur]_{\omega_{\bmfy}} &
 }
\]
commutes for all $\bmalpha \in D_1$.
\end{itemize}
\end{enumerate}
Then $\C$ has contractible classifying space.
\elemma
\bproof
(i) implies that the classifying space $\vert N(\C) \vert$ is not empty. (ii) implies that $\vert N(\C) \vert$ is connected. Now take an arbitrary element 
of $\pi_n( \vert N(\C) \vert )$, for $n \geq 1$. Find a continuous map $f \colon S^n \to \vert N(\C) \vert$ representing this element of 
$\pi_n( \vert N(\C) \vert )$. The image of $f$ is contained in $\vert K \vert$, where $K \subseteq N(\C)$ is a finite sub-simplicial set (in the sense of 
\cite{kerodon}*{\href{https://kerodon.net/tag/0131}{Tag 0131}}). Now set $D_0 \coloneq K_0$, $D_1 \coloneq K_1$. By (ii), we may assume without loss of 
generality that $D_0$ is $D_1$-connected. Now let $\iota \colon K \to N(\C)$ be the canonical inclusion. Define the simplicial map $\iota' \colon K \to 
N(\C)$ by $\iota'(\bmfx) \coloneq \bmfx'$ for $\bmfx \in D_0$, $\iota'(\bmalpha) \coloneq \bmalpha'$ for $\bmalpha \in D_1$, and $\iota'(\bmalpha_1, \dotsc, 
\bmalpha_p) \coloneq (\bmalpha_1', \dotsc, \bmalpha_p')$ at the $p$-th level. Moreover, the morphisms $\gekl{\theta_{\bmfx}}_{\bmfx \in D_0}$ give rise to a 
simplicial homotopy $\iota \sim_h \iota'$ (see for example \cite{Fri2012}*{Definition~8.6}) determined by $h_i^p \colon K_p \to N(\C)_{p+1}, \, (\bmalpha_1, 
\dotsc, \bmalpha_p) \ma (\bmalpha_1, \dotsc, \bmalpha_i, \theta_{\mfd(\bmalpha_i)}, \bmalpha_{i+1}', \dotsc, \bmalpha_p')$. Furthermore, define the 
simplicial map $\bigast \colon K \to N(\C)$ by $\bigast(\bmfx) \coloneq \bm{\ast}$ for $\bmfx \in D_0$, $\bigast(\bmalpha) \coloneq \bm{\ast}$ for $\bmalpha 
\in D_1$, and $\bigast(\bmalpha_1, \dotsc, \bmalpha_p) \coloneq (\bm{\ast}, \dotsc, \bm{\ast})$ at the $p$-th level. The morphisms $\gekl{\omega_{\bmfx}}
_{\bmfx \in D_0}$ give rise to a simplicial homotopy $\iota' \sim_h \bigast$ (see \cite{Fri2012}*{Definition~8.6}) determined by $h_i^p \colon K_p \to 
N(\C)_{p+1}, \, (\bmalpha_1, \dotsc, \bmalpha_p) \ma (\bm{\ast}, \dotsc, \bm{\ast}, \omega_{\mft(\bmalpha_{i+1}')}, \bmalpha_{i+1}', \dotsc, \bmalpha_p')$. 

Now let $f^{\vert K \vert}$ be the map $S^n \to \vert K \vert$ given by restricting the codomain of $f$ to $\vert K \vert$. Then we conclude
\[
 f = \vert \iota \vert \circ f^{\vert K \vert} \sim_h \vert \iota' \vert \circ f^{\vert K \vert} \sim_h \vert \bigast \vert \circ f^{\vert K \vert}.
\]
This shows that $f$ is homotopic to a constant map, which in turn implies that the element of $\pi_n( \vert N(\C) \vert )$ we started with was trivial.
\eproof

\subsection{Excision}
\label{ss:Excision}

Let $G$ be an ample groupoid. Assume that $U$ is an open, $G$-invariant subspace of $\Gn$, and let $C \defeq \Gn \setminus U$. Then $C$ is a closed $G$-invariant subspace of $\Gn$. Restriction gives rise to the groupoids $G \vert_U \defeq \menge{g \in G}{\rmr(g), \, \rms(g) \in U}$ and $G \vert_C$. We set $\fB_U \defeq \fB_{G \vert_U}$, $\fB_C \defeq \fB_{G \vert_C}$ and $\fB \defeq \fB_G$. The inclusion $G \vert_U \into G$ is an open embedding, so that we obtain an induced permutative functor $\iota: \: \fB_U \to \fB$ (see \cite{Li22}*{\S~3.1.1}). Our goal is to identify (up to homotopy) $\fB_C$ with the mapping cone of $\iota$ (in the sense of \cite{Tho82}*{\S~3}). 

To be more precise, let $\fB_{\emptyset}$ be the trivial permutative category consisting of the single object $\emptyset$ and the only morphism given by $\id_{\emptyset}$. Consider the diagram
$$
 \xymatrix{
 \fB_U \ar[r]^{\iota} \ar[d] & \fB \\
 \fB_{\emptyset} &
 }
$$
Let $\fP$ be the homotopy pushout of this diagram (or double mapping cylinder) as constructed in \cite{Tho82}*{Construction~5.1}. Our goal is to construct a permutative functor $\fP \to \fB_C$ which induces a homotopy equivalence on classifying spaces.

Let us start by describing $\fP$, following \cite{Tho82}*{Construction~5.1}. Objects in $\fB$ are pairs of the form $(\mathring{u},u)$, where $\mathring{u} \in \obj \fB_U$ and $u \in \obj \fB$. A morphism $(\mathring{u},u) \to (\mathring{v},v)$ is represented by objects $a,z \in \obj \fB_U$ and morphisms $\mathring{\sigma}: \: \mathring{u} \to a \oplus \mathring{v} \oplus z$ in $\fB_U$ and $\sigma: \: u \oplus \iota(a) \to v$ in $\fB$. Let $a', z', \mathring{\tau},\tau$ be another such set of data, i.e., $\mathring{\tau}: \: \mathring{u} \to a' \oplus \mathring{v} \oplus z'$ is a morphism in $\fB_U$ and $\tau: \: u \oplus \iota(a') \to v$ is a morphism in $\fB$. We define an equivalence relation by setting $(\mathring{\sigma},\sigma) \sim (\mathring{\tau},\tau)$ if there exist morphisms $\alpha: \: a \to a'$ and $\zeta: \: z \to z'$ in $\fB_U$ such that $\mathring{\tau}$ is given by the composite
$$
 \xymatrix{
 \mathring{u} \ar[r]^{\hspace{-0.5cm} \mathring{\sigma}} & a \oplus \mathring{v} \oplus z \ar[r]^{\alpha \oplus \mathring{v} \oplus \zeta} & a' \oplus \mathring{v} \oplus z'
 }
$$
and $\tau$ is given by the composite
$$
 \xymatrix{
 u \oplus \iota(a') \ar[r]^{\hspace{0.5cm} u \oplus \iota(\alpha)^{-1}} & u \oplus \iota(a) \ar[r]^{\hspace{0.25cm} \sigma} & v
 }
$$
Composition of morphisms is described in \cite{Tho82}*{Construction~5.1}.

Let us now construct a permutative functor $\Phi: \: \fP \to \fB_C$. Let $R: \: \fB \to \fB_C$ be the permutative functor given by restriction and deleting $\emptyset$ whenever necessary. So $R(u) = u \cap C$ for $u \subseteq \Gn$, and if for instance $u = (1,U_1) \amalg (2, U_2)$ with $U_1 \cap C = \emptyset$, then set $R(u) = (1,U_2 \cap C)$. (This is what we mean by \an{deleting $\emptyset$}). Similarly, on morphisms, $R(\sigma) = \sigma \cap (\cR \times G \vert_C)$ (again up to deleting $\emptyset$). Clearly, $R \circ \iota$ is the trivial functor sending all objects to $\emptyset$ and all morphisms to the identity $\id_{\emptyset}$. Hence, by universal property (see \cite{Tho82}*{\S~5}), we obtain a permutative functor $\Phi: \: \fP \to \fB_C$ given by $(\mathring{u},u) \ma R(u)$ and $(\mathring{\sigma},\sigma) \ma R(\sigma)$. Note that our permutative functors preserve $\oplus$ (and $\emptyset$) on the nose.

In order to show that $\Phi$ induces a homotopy equivalence of classifying spaces, we study categories of the form $\mfc \backslash \Phi$ as in \cite{Qui}*{\S~1}, for arbitrary objects $\mfc \in \obj \fB_C$. Objects in $\mfc \backslash \Phi$ are given by pairs $(\bmu, \beta)$, where $\bmu \in \obj \fP$ and $\beta \colon \mfc \to \Phi(\bmu)$ is a morphism in $\fB_C$. A morphism $(\bmu, \beta) \to (\bmu', \beta')$ in $\mfc \backslash \Phi$ is given by a morphism $\bmsigma: \: \bmu \to \bmu'$ in $\fP$ such that the diagram
\[
 \xymatrix{
 \mfc \ar[r]^{\beta} \ar[rd]_{\beta'} & \Phi(\bmu) \ar[d]^{\Phi(\bmsigma)} \\
 & \Phi(\bmu')
 }
\]
commutes.

Our goal is to prove the following.
\btheo
\label{thm:cPhi}
For all $\mfc \in \obj \fB_C$, the category $\mfc \backslash \Phi$ has contractible classifying space.
\etheo

We will use the following observation several times.
\blemma
\label{lem:Rrho=pi}
Let $o \subseteq \cR \times G$ be an open set such that $\pi \defeq o \cap (\cR \times (G \vert_C))$ is a compact open bisection in $\cR \times (G \vert_C)$. Then there exists a compact open bisection $\rho \subseteq o$ such that $R(\rho) = \pi$.
\elemma
\bproof
For all $g \in \pi$, there exists a compact open bisection $o_g \subseteq o$ such that $g \in o_g$. Then $\pi = \bigcup_{g \in \pi} R(o_g)$, and all $R(o_g)$ are compact open bisections. By compactness of $\pi$, there exists $n \in \Nz$ and $\gekl{g_i}_{i=1}^n \subseteq \pi$ such that, with $o_i \defeq o_{g_i}$, we have $\pi = \bigcup_{i=1}^n R(o_i)$. Now construct $o'_i \subseteq o_i$ such that we still have $\pi = \bigcup_i R(o'_i)$, but now $\rmr(o'_i)$ are pairwise disjoint, and similarly for $\rms(o'_i)$. To achieve this, just start with $o'_1 \defeq o_1$ and, proceeding inductively, if $o'_1, \dotsc, o'_{i-1}$ have been constructed, then set
\[
 o'_i \defeq (\rmr(o_i) \setminus (\rmr(o'_1) \cup \dotso \cup \rmr(o'_{i-1}))) o_i (\rms(o_i) \setminus (\rms(o'_1) \cup \dotso \cup \rms(o'_{i-1}))).
\]
Indeed, the $o'_i$ have pairwise disjoint codomain and domain, and we just need to verify that $\pi = \bigcup_i R(o'_i)$. Take $g \in \pi$, and let $k \defeq \min(\gekl{1 \leq j \leq n}{g \in o_j}$. Then $g \notin o_j$ for all $1 \leq j \leq k-1$, so that $\rms(g) \notin \bigcup_{j=1}^{k-1} \rms(o_j)$ and thus $\rms(g) \notin \bigcup_{j=1}^{k-1} \rms(o'_j)$. Similarly, $\rmr(g) \notin \bigcup_{j=1}^{k-1} \rmr(o'_j)$. It follows that $g \in o'_k$. This shows $\pi = \bigcup_i R(o'_i)$, as desired.

Now set $\rho \defeq \coprod_i o'_i$.
\eproof

Now let us prove Theorem~\ref{thm:cPhi}. The strategy is to use Lemma~\ref{Lem:ContractCrit}.
\bremark
\cite{Li22}*{Proposition~4.1} is not applicable because $\Phi$ is not faithful.
\eremark

\blemma
\label{lem:cPhiNonEmpty}
For all $\mfc \in \obj \fB_C$, $\mfc \backslash \Phi$ is not empty.
\elemma
\bproof
Since $\mfc$ is an open subspace of $\Nz \times C$, there exists an open subspace $o \subseteq \Nz \times \Gn$ such that $R(o) = \mfc$. Apply Lemma~\ref{lem:Rrho=pi} to this situation, with $\pi = \mfc$, to obtain a compact open bisection $u \subseteq o$ with $R(u) = \mfc$. Set $\bmu \defeq (\emptyset,u)$ and $\beta \defeq \mfc$. Then $(\bmu,\beta) \in \mfc \backslash \Phi$.
\eproof

\blemma
\label{lem:w-u,w-v}
Let $\mfc \in \obj \fB_C$ be arbitrary. For all $\bmfx = (\bmu_{\bmfx}, \beta_{\bmfx}), \bmfy = (\bmu_{\bmfy}, \beta_{\bmfy}) \in \obj (\mfc \backslash \Phi)$, there exists $\bmfz = (\bmu_{\bmfz}, \beta_{\bmfz}) \in \obj (\mfc \backslash \Phi)$ together with morphisms $\bmfz \to \bmfx$ and $\bmfz \to \bmfy$ in $\mfc \backslash \Phi$.
\elemma

\bproof
$\beta_{\bmfy} \beta_{\bmfx}^{-1}$ is a compact open bisection in $\cR \times (G \vert_C)$, so there exists an open subspace $\gamma \subseteq \cR \times G$ such that $\beta_{\bmfy} \beta_{\bmfx}^{-1} = \gamma \cap (\cR \times (G \vert_C))$. Lemma~\ref{lem:Rrho=pi}, applied to $\pi \defeq \beta_{\bmfy} \beta_{\bmfx}^{-1}$ and $o \defeq u_{\bmfy} \gamma u_{\bmfx}$, where $\bmu_{\bmfx} = (\mathring{u}_{\bmfx},u_{\bmfx})$ and $\bmu_{\bmfy} = (\mathring{u}_{\bmfy},u_{\bmfy})$, produces a compact open bisection $\rho \subseteq o$ with $R(\rho) = \pi$. Let $u_{\bmfz} \defeq \rms(\rho)$. Define 
\[
 \mathring{u}_{\bmfz} \defeq (u_{\bmfx} \setminus u_{\bmfz}) \oplus \mathring{u}_{\bmfx} \oplus (u_{\bmfy} \setminus \rmr(\rho)) \oplus \mathring{u}_{\bmfy}.
\] 
Note that $u_{\bmfx} \setminus u_{\bmfz} \subseteq \Nz \times U$ and $u_{\bmfy} \setminus \rmr(\rho) \subseteq \Nz \times U$. Set $\bmu_{\bmfz} \defeq (\mathring{u}_{\bmfz},u_{\bmfz})$ and $\beta_{\bmfz} \defeq \beta_{\bmfx}$.

Define the morphism $\bmsigma \colon \bmu_{\bmfz} \to \bmu_{\bmfx}$ in $\fP$ by setting $a_{\bmsigma} \defeq u_{\bmfx} \setminus u_{\bmfz}$, $z_{\bmsigma} \defeq (u_{\bmfy} \setminus \rmr(\rho)) \oplus \mathring{u}_{\bmfx}$, $\mathring{\sigma} \coloneq \mathring{u}_{\bmfz} \colon \mathring{u}_{\bmfz} \to (u_{\bmfx} \setminus u_{\bmfz}) \oplus \mathring{u}_{\bmfx} \oplus (u_{\bmfy} \setminus \rmr(rho)) \oplus \mathring{u}_{\bmfy}$ and $\sigma \colon u_{\bmfz} \oplus \iota(u_{\bmfx} \setminus u_{\bmfz}) \to u_{\bmfx}$ as the composite
\[
 u_{\bmfz} \oplus \iota(u_{\bmfx} \setminus u_{\bmfz}) \cong u_{\bmfz} \amalg (u_{\bmfx} \setminus u_{\bmfz}) = u_{\bmfx},
\]
where the first map is the natural morphism. Note that $R(\sigma) = R(u_{\bmfx})$, so that this data indeed defines the desired morphism.

Define the morphism $\bmtau \colon \bmu_{\bmfz} \to \bmu_{\bmfy}$ in $\fP$ by setting $a_{\bmtau} \defeq u_{\bmfy} \setminus \rmr(\rho)$, $z_{\bmtau} \defeq (u_{\bmfx} \setminus u_{\bmfz}) \oplus \mathring{u}_{\bmfx}$, 
\[
\mathring{\tau} \defeq \pi_{(u_{\bmfx} \setminus u_{\bmfz}) \oplus \mathring{u}_{\bmfx},(u_{\bmfy} \setminus \rmr(\rho)) \oplus \mathring{u}_{\bmfy}} \colon \mathring{u}_{\bmfz} \to (u_{\bmfy} \setminus \rmr(\rho)) \oplus \mathring{u}_{\bmfy} \oplus (u_{\bmfx} \setminus u_{\bmfz}) \oplus \mathring{u}_{\bmfx}
\]
and $\tau \colon u_{\bmfz} \oplus \iota(u_{\bmfy} \setminus \rmr(\rho)) \to u_{\bmfy}$ as the composite
\[
 u_{\bmfz} \oplus \iota(u_{\bmfy} \setminus \rmr(\rho)) \cong \rmr(\rho) \oplus \iota(u_{\bmfy} \setminus \rmr(\rho)) \cong \rmr(\rho) \amalg (u_{\bmfy} \setminus \rmr(\rho)) = u_{\bmfy},
\]
where the first map is given by $\rho \oplus \iota(u_{\bmfy} \setminus \rmr(\rho))$ and the second map is the natural morphism. Note that $R(\tau) = R(\rho) = \beta_{\bmfy} \beta_{\bmfx}^{-1}$, so that this data indeed defines the desired morphism.
\eproof

\bprop
\label{prop:FF'T}
Let $\mfc \in \obj \fB_C$ be arbitrary. For all finite subsets $D_0 \subseteq \obj(\mfc \backslash \Phi)$ and $D_1 \subseteq \mor(\mfc \backslash \Phi)$ with $\mft(D_1), \mfd(D_1) \subseteq D_0$, such that $D_0$ is $D_1$-connected, the following are true:
\begin{itemize}
 \item There exist maps $D_0 \to \obj(\mfc \backslash \Phi), \, \bmfx \ma \bmfx'$, and $D_1 \to \mor(\C), \, \bmalpha \ma \bmalpha'$, such that, for every $\bmfx \in \obj(\mfc \backslash \Phi)$ which is contained in $D_1$ (here we view $\bmfx$ as the identity morphism $\bmfx \to \bmfx$ in $\mor(\mfc \backslash \Phi)$), we have $\bmfx' \in \obj(\C)$, and we have $(\bmalpha_1 \bmalpha_2)' = \bmalpha_1' \bmalpha_2'$ for any $\bmalpha_1, \bmalpha_2 \in D_1$ such that $\bmalpha_1 \bmalpha_2 \in D_1$, and there exist morphisms $\theta_{\bmfx} \colon \bmfx' \to \bmfx$ in $\mfc \backslash \Phi$ for all $\bmfx \in D_0$ such that $\bmalpha \circ \theta_{\bmfx} = \theta_{\bmfy} \circ \bmalpha'$ for all morphisms $\bmalpha \colon \bmfx \to \bmfy$ in $D_1$.
 \item Moreover, there exists $\bm{\ast} \in \obj(\mfc \backslash \Phi)$ together with morphisms $\omega_{\bmfx} \colon \bmfx' \to \bm{\ast}$ in $\C$ such that $\omega_{\bmfy} \circ \bmalpha' = \omega_{\bmfx}$ for all morphisms $\bmalpha \colon \bmfx \to \bmfy$ in $D_1$.
\end{itemize}
\eprop

\bproof
Fix $\bmfw \in D_0$ and write $\bmfw = ((\mathring{u}_{\bmfw},u_{\bmfw}), \beta_{\bmfw})$. For all $\bmsigma \in D_1$, write $\bmsigma = (\mathring{\sigma},\sigma)$. Given $\bmfx \in D_0$, since $D_0$ is $D_1$-connected, there exist finitely many $\bmfx_m \in D_0$ such that $\bmfx_1 = \bmfw$, $\bmfx_n = \bmfx$ and, for all $1 \leq m \leq n-1$, there exists $\bmsigma_m \in D_1 \cap ((\mfc \backslash \Phi)(\bmfx_{m+1},\bmfx_m))$ or $\bmsigma_m \in D_1 \cap ((\mfc \backslash \Phi)(\bmfx_m,\bmfx_{m+1}))$. Form
\[
 \sigma_{\bmfx} \defeq \sigma_{n-1}^{\epsilon_{n-1}} \dotsm \sigma_1^{\epsilon_1},
\]
where $\epsilon_m = 1$ if $\bmsigma_m \in (\mfc \backslash \Phi)(\bmfx_{m+1},\bmfx_m)$ and $\epsilon_m = -1$ if $\bmsigma_m \in (\mfc \backslash \Phi)(\bmfx_m,\bmfx_{m+1})$. We will always arrange that $\sigma_{\bmfw} = u_{\bmfw}$. Note that $u_{\bmfw} \subseteq \rms(\sigma_{\bmfx})$ and $\rmr(\sigma_{\bmfx}) \subseteq u_{\bmfx}$.

Take a morphism $\bmsigma \colon \bmfx \to \bmfy$ in $D_1$ given by $\bmsigma = [\mathring{\sigma},\sigma]$. Set $\varsigma \defeq \beta_{\bmfy} \beta_{\bmfx}^{-1}$. Then $o \defeq \sigma \cap \sigma_{\bmfy} \sigma_{\bmfx}^{-1}$ is an open subspace of $\cR \times G$ with $o \cap (\cR \times (G \vert_C)) = \varsigma$. So Lemma~\ref{lem:Rrho=pi} applies and produces a compact open bisection $\rho_{\bmsigma} \subseteq o$ with $R(\rho_{\bmsigma}) = \varsigma$. In particular, $\rho_{\bmsigma} \subseteq \sigma_{\bmfy} \sigma_{\bmfx}^{-1}$, so that $\rms(\rho_{\bmsigma}) \subseteq \rmr(\sigma_{\bmfx})$. Define
\[
 v_{\bmfw} \defeq \bigcap_{\bmsigma \in D_1} \sigma_{\mfd(\bmsigma)}^{-1} \rms(\rho_{\bmsigma}) \sigma_{\mfd(\bmsigma)}.
\]
For every $\bmfx \in D_0$, define $v_{\bmfx} \defeq \sigma_{\bmfx} v_{\bmfw} \sigma_{\bmfx}^{-1}$. Note that $\sigma v_{\bmfx} = \sigma_{\bmfy} \sigma_{\bmfx}^{-1} v_{\bmfx}$ because $v_{\bmfx} \subseteq \rms(\rho_{\bmsigma}) \subseteq \rms(o)$. It follows that $\sigma v_{\bmfx} \sigma^{-1} = \sigma_{\bmfy} \sigma_{\bmfx}^{-1} v_{\bmfx} \sigma_{\bmfx} \sigma_{\bmfy}^{-1} = \sigma_{\bmfy} v_{\bmfw} \sigma_{\bmfy}^{-1} = v_{\bmfy}$. Also, observe that $v_{\bmfx} \subseteq \rmr(\sigma_{\bmfx}) \subseteq u_{\bmfx}$. 

Now define $w_{\bmfx} \defeq u_{\bmfx} \setminus v_{\bmfx}$. Moreover, given $\bmfx = ((\mathring{u}_{\bmfx}, u_{\bmfx}), \beta_{\bmfx}) \in D_0$, define $\bmfx' \coloneq ((\mathring{u}_{\bmfx} \oplus w_{\bmfx}, v_{\bmfx}), \beta_{\bmfx})$. Now suppose that $\bmsigma \colon \bmfx \to \bmfy$ is a morphism in $D_1$ given by $\bmsigma = [\mathring{\sigma},\sigma]$. First of all, note that $\sigma$ is a morphism 
\[
 (v_{\bmfx} \amalg w_{\bmfx}) \oplus \iota(a_{\bmsigma}) = u_{\bmfx} \oplus \iota(a_{\bmsigma}) \to u_{\bmfy} = v_{\bmfy} \amalg w_{\bmfy}.
\]
As $\sigma v_{\bmfx} \sigma^{-1} = v_{\bmfy}$, $\sigma$ restricts to a morphism $w_{\bmfx} \oplus \iota(a_{\bmsigma}) \to w_{\bmfy}$ and hence to a morphism $\bar{\sigma} \colon w_{\bmfx} \oplus a_{\bmsigma}\to w_{\bmfy}$. Moreover, we have the morphism $\mathring{\sigma} \colon \mathring{u}_{\bmfx} \to a_{\bmsigma} \oplus \mathring{u}_{\bmfy} \oplus z_{\bmsigma}$. 

Define $\sigma' \defeq \sigma v_{\bmfx} \colon v_{\bmfx} \to v_{\bmfy}$ and $\mathring{\sigma}'$ as the composite
\[
 \mathring{u}_{\bmfx} \oplus w_{\bmfx} \to a_{\bmsigma} \oplus \mathring{u}_{\bmfy} \oplus z_{\bmsigma} \oplus w_{\bmfx}
 \to \mathring{u}_{\bmfy} \oplus (w_{\bmfx} \oplus a_{\bmsigma}) \oplus z_{\bmsigma} \to \mathring{u}_{\bmfy} \oplus w_{\bmfy} \oplus z_{\bmsigma},
\]
where the first map is given by $\mathring{\sigma} \oplus w_{\bmfx}$, the second map is the obvious one provided by our permutative category, and the third map is given by $\mathring{u}_{\bmfy} \oplus \bar{\sigma} \oplus z_{\bmsigma}$. By construction, we arrive at a morphism $\bmsigma' \coloneq [\mathring{\sigma}',\sigma'] \colon \bmfx' \to \bmfy'$ in $\mfc \backslash \Phi$. It is now straightforward to check that $\bmfx \ma \bmfx'$, $\bmsigma \ma \bmsigma'$ has the desired functoriality properties.

Next, for $\bmfx \in D_0$, let $\theta_{\bmfx} \colon (\mathring{u}_{\bmfx} \oplus w_{\bmfx}, v_{\bmfx}) \to (\mathring{u}_{\bmfx}, u_{\bmfx})$ be given by $a_{\theta_{\bmfx}} \defeq w_{\bmfx}$, $z_{\theta_{\bmfx}} \defeq \emptyset$, $\mathring{\vartheta}_{\bmfx} \defeq \pi_{\mathring{u}_{\bmfx},w_{\bmfx}} \colon \mathring{u}_{\bmfx} \oplus w_{\bmfx} \to w_{\bmfx} \oplus \mathring{u}_{\bmfx}$ and $\vartheta_{\bmfx} \colon v_{\bmfx} \oplus \iota(w_{\bmfx}) \cong v_{\bmfx} \amalg w_{\bmfx} = u_{\bmfx}$. This data indeed defines a morphism in $\mfc \backslash \Phi$ because $R(\vartheta_{\bmfx}) = R(u_{\bmfx})$ implies that $\beta_{\bmfx} R(\vartheta_{\bmfx}) = \beta_{\bmfx'}$.

Let us check that $\bmsigma \theta_{\bmfx} = \theta_{\bmfy} \bmsigma'$ for all $\bmsigma \colon \bmfx \to \bmfy$ in $D_1$. $\bmsigma \theta_{\bmfx}$ is given by the pair
\[
 \xymatrix{
 v_{\bmfx} \oplus \iota(w_{\bmfx}) \oplus \iota(a_{\bmsigma}) \ar[rr]^{\vartheta_{\bmfx} \oplus \iota(a_{\bmsigma})} & & u_{\bmfx} \oplus \iota(a_{\bmsigma}) \ar[rr]^{\sigma} & & u_{\bmfy} \\
 \mathring{u}_{\bmfx} \oplus w_{\bmfx} \ar[rr]^{\pi_{\mathring{u}_{\bmfx}, w_{\bmfx}}} & & w_{\bmfx} \oplus \mathring{u}_{\bmfx} \ar[rr]^{w_{\bmfx} \oplus \mathring{\sigma}} & & w_{\bmfx} \oplus a_{\bmsigma} \oplus \mathring{u}_{\bmfy} \oplus z_{\bmsigma}.
 }
\]
$\theta_{\bmfy} \bmsigma'$ is given by the pair
\[
 \xymatrix{
 v_{\bmfx} \oplus \iota(w_{\bmfy}) \ar[rr]^{\sigma' \oplus \iota(w_{\bmfy})} & & v_{\bmfy} \oplus \iota(w_{\bmfy}) \ar[rr]^{\vartheta_{\bmfy}} & & u_{\bmfy} \\
 \mathring{u}_{\bmfx} \oplus w_{\bmfx} \ar[rr]^{\mathring{\sigma}'} & & \mathring{u}_{\bmfy} \oplus w_{\bmfy} \oplus z_{\bmsigma} \ar[rr]^{\pi_{\mathring{u}_{\bmfy}, w_{\bmfy}} \oplus z_{\bmsigma}} & & w_{\bmfy} \oplus \mathring{u}_{\bmfy} \oplus z_{\bmsigma}.
 }
\]
Now it is straightforward to see, using the definition of the equivalence relation for representatives of morphisms in $\fP$ and $\bar{\sigma} \colon w_{\bmfx} \oplus a_{\bmsigma} \to w_{\bmfy}$, that $\bmsigma \theta_{\bmfx} = \theta_{\bmfy} \bmsigma'$, as desired.

As the next step, set $\bm{\ast} \defeq (\emptyset, v_{\bmfw})$ and $\beta_{\bm{\ast}} \defeq \beta_{\bmfw}$. Let $\omega_{\bmfx} \colon (\mathring{u}_{\bmfx} \oplus w_{\bmfx}, v_{\bmfx}) \to (\emptyset, v_{\bmfw})$ be the morphism given by the data $a_{\omega_{\bmfx}} \defeq \emptyset$, $z_{\omega_{\bmfx}} \defeq \mathring{u}_{\bmfx} \oplus w_{\bmfx}$, $\sigma_{\bmfx}^{-1} v_{\bmfx} \colon v_{\bmfx} \to v_{\bmfw}$ and $\mathring{u}_{\bmfx} \oplus w_{\bmfx} \colon \mathring{u}_{\bmfx} \oplus w_{\bmfx} \to \mathring{u}_{\bmfx} \oplus w_{\bmfx}$.

We have $R(\sigma_{\bmfx}^{-1} v_{\bmfx}) = R(\sigma_{\bmfx}^{-1}) = \beta_{\bmfw} \beta_{\bmfx}^{-1}$. Moreover, we have $\omega_{\bmfx} = \omega_{\bmfy} \bmsigma'$ for all $\bmsigma \colon \bmfx \to \bmfy$ in $D_1$. Indeed, $\omega_{\bmfy} \bmsigma'$ is given by the pair
\[
 \xymatrix{
 v_{\bmfx} \ar[rr]^{\sigma'} & & v_{\bmfy} \ar[rr]^{\sigma_{\bmfy}^{-1} v_{\bmfy}} & &  v_{\bmfw} \\
 \mathring{u}_{\bmfx} \oplus w_{\bmfx} \ar[rr]^{\mathring{\sigma}'} & & \mathring{u}_{\bmfy} \oplus w_{\bmfy} \oplus z_{\bmsigma} \ar[rr]^{\mathring{u}_{\bmfy} \oplus w_{\bmfy} \oplus z_{\bmsigma}} & & \mathring{u}_{\bmfy} \oplus w_{\bmfy} \oplus z_{\bmsigma}.
 }
\]
Now $\sigma_{\bmfy}^{-1} v_{\bmfy} \sigma' = \sigma_{\bmfy}^{-1} v_{\bmfy} \sigma v_{\bmfx} = \sigma_{\bmfy}^{-1} v_{\bmfy} \sigma_{\bmfy} \sigma_{\bmfx}^{-1} v_{\bmfx} = v_{\bmfw} \sigma_{\bmfx}^{-1} v_{\bmfx} = \sigma_{\bmfx}^{-1} v_{\bmfx}$ and $\mathring{u}_{\bmfx} \oplus w_{\bmfx} \cong \mathring{u}_{\bmfy} \oplus w_{\bmfy} \oplus z_{\bmsigma}$ is given by $\mathring{\sigma}'$. This shows that $\omega_{\bmfx} = \omega_{\bmfy} \bmsigma'$, as desired.
\eproof

\bproof[Proof of Theorem~\ref{thm:cPhi}]
Theorem~\ref{thm:cPhi} follows from Lemma~\ref{Lem:ContractCrit} because of Lemma~\ref{lem:cPhiNonEmpty}, Lemma~\ref{lem:w-u,w-v} and Proposition~\ref{prop:FF'T}.
\eproof

\bcor
\label{cor:Excision}
$\Phi$ induces a homotopy equivalence of the classifying spaces of $\fP$ and $\fB_C$, and hence $\Kz \fP \sim \Kz \fB_C$.

Therefore, $\fB_C$ is equivalent to the mapping cone of $\iota: \: \fB_U \to \fB$, and $\Kz \fB_C$ is equivalent to the mapping cone of $\Kz \iota: \: \Kz \fB_U \to \Kz \fB$.
\ecor

\bcor \label{cor:Exci2}
There are long exact sequences
\[
 \xymatrix{
 \dotso \ar[r] & \ti{H}_*(\Kz \fB_U) \ar[r]^{(\Kz \iota)_*} & \ti{H}_*(\Kz \fB) \ar[r]^{(\Kz R)_*} & \ti{H}_*(\Kz \fB_C) \ar[r] & \dotso
 }
\]
and
\[
 \xymatrix{
 \dotso \ar[r] & \pi_*(\Kz \fB_U) \ar[r]^{(\Kz \iota)_*} & \pi_*(\Kz \fB) \ar[r]^{(\Kz R)_*} & \pi_*(\Kz \fB_C) \ar[r] & \dotso
 }
\]
\ecor
Note that the long exact sequence in homology is the usual one in groupoid homology.

\subsection{Continuity}
\label{ss:Cont}

The goal of this section is to show continuity of the functor $G \ma \fB_G$. 

Let us first explain the setting of limits of groupoids as in \cite{BLII}*{\S~3} and \cite{Li20}*{\S~5}. Let $G_k$ be ample groupoids, and suppose that we are given open subgroupoids $H_k \subseteq G_{k+1}$ and fibrewise bijective proper surjections $p_k: \: H_k \onto G_k$. Denote the open embedding $H_k \into G_{k+1}$ by $\iota_k$. Then set $G_{k,0} \defeq G_k$ and $G_{k,j+1} \defeq p_{k+j}^{-1}(G_{k,j})$. Moreover, define $\bar{G}_k \defeq \plim_j \gekl{G_{k,j}; p_{k+j}}$. The embeddings $\iota_k$ induce embeddings $\bar{G}_k \into \bar{G}_{k+1}$, and we set $\bar{G} \defeq \ilim_k \bar{G}_k$. $G_{k,j}$, $\bar{G}_k$ and $\bar{G}$ are again ample groupoids. Note that \cite{BLII}*{\S~3} and \cite{Li20}*{\S~5} only treat the case of Hausdorff groupoids, but analogous results are true in the non-Hausdorff setting.

Now let $\fB_l \defeq \fB_{G_l}$ and consider the permutative functor $I_l \defeq F_{\iota_l} F^{p_l}: \: \fB_l \to \fB_{l+1}$ (using functoriality of our construction as explained in \cite{Li22}*{\S~3.1}). Define $I_{l',l} \defeq I_{l'-1} \circ \dotso \circ I_l$ for $l' \geq l$. $\Nz$ with the usual order is a poset, hence a category. The objects are given by $\Nz$, and a morphism $l' \to l$ exists if and only if $l \leq l'$. Consider the functor $F$ from $\Nz$ to small permutative categories sending $l$ to $\fB_l$ and a morphism $l \to l'$ to $I_{l',l}$. 

Our goal is to identify $\fB_{\bar{G}}$ with the homotopy colimit of $F$.

Let $\fL \defeq {\rm hocolim} F$. First, we need the explicit construction of $\fL$ from \cite{Tho82}*{Construction~3.22}. Objects of $\fL$ are of the form
\[
 (l_1,u_1) \oplus \dotso \oplus (l_n,u_n),
\]
where $l_i \in \Nz$ and $u_i$ are objects in $F(l_i) = \fB_{l_i}$. A morphism from $(l_1,u_1) \oplus \dotso \oplus (l_n,u_n)$ to $(l'_1,u'_1) \oplus \dotso \oplus (l'_m,u'_m)$ is given by a triple $(\lambda_i, \psi, \chi_j)$, where $\psi$ is a surjective map $\gekl{1, \dotsc, n} \to \gekl{1, \dotsc, m}$, $\lambda_i: \: l_i \to l'_{\psi(i)}$ are morphisms in $\Nz$ (in our setting, this simply means that $l_i \leq l'_{\psi(i)}$), and 
\[
 \chi_j: \: \bigoplus_{\psi(i) = j} F(\lambda_i)(u_i) = \bigoplus_{\psi(i) = j} I_{l'_{\psi(i)},l_i} (u_i) \to u'_j
\]
is a morphism in $\fB_{l'_j}$.

Now write $\fB \defeq \fB_{\bar{G}}$. Let us construct a permutative functor $I: \: \fL \to \fB$. The canonical projection $p_0^l: \: \bar{G}_l \onto G_l$ and the inclusion $\bar{G}_l \into \bar{G}$ induce the permutative functor $I_{\infty,l}: \: \fB_l \to \fB$. Now our functor $I$ sends $(l_1,u_1) \oplus \dotso \oplus (l_n,u_n) = \bigoplus_i (l_i,u_i)$ to $\bigoplus_i I_{\infty,l_i}(u_i)$. Moreover, $I$ sends a morphism given by the triple $(\lambda_i, \psi, \chi_j)$ to the morphism
\[
 \xymatrix{
 \bigoplus_i I_{\infty,l_i}(u_i) = \bigoplus_j \bigoplus_{\psi(i) = j} I_{\infty,l'_j}(I_{l'_j,l_i}(u_i)) \ar[rr]^{\hspace{2.5cm} \bigoplus_j I_{\infty,l'_j}(\chi_j)} & & 
 \bigoplus_j I_{\infty,l'_j}(u'_j). 
 }
\]
Here we used that $I_{\infty,l'} \circ I_{l',l} = I_{\infty,l}$.

First, we note that $I$ is faithful. Indeed, $\psi$ is determined by which $i$th summand in $\bigoplus_i I_{\infty,l_i}(u_i)$ is mapped to which $j$th summand in $\bigoplus_j I_{\infty,l'_j}(u'_j)$, $\lambda_i$ is determined by $\psi$ as we must have $l_i \leq l'_{\psi(i)}$, and we recover $\chi_j$ from $I_{\infty,l'_j}(\chi_j)$ by applying $p_0^{l'_j}$.

Moreover, $\fB$ is a groupoid. Hence \cite{Li22}*{Proposition~4.1} applies. Fix $u \in \obj \fB$ and form the category $\fF_{u,I}$ as in \cite{Li22}*{\S~4.1}. Objects in $\fF_{u,I}$ are equivalence classes $[v,\sigma]$, where $v \in \obj \fL$ and $\sigma: \: u \to I(v)$. Here we set $(v,\sigma) \sim (w,\tau)$ if there exists an invertible morphism $a \in \fL(w,v)$ such that $I(a) \sigma = \tau$. Morphisms in $\fF_{u,I}$ from $[v,\sigma]$ to $[w,\tau]$ are equivalence classes of morphisms $f \in \fL(w,v)$ which satisfy $I(f) \sigma = \tau$. Here we set $f \sim f'$ if there exist invertible morphisms $b, b' \in u \backslash I$ such that $bf = f'b'$.

Now our goal is to prove that $\fF_{u,I}$ is contractible. Let us start with the following observation.
\blemma
$\fF_{u,I}$  is a poset.
\elemma
\bproof
Given $v \in \obj \fL$, any morphism in $\fL(v,v)$ is invertible. Hence $\fF_{u,I}$ is a poset by construction.
\eproof

We write $[v,\sigma] \leq [w,\tau]$ if there exists a morphism $[w,\tau] \to [v,\sigma]$.

\blemma
For every $u$, $\fF_{u,I} \neq \emptyset$.
\elemma
\bproof
The topology of $\bar{G}$ is generated by cylinder sets of the form $(p_0^l)^{-1}(\sigma_l)$ for compact open bisections $\sigma_l \subseteq G_l$. Thus
\[
 u = \coprod_i (p_0^{l_i})^{-1}(u_i) = \coprod_i I_{\infty,l_i}(u_i).
\]
Let $v \defeq \bigoplus_i (l_i,u_i$ and $\sigma$ given by the canonical morphism (see \cite{Li22}*{Remark~3.3})
\[
 u = \coprod_i I_{\infty,l_i}(u_i) \to \bigoplus_i I_{\infty,l_i}(u_i).
\]
Then $[v,\sigma] \in \fF_{u,I}$.
\eproof

\bprop
$\fF_{u,I}$ is directed, i.e. for all $[v,\sigma], [w,\tau] \in \obj \fF_{u,I}$, there exists $[x,\upsilon] \in \obj \fF_{u,I}$ with $[x,\upsilon] \leq [v,\sigma]$ and $[x,\upsilon] \leq [w,\tau]$.
\eprop
\bproof
Take $[v,\sigma] \in \obj \fF_{u,I}$. $\sigma$ lies in $\bar{G}_l$, hence is of the form $(p_0^L)^{-1}(\sigma_L)$ for a compact open bisection $\sigma_L \subseteq G_L$. Let $\ti{v}_L \defeq \rms(\sigma_L)$. Set $\ti{v} \defeq (L,\ti{v}_L)$. Then $I_{\infty,L}(\ti{v}) = u$. Moreover, $f = \sigma_L^{-1}$ determines a morphism from $[v,\sigma]$ to $[\ti{v},u]$. Hence $[\ti{v},u] \leq [v,\sigma]$. Similarly, construct $\ti{w}$ such that $[\ti{w},u] \leq [w,\tau]$. We may arrange that $\ti{w} = (L,\ti{w}_L)$ for the same $L \in \Nz$ as in the construction of $\ti{v}_L$. Now $I_{\infty,L}(\ti{v}) = u = I_{\infty,L}(\ti{w})$ implies that $\ti{v} = \ti{w}$. We set $[x,\upsilon] \defeq [\ti{v},u] = [\ti{w},u]$, and we are done.
\eproof

\bcor
$I$ induces a homotopy equivalence between the classifying spaces of $\fL$ and $\fB$. 
\ecor

\bcor\label{cor:continuity}
Let $G_l$ and $\bar{G}$ be as above. We have $\Kz \fB_{\bar{G}} \sim \hocolim_l \, \Kz \fB_{G_l}$.
\ecor
Note that this produces formulas in homology and homotopy for $\Kz \fB_{\bar{G}}$ (see for example \S~\ref{sss:AF}).

\subsection{Examples}
\label{ss:Ex}

In the following, we describe algebraic $K$-theory spectra of example classes of ample groupoids. For the first three example classes, we use discretisation and excision. For the last example class of AF groupoids, we use continuity.

\bremark
It is possible to develop a general approach to describing and computing algebraic $K$-theory spectra of ample groupoids using discretisation, excision and the notion of independent resolutions (for the latter, see for instance \cite{LM}*{\S~3}). It has been described in \cite{LM}*{\S~6.1} how this approach can be used to compute groupoid homology of ample groupoids (which is isomorphic to $\ti{H}_*(\Kz \fB_G)$ by \cite{Li22}), and there is a similar description for computing stable homotopy groups of algebraic $K$-theory spectra (i.e. $\pi_*(\Kz \fB_G)$).
\eremark

\subsubsection{Graphs}
\label{sss:Graphs}

Let $\cE = (\cE^0,\cE^1)$ be a graph with target and domain maps $\mft, \mfd: \: \cE^1 \to \cE^0$. Let $S$ be the left inverse hull of $\cE$ (where we view $\cE$ as a left cancellative small category). Let $E$ be the semilattice of idempotents in $S$. Let $\cE_0^0 = \menge{v \in \cE^0}{0 < \# \, \mft^{-1}(v) < \infty}$ and $E_0 \defeq E \cE_0^0$ (the set of finite paths in $\cE$ starting in $\cE_0^0$ together with a zero element). We have a decomposition $\widehat{E} = E_0\reg \amalg \partial \widehat{E}$, and the groupoid $S \ltimes \partial \widehat{E}$ is a model for graph algebras. 

An independent resolution of $S \ltimes \partial \widehat{E}$ is given by
\[
 S \ltimes E_0\reg \to S \ltimes \widehat{E} \to S \ltimes \partial \widehat{E}.
\]
This appears already in \cite{LNII}*{\S~6.1}.

Thus $\fB_{S \ltimes \partial \widehat{E}}$ is equivalent to the mapping cone of the map $\iota: \: \fB_{S \ltimes E_0\reg} \to \fB_{S \ltimes \widehat{E}}$ induced by the canonical inclusion $S \ltimes E_0\reg \into S \ltimes \widehat{E}$. Let us analyse this map and the induced map on homotopy groups.

Let $\delta: \: \fB_{S \ltimes E\reg} \to \fB_{S \ltimes \widehat{E}}$ be the permutative functor from \S~\ref{s:Disc_KBG}. Let $i: \: \fB_{S \ltimes E_0\reg} \to \fB_{S \ltimes E\reg}$ be the permutative functor induced by the canonical inclusion $E_0 \into E$. Let $k: \: \fB_{S \ltimes E_0\reg} \to \fB_{S \ltimes E\reg}$ be the permutative functor sending $v \in \cE_0^0$ to $\bigoplus_{\mft(e) = v} e$ and more generally $\mu v$ to $\bigoplus_{\mft(e) = v} \mu e$. $k$ can be decomposed into permutative functors $k_{v,e}: \: \fB_{S \ltimes E_0\reg} \to \fB_{S \ltimes E\reg}$, for $v \in \cE_0^0$ and $e \in \mft^{-1}(v)$, which send $v$ to $e$ and more generally $\mu v$ to $\mu e$, and which maps all other $\nu$ to $\emptyset$.

We have symmetric monoidal transformations (in the sense of \cite{Tho78}*{\S~1} $k \Rightarrow \bigoplus_{v,e} k_{v,e}$ and $\delta \circ k \oplus \iota \Rightarrow \delta \circ i$. 

Now let $\fB_0^0$ be the free permutative category on $\cE_0^0$ and $\fB_0$ the free permutative category on $\cE^0$. Then we have, for fixed $v, e$ as above, the following commutative diagram
\[
 \xymatrix{
 \fB_0^0 \ar[r] \ar[d]_{j^0} & \fB_0 \ar[d]^{j_{v,e}} \\
 \fB_{S \ltimes E_0\reg} \ar[r]^{k_{v,e}} & \fB_{S \ltimes E\reg}
 }
\]
Here the top horizontal arrow sends $v$ to $\mfd(e)$, $j_0$ sends $v$ to $v$ and $j_{v,e}$ sends $\mfd(e)$ to $e$ and all other $w$ to $w$. Let $j: \: \fB_0 \to \fB_{S \ltimes E\reg}$ be the canonical inclusion. Then we have a symmetric monoidal transformations (in the sense of \cite{Tho78}*{\S~1} $j \Rightarrow j_{v,e}$. 

All this implies that we get the following long exact sequence in homotopy (and similarly in homology): 

\[
\begin{tikzpicture}[descr/.style={fill=white,inner sep=1.5pt}]
        \matrix (m) [
            matrix of math nodes,
            row sep=1em,
            column sep=2.5em,
            text height=1.5ex, text depth=0.25ex
        ]
        { 0 & \bigoplus_{\cE_0^0} \pi_*(\Sz) & \bigoplus_{\cE_0} \pi_*(\Sz) & \pi_*(\Kz \fB_{S \ltimes \partial \widehat{E}}) & \dotso \\
          \dotso & \bigoplus_{\cE_0^0} \pi_0(\Sz) & \bigoplus_{\cE_0} \pi_0(\Sz) & \pi_0(\Kz \fB_{S \ltimes \partial \widehat{E}}) & 0. \\
        };

        \path[overlay,->, font=\scriptsize,>=latex]
        (m-1-1) edge (m-1-2)
        (m-1-2) edge node[midway,above] {$\rukl{\begin{smallmatrix} I - A_0^t \\ - A_1^t \end{smallmatrix}} $} (m-1-3)
        (m-1-3) edge (m-1-4)
        (m-1-4) edge (m-1-5)
        (m-1-5) edge[out=355,in=175] (m-2-1)
        (m-2-1) edge (m-2-2)
        (m-2-2) edge node[midway,below] {$\rukl{\begin{smallmatrix} I - A_0^t \\ - A_1^t \end{smallmatrix}} $} (m-2-3)
        (m-2-3) edge (m-2-4)
        (m-2-4) edge (m-2-5);
\end{tikzpicture}
\]

Here $I$ is the identity matrix, and $A_0$ and $A_1$ are given as follows: The incidence matrix $A$ of $\cE$ (defined by $A(v,w) = \# \, \menge{e \in \cE^1}{\mft(e) = v, \, \mfd(e) = w}$ is of the form
$
  \rukl{
  \begin{smallmatrix}
  A_0 & A_1 \\
  * & *
  \end{smallmatrix}
  }
$
with respect to the decomposition $\cE^0 = \cE_0^0 \amalg (\cE^0 \setminus \cE_0^0)$. 

\bremark
This generalizes the description of the algebraic K-theory spectrum for the permutative category of the ample groupoid corresponding to Thompson's group $V_n$ as the Moore spectrum $\Mz_{n-1}$ from \cite{SW}.
\eremark

\subsubsection{Groupoids arising from piecewise affine transformations}
\label{sss:AffineTrafos}

This class of examples appears also in, for instance, \cite{Li22}*{\S~2.2.7}. Let $\Pi$ be a multiplicative subgroup of $\Rz_{>0}$ and $A_+$ be the monoid $A \cap [0,\infty)$ for some subgroup $A \subseteq (\Rz,+)$. Assume that $\Pi$ acts on $A_+$ by multiplication. Let $S$ be the left inverse hull of the $ax+b$-monoid $A_+ \rtimes \Pi$ and $E$ its semilattice of idempotents. We have a decomposition $\widehat{E} = \Omega \amalg \gekl{\infty}$. We are interested in a restriction of the groupoid $S \ltimes \Omega$; let us call it $G$. The topological full groups of such groupoids are generalisations of Thompson's group $V$ (where we vary the conditions on slopes and breakpoints).

This time, $\fB_{A \ltimes \Pi}$ is equivalent to the mapping cone of the permutative functor $\iota: \: \fB_G \to \fB_{S \ltimes \widehat{E}}$ induced by the canonical embedding $G \into S \ltimes \widehat{E}$.

Consider the diagram
\[
 \xymatrix{
 \fB_G \ar[r]^{\iota} & \fB_{S \ltimes \widehat{E}} \ar[r] & \fB_{A \ltimes \Pi} \\
 & \fB_{S \ltimes E\reg} \ar[u] & \\
 & \fB_{\Pi} \ar[u] \ar[uur]^{\iota_{\Pi}} &
 }
\]
Here the arrow $\fB_{S \ltimes E\reg} \to \fB_{S \ltimes \widehat{E}}$ is the permutative functor from \S~\ref{s:Disc_KBG}, $\fB_{\Pi} \to \fB_{S \ltimes E\reg}$ is the canonical inclusion (coming from viewing $\Pi$ as the stabilizer group of $A_+ \ltimes \Pi \in E$), and $\iota_{\Pi}$ is induced by the canonical embedding $\Pi \into A \ltimes \Pi$. This diagram (or rather the right triangle in it) commutes.

Therefore, all in all, we obtain the following long exact sequence in homotopy (and similarly in homology):
\[
\begin{tikzpicture}[descr/.style={fill=white,inner sep=1.5pt}]
        \matrix (m) [
            matrix of math nodes,
            row sep=1em,
            column sep=2.5em,
            text height=1.5ex, text depth=0.25ex
        ]
        { 0 & \pi_*(\Kz \fB_G) & \pi_*(\Sigma^{\infty} (B \Pi)_+) & \pi_*(\Sigma^{\infty} (B (A \ltimes \Pi))_+) & \dotso \\
          \dotso & \pi_0(\Kz \fB_G) & \pi_0(\Sigma^{\infty} (B \Pi)_+) & \pi_0(\Sigma^{\infty} (B (A \ltimes \Pi))_+) & 0. \\
        };

        \path[overlay,->, font=\scriptsize,>=latex]
        (m-1-1) edge (m-1-2)
        (m-1-2) edge (m-1-3)
        (m-1-3) edge node[midway,above] {$(\iota_{\Pi})_*$} (m-1-4)
        (m-1-4) edge (m-1-5)
        (m-1-5) edge[out=355,in=175] (m-2-1)
        (m-2-1) edge (m-2-2)
        (m-2-2) edge (m-2-3)
        (m-2-3) edge node[midway,below] {$(\iota_{\Pi})_*$} (m-2-4)
        (m-2-4) edge (m-2-5);
\end{tikzpicture}
\]
Moreover, note that $(\iota_{\Pi})_*$ is split-injective as the canonical projection $A \ltimes \Pi \onto \Pi$ induces a split for $\iota_{\Pi}$.

\subsubsection{Groupoids arising from self-similar groups}
\label{sss:SelfSimilarGps}

Consider a self-similar group $\Gamma$ (as in, for instance, \cite{Nek}), which comes with an action $\Gamma \acts X$ on a finite set and a faithful action $\Gamma \acts X^*$. We think of $X^*$ as a tree. At the same time, $X^*$ can be canonically identified with $\Nz_0^{*d}$, where $d = \# \, X$. Form the Zappa–Sz{\'e}p product monoid $\Nz_0^{*d} \bowtie \Gamma$ and let $S$ be its left inverse hull. Let $E$ be the semilattice of idempotents in $S$. We are interested in the groupoid $S \ltimes \partial X^*$. The decomposition $\widehat{E} = X^* \amalg \partial X^*$ gives rise to the following independent resolution for $S \ltimes \partial X^*$:
\[
 S \ltimes X^* \to S \ltimes \widehat{E} \to S \ltimes \partial X^*.
\]
Thus $\fB_{S \ltimes \partial X^*}$ is equivalent to the mapping cone of the map $\iota: \: \fB_{S \ltimes X^*} \to \fB_{S \ltimes \widehat{E}}$ induced by the canonical inclusion $S \ltimes X^* \into S \ltimes \widehat{E}$. Let us analyse this map and the induced map on homotopy groups.

Consider the following diagram
\begin{equation}
\label{e:self-sim}
 \xymatrix{
 \fB_{S \ltimes X*} \ar[r]^{\iota} & \fB_{S \ltimes \widehat{E}} \ar[r] & \fB_{S \ltimes \partial X^*} \\
 \fB_{\Gamma} \ar[u] \ar@{..>}[r] & \fB_{\Gamma} \ar[u]
 }
\end{equation}
Here the vertical arrows are provided by \S~\ref{s:Disc_KBG}. Our task is to describe the map in homotopy corresponding to the dotted arrow.

For that purpose, given $u \in X$, consider the diagram
\[
 \xymatrix{
 \fB_{\Gamma} \ar[r] \ar@{..>}[dr]_{\rho_u} & \fB_{\cR_{\Gamma / \Gamma_u} \times \Gamma_u} & \\
 & \fB_{\Gamma_u} \ar[u] \ar[r]^{r_u} & \fB_{\Gamma}
 }
\]
Here $\Gamma_u$ is the stabilizer group at $u$. Moreover, the upper horizontal arrow is given as follows: Let $r$ be a set-theoretic split for the canonical quotient map $\Gamma \onto \Gamma / \Gamma_u$. Then the upper horizontal arrow is the permutative functor sending $\gamma \in \Gamma$ to $\coprod_{x \Gamma_u \in \Gamma / \Gamma_u} ((gx \Gamma_u, x \Gamma_u), r(gx \Gamma_u)^{-1} g r(x \Gamma_u))$. The vertical arrow is given by the identification of $\Gamma_u$ with $\gekl{(\Gamma_u, \Gamma_u)} \times \Gamma_u \subseteq \cR_{\Gamma / \Gamma_u} \times \Gamma_u$. The dotted arrow $\rho_u$ is the composition of the induced map in homotopy of the upper horizontal arrow and the inverse of the map induced by the vertical arrow. Furthermore, the lower horizontal arrow $r_u$ sends $\gamma$ to $\gamma \vert_u$ (using the self-similar group structure on $\Gamma$). Now it turns out that the dotted arrow in \eqref{e:self-sim} is given by $\id - \sum_{[u] \in \Gamma \backslash X} (r_u)_* \circ \rho_u$ in homotopy. 

All in all, we have the following long exact sequence in homotopy:

\[
\begin{tikzpicture}[descr/.style={fill=white,inner sep=1.5pt}]
        \matrix (m) [
            matrix of math nodes,
            row sep=1em,
            column sep=2.5em,
            text height=1.5ex, text depth=0.25ex
        ]
        { 0 & \pi_*(\Sigma^{\infty} (B \Gamma)_+) & & & \pi_*(\Sigma^{\infty} (B \Gamma)_+) & \pi_*(\Kz \fB_{S \ltimes \partial X^*}) & \dotso \\
          \dotso & \pi_0(\Sigma^{\infty} (B \Gamma)_+) & & & \pi_0(\Sigma^{\infty} (B \Gamma)_+) & \pi_0(\Kz \fB_{S \ltimes \partial X^*}) & 0. \\
        };

        \path[overlay,->, font=\scriptsize,>=latex]
        (m-1-1) edge (m-1-2)
        (m-1-2) edge node[midway,above] {$\id - \sum_{u \in \Gamma \backslash X} (r_u)_* \circ \rho_u$} (m-1-5)
        (m-1-5) edge (m-1-6)
        (m-1-6) edge (m-1-7)
        (m-1-7) edge[out=355,in=175] (m-2-1)
        (m-2-1) edge (m-2-2)
        (m-2-2) edge node[midway,below] {$\id - \sum_{u \in \Gamma \backslash X} (r_u)_* \circ \rho_u$} (m-2-5)
        (m-2-5) edge (m-2-6)
        (m-2-6) edge (m-2-7);
\end{tikzpicture}
\]

There is an analogous long exact sequence in homology. Note that the map in homology corresponding to $\rho_u$ is given by the restriction homomorphism $\res: \: H_*(\Gamma) \to H_*(\Gamma_u)$.

\bremark
Many detailed homology computations were obtained in \cite{steimil}, using the idea of independent resolutions as described above (but derived independently from this paper in \cite{steimil}), in combination with sophisticated techniques developed in \cite{steimil}.
\eremark

\subsubsection{AF groupoids}
\label{sss:AF}

Let us now apply continuity as in \S~\ref{ss:Cont}. Let $G$ be an AF groupoid. In that case $G_l = \coprod_{\Iz_l} \cR_{\kappa_l}$, where $\cR_{\kappa}$ is the full equivalence relation on a finite set with $\kappa$ elements. Hence $\fB_l = \bigoplus_{\Iz_l} \fB_{\cR_{\kappa_l}} \sim \bigoplus_{\Iz_l} \fB_{\bullet}$, where $\bullet$ denotes the trivial groupoid with just one object and only the identity morphism. Hence $\Kz \fB_l \sim \bigvee_{\Iz_l} \Sz$ and 
\[
 \Kz \fB_G \sim {\rm Hocolim}_l \, \bigvee_{\Iz_l} \Sz
\]
and the structure maps are induced by the connecting maps defining $G$ (which for instance can be read off from a Bratteli diagram describing $G$).

\appendix
\numberwithin{equation}{section}
\bigskip
\section{The local structure of proper ample groupoids}\label{app1:proper}

Let $G$ be an ample groupoid, $U \subseteq \Gu$ a compact open set and $F \subseteq \bmF(G_U^U)$ a subgroup of the topological full group of $G_U^U$. Then $G(F) \coloneq \bigcup_{\sigma \in F} \sigma$ is an open subgroupoid of $G$. Moreover, given a compact open subset $V \subseteq U$ which is $F$-invariant, we set $G(F,V) \coloneq G(F)_V^V$. $G(F,V)$ is an open subgroupoid of $G$. In the cases of interest later on, $F$ will be a finite group, and then $G(F)$ as well as $G(F,V)$ will be compact.

The proof of the following is very similar to that of \cite{Mil25b}*{Lemma~3.4}.
\blemma
\label{lem:HxxFU}
Let $H$ be a proper ample groupoid. For every $x \in \Hu$ there exists a compact open neighbourhood $U \subseteq \Hu$ of $x$ and a finite subgroup $F \subseteq \bmF(H_U^U)$ such that $\sigma\cdot x=x$ for all $\sigma \in F$, and $F \to H_x^x, \, \sigma \ma \sigma x$ is a group isomorphism.
\elemma
\bproof
For every $h \in H_x^x$, choose a compact open bisection $\tau_h$ such that $h \in \tau_h$. Since $H$ is proper, we know that $H_x^x$ is finite. Then there exists a compact open neighbourhood $\ti{U}$ of $x$ such that $\tau_1 \ti{U} = \ti{U}$, $\tau_{h^{-1}} \ti{U} = \tau_h^{-1} \ti{U}$ and $\tau_{h'} \tau_h \ti{U} = \tau_{h'h} \ti{U}$ for all $h, h' \in H_x^x$. Now define $U \coloneq \bigcap_{h \in H_x^x} \tau_h.\ti{U}$ and $\sigma_h \coloneq \tau_h U$ for all $h \in H_x^x$. By construction, we have $c(\sigma_h) = U = d(\sigma_h)$. Thus $\sigma_h \in \bmF(H_U^U)$. Moreover, $U$ is invariant under all $\tau_h$. We have
\[
 \sigma_{h'} \sigma_h = \tau_{h'} U \tau_h U = \tau_{h'} \tau_h (\tau_h^{-1}.U) U = \tau_{h'} \tau_h U = \tau_{h'h} U = \sigma_{h'h}.
\]
This shows that the $\sigma_h$ form a finite group, and we can set $F \coloneq \menge{\sigma_h}{h \in H_x^x}$.
\eproof

\begin{rem}\label{rem:HxxFU}
In the situation of Lemma \ref{lem:HxxFU}, assume that $H=S\ltimes X$ for some action $S\acts X$ of an inverse semigroup $S$ by homeomorphisms between open subsets on a Hausdorff, locally compact, totally disconnected space $X$. Then $\cB=\{[s,V]\colon \dom(s)\supset V\in\CO(X)\}$ is a basis for the topology of $H$. Hence the $\tau_h$ at the beginning of the proof of the lemma can be chosen in $\cB$, and so following the proof, every element of the group $F$ obtained is of the form $[s,U]$ for $U$ as in the lemma and some $s\in S$. 
\end{rem}

The following lemma is straightforward from \cite{Mil25b}*{Lemma 3.3}. 

\blemma
\label{lem:HVVinHF}
Let $H$ be a proper ample groupoid. For all $x \in \Hu$ there exist $U$ and $F$ as in Lemma~\ref{lem:HxxFU} and a compact open neighbourhood $V$ of $x$ with $V \subseteq U$ such that $H_V^V \subseteq H(F)$. Moreover, $V$ may be chosen to be $F$-invariant.
\elemma
\begin{proof}
Applying \cite{Mil25b}*{Lemma 3.3} to the proper groupoid $P=H$ and the open subset $O=H(F)$, we obtain an open subset $x\in U'\subset U$ such that $H_{U'}^{U'}\subset H(F)$. Then $V=\cap_{\sigma\in F}\sigma\cdot U'$ is $F$-invariant, compact and open (as $F$ is finite) and $H_V^V\subset H_{U'}^{U'}\subset H(F)$.
\end{proof}

\bcor
\label{cor:HWW}
Let $H$ be a compact ample groupoid. Let $\pi \colon \Hu \to \Hu/H$ be the canonical projection. Then there exist finitely many points $x_i \in \Hu$, and for each $i$, $U_i$ and $F_i$ as in Lemma~\ref{lem:HxxFU}, $V_i \subseteq U_i$ which is $F_i$-invariant as in Lemma~\ref{lem:HVVinHF}, compact open $F_i$-invariant subsets $W_i \subseteq V_i$ which are pairwise disjoint, such that $W \coloneq \coprod_i W_i$ satisfies $\pi(W) = \Hu / H$ and
\[
 H_W^W = \coprod_i H_{W_i}^{W_i} = \coprod_i H(F_i,W_i).
\]
\ecor
\bproof
For every $x \in \Hu$, choose $U_x$ and $F_x$ as in Lemma~\ref{lem:HxxFU} and $F_x$-invariant $V_x \subseteq U_x$ as in 
Lemma~\ref{lem:HVVinHF}. Then $\Hu = \bigcup_x V_x$. By compactness of $\Hu$, we can find a finite cover $\Hu = \bigcup_i V_{x_i}$. Set 
$U_i \coloneq U_{x_i}$, $V_i \coloneq V_{x_i}$ and $F_i \coloneq F_{x_i}$. Since $\pi$ is open, the $\pi(V_i)$ are compact open. As $H$ is compact, $\Hu / H$ is Hausdorff. It is also totally disconnected. Hence we can 
disjointify and obtain a partition $\Hu / H = \bigsqcup_i \pi(V_i)'$, where $\forall i$, $\pi(V_i)'\subset \pi(V_i)$ is compact open. Set 
$W_i \coloneq \pi^{-1}(\pi(V_i)') \cap V_i$. Then $W_i$ is $F_i$-invariant. By construction, we have $H_{W_i}^{W_i} \subseteq H_{V_i}^{V_i} 
\subseteq H(F_i,V_i)$ and thus $H_{W_i}^{W_i} = H(F_i,W_i)$. Moreover, $\pi(W_i) = \pi(V_i)'$ are pairwise disjoint. It follows that, with $W 
\coloneq \coprod_i W_i$, we have $H_W^W = \coprod_i H_{W_i}^{W_i}$. Finally, we have $\pi(W) = \bigcup_i \pi(W_i) = \Hu / H$ by construction.
\eproof

An inverse semigroup $S$ is called \emph{locally finite} if for every finite subset $F\subset S$ there exists a finite inverse subsemigroup $F\subset \bar{F}\subset S$.

\blemma\label{lem:FHLocFin} 
Let $H$ be a compact ample groupoid. Then $\bmF(H)$ is locally finite.
\elemma
\bproof
Let $W$ be as in Corollary~\ref{cor:HWW}. Then there exists a decomposition $\Hu = \coprod_{n=1}^N X_n$ into compact open subsets $X_n$ for which there exist compact open bisections in $H$ with $\rms(\sigma_n) = X_n$ and $\rmr(\sigma_n) \subseteq W$. We can and will always arrange that $X_1 = W$ and $\sigma_1 = W$. Then we have an open groupoid embedding $H \into H_W^W \times R_N = (\coprod_i H(F_i,W_i)) \times R_N$ sending $h \in H$ to $(\sigma_n h \sigma_m^{-1}, (n,m))$ if $\rms(h) \in X_m$ and $\rmr(h) \in X_n$. Here $R_N = \gekl{1, \dotsc, N} \times \gekl{1, \dotsc, N}$ is the full equivalence relation on $\gekl{1, \dotsc, N}$. Write $H' \coloneq (\coprod_i H(F_i,W_i)) \times R_N$. Now $F = \menge{(\coprod_i \sigma_i W_i) \times \sigma}{\sigma_i \in F_i, \, \sigma \in S_N}$, where $S_N$ is the symmetric group on $\gekl{1, \dotsc, N}$, is a finite subgroup of $\bmF(H')$ such that $H' = \bigcup_{\tau \in F} \tau$.

As $H$ embeds into $H'$ as an open subgroupoid, it suffices to show that $\bmF(H')$ is locally finite, which we now prove. The important property of $H'$ we will use is the existence of the finite subgroup $F \subseteq \bmF(H')$ such that $H' = \bigcup_{\tau \in F} \tau$. Now take finitely many $\sigma_i \in \bmF(H')$. Then $\rho_i = \coprod_h \tau_{ih} C_{ih}$, where $\tau_{ih} \in F$ and $C_{ih}$ are compact open subsets of $(H')^{(0)}$, and for each fixed $i$, $\cC_i \coloneq \gekl{C_{ih}}_h$ is a partition of $(H')^{(0)}$. Let $\cC$ be the common refinement of all $\cC_i$ given by all non-empty intersections of the form $\bigcap_i C_{i h_i}$. We further consider the $F$-invariant partition $\cD$ given by all non-empty intersections of the form $\bigcap_{\tau \in F} \tau.C_{\tau}$, where $C_{\tau} \in \cC$. Note that $\cD$ refines $\cC$. Now every $\rho_i$ is of the form $\coprod_j \tau_{ij} D_{ij}$ with $D_{ij} \in \cD$. Note that we also have $\tau_{ij}.D_{ij} \in \cD$ as $\cD$ is $F$-invariant. 

Now let $\Gamma$ be the set of all $\rho \in \bmF(H)$ such that $\rho = \coprod_j \tau_j D_j$ for some $\tau_j \in F$ and $D_j \in \cD$. Then $\rho_i \in \Gamma$ for all $i$. Moreover, $\Gamma$ is a group because
\[
 (\coprod_j \tau_j D_j) (\coprod_k \tau_k D_k) = \coprod_{j,k} \tau_j \tau_k ((\tau_k^{-1}.D_j) \cap D_k),
\]
and $\tau_j \tau_k \in F$, $(\tau_k^{-1}.D_j) \cap D_k \in \cD$ for all $j$, $k$. Finally, $\Gamma$ is finite because $\# \Gamma \leq (\# F)^{\# \cD}$.
\eproof

\bcor
\label{cor:H=UF}
Let $H$ be a compact ample groupoid. Then
\begin{enumerate}[label=(\roman*)]
\item Every $h \in H$ is contained in an element $\sigma \in \bmF(H)$.
\item There is a finite subgroup $F \subseteq \bmF(H)$ such that $H = \bigcup_{\tau \in F} \tau$.
\end{enumerate}
\ecor
\bproof
(i) If $\rmr(h) \neq \rms(h)$, then find a compact open bisection $\rho$ in $H$ with $h \in \rho$ and $\rms(\rho) \cap \rmr(\rho) = \emptyset$. Then $\sigma \coloneq (\rho \amalg \rho^{-1}) \amalg (\Hu \setminus (\rmr(\rho) \amalg \rms(\rho)))$ is in $\bmF(H)$ with $h \in \sigma$.

If $\rmr(h) = \rms(h)$, Lemma~\ref{lem:HxxFU} produces $\sigma \in \bmF(H)$ with $h \in \sigma$.

(ii) By (i), for every $h \in H$, there exists $\sigma_h \in \bmF(H)$ such that $h \in \sigma_h$. So $H = \bigcup_{h \in H} \sigma_h$. By compactness, $H = \bigcup_{i=1}^n \sigma_i$ for some $\sigma_i \in \bmF(H)$. Now $\bmF(H)$ is locally finite by Lemma~\ref{lem:FHLocFin}, so that $\gekl{\sigma_i}$ generates a finite subgroup $F$ of $\bmF(H)$, and we have $H = \bigcup_{i=1}^n \sigma_i \subseteq \bigcup_{\tau \in F} \tau \subseteq H$.
\eproof

\begin{prop}\label{prop:semilocafin}
 Let $G$ be a proper ample groupoid. Then the inverse semigroup $\cS^{\cop}(G)$ of compact open bisections is locally finite.
\end{prop}
\begin{proof}
Put $S=\cS^{\cop}(G)$ and $X=G^{(0)}$, and let $F\subset S$ be a finite subset. Then $U=\bigcup_{\sigma\in F}(d(\sigma)\cup c(\sigma))$ is compact, whence properness of $G$ implies that $G_U^U$ is compact. Moreover, $F\subset\cS^{\cop}(G_U^U)$. Thus we may assume that $G$ is compact. By Corollary \ref{cor:H=UF} there is a finite subgroup 
$\Gamma\subset F(G)$ such that $G=\bigcup_{\gamma\in \Gamma}\gamma$. Because $F$ is finite, for every $x\in X$ there exists a compact open 
$x\in U_x\subset X$ such that if $g\in F$ and $x\in d(g)$ then $U_x\subset d(g)$ and there is a $\gamma(g,x)\in \Gamma$ such that $gU_x=\gamma(g,x)U_x$. Because $X$ is compact, 
there are $x_1,\dots,x_n\in X$ such that for $U_i=U_{x_i}$ we have $X=\bigcup_{i=1}^nU_i$. Disjointifying $\cU=\{U_i\colon 1\le i\le n\}$ we obtain a compact open partition of $X$ refining $\cU$, which by finiteness of $\Gamma$ we may further refine to a $\Gamma$-invariant compact open partition $\cP=\{V_j\colon 1\le j\le m\}$.  By construction, for each $g\in F$ and $1\le j\le m$ there is $\gamma(g,j)\in\Gamma$ such that $gV_j=\gamma(g,j)V_j$. Hence if $g\in F$, then for $\Lambda(g)=\{1\le j\le m\colon V_j\subset d(g)\}$ we have
\[
g=\bigsqcup_{j\in\Lambda(g)}\gamma(g,j)V_j. 
\]
Hence $F$ is a subset of 
\[
\Gamma':=\{\bigsqcup_{j\in \Lambda}\gamma_jV_j\colon \Lambda\subset\{1,\dots, m\},\, \gamma:\Lambda\to \Gamma\}.
\]
It follows from the facts that $\Gamma\subset F(G)$ is a subgroup and that $\cP$ is $\Gamma$-invariant, that $\Gamma'$ is an inverse subsemigroup of $S$. This concludes the proof. 
\end{proof}

\numberwithin{equation}{section}
\bigskip
\section{\topdf{$K$}{K}-theory of Cuntz-Pimsner-Toeplitz rings}\label{app2:lescure}
For the sake of completenes, we review and reprove Lescure's results from \cite{lescure} about the $K$-theory of the Pimsner-Toeplitz rings introduced in \cite{toke}.

Let $R$ be a (nonzero) ring with local units; put $\otimes=\otimes_R$. An \emph{$R$-system} is a triple $\cX=(P,Q,\psi)$ consisting of two
$R$-bimodules $P$ and $Q$ and a bimodule homomorphism $\psi:Q\otimes P\to R$. We write $\langle q,p\rangle=\psi(q\otimes p)$. Let $\cK(\cX)=P\otimes Q$ be the 
$R$-algebra that results from equipping the $R$-bimodule $P\otimes Q$ with the product
\[
(p_1\otimes q_1)(p_2\otimes q_2)=p_1\langle q_1,p_2\rangle\otimes q_2.
\]
Remark that $\cK(\cX)$ acts by right $R$-module endomorphisms on $P$ and by left $R$-module endomorphisms on $Q$ via
\[
(p\otimes q)\cdot p'=p\langle q,p'\rangle,\,\, q'\cdot (p\otimes q)=\langle q',p\rangle q.
\]
Hence we may regard $P$ and $Q$ respectively as a $(\cK(\cX),R)$ and as an $(R,\cK(\cX))$-bimodule. The $R$-system $\cX$ is said to satisfy \emph{condition FS} if for any finite subsets $F\subset P$ and $F'\subset Q$ there exist elements $\Theta,\Delta\in\cK(\cX)$ such that $\Theta\cdot f=f$ and $f'\cdot \Delta=f'$ for all $f\in F$ and $f'\in F'$. 

\begin{rem}\label{rem:unitality}
Because we are assuming that $R$ has local units, the canonical maps $M\otimes R\to M$ and $R\otimes N\to N$ are injective for any right $R$-module $M$ and any left $R$-module $N$. If $\cX=(P,Q,\psi)$ satisfies FS, then it follows from the definitions that $P\otimes R\to P$ and $R\otimes Q\to Q$ are also surjective, and therefore bijective.     
\end{rem}

Let $\{\cX_i=(P_i,Q_i,\psi_i)\colon i\in I\}$ be a family of $R$-systems. Write $\bigoplus_{i\in I}\cX_i$ for the $R$-system consisting of the bimodules $P=\bigoplus_{i\in I}P_i$ and $Q=\bigoplus_{i\in I}Q_i$ and the composite
\[
\psi:Q\otimes P\onto \bigoplus_{i\in I}Q_i\otimes P_i\overset{\sum\psi_i}{\lra}R
\]
\begin{lem}\label{lem:sumfs}
 Let $\{\cX_i\colon i\in I\}$ be a family of $R$-systems and let $\cX=\bigoplus_{i\in I}\cX_i$. If $\forall i\in I$ $\cX_i$ satisfies FS, then $\cX$ satisfies FS.
\end{lem}
\begin{proof}
Let $F, F'\subset I$  and $F_i, F'_i\subset \Nz$ $(i\in I)$ be finite and $\{x_{i,j}\colon i\in F, \, j\in F_i\}\subset P$ and $\{x'_{i,j}\colon i\in F', j\in F'_j\}\subset Q$ with $x_{i,j}\in P_i$ and $y_{i,j}\in Q_i$. For $i\in F$ and $i'\in F'$, choose 
$\Theta_i\in P_i\otimes Q_i$, $\Delta_{i'}\in P_{i'}\otimes Q_{i'}$ such that $\Theta_i x_{i,j}=x_{i,j}$ and $y_{i',j'}\Delta_{i'}=y_{i',j'}$ for all $j\in F_i$ and $j'\in F'_{i'}$. Then $\Theta=\sum_{i\in F}\Theta_i$ and $\Delta_{i'\in F'}\Delta_{i'}$ satisfy $\Theta x_{i,j}=x_{i,j}$ and $y_{i',j'}\Delta=y_{i',j'}$ for all $i,j,i'$ and $j'$.
\end{proof}

A \emph{weak homomorphism} of $R$-systems $f:\cX=(P,Q,\psi)\to \cX'=(P',Q',\psi')$ is a pair $f=(f_P,f_Q)\in\Hom_R(P_R,P'_R)\times \Hom_R({}_RQ,{}_RQ')$ such that the following diagram commutes
\[
\xymatrix{Q\otimes_{\Z} P\ar[rr]^{f_Q\otimes f_P}\ar[d]&& Q'\otimes_\Z P'\ar[d]\\
           Q\otimes P\ar[r]^{\psi}& R& Q'\otimes P'\ar[l]_{\psi'}}
\]
We do not require $f_P$ to be a left $R$-module homomorphism, nor $f_Q$ to be a right $R$-module homomorphism; if $f$ has both of these properties, we say that $f$ is \emph{strong}. For any weak homomorphism $f$ as above, $f_P\otimes f_Q:P\otimes R\to P'\otimes Q'$ is well-defined and gives a ring homomorphism $\cK(f):\cK(\cX)\to\cK(\cX')$.

Because $R$ has local units, multiplication induces a bimodule isomorphism $\mu:R\otimes R\isom R$; we shall write $R$ for the $R$-system
$(R,R,\mu)$, and set $R^{(I)}=\bigoplus_{i\in I}R$. Let $I$ be a nonempty set and $\cX$ an $R$-system. We say that $\cX$ has \emph{property} $BL(I)$ if there exists a weak homomorphism $f:\cX\to R^{(I)}$. We say that $\cX$ has property BL if it has property $BL(I)$ for some $I$.

\begin{ex}
Let $G$ be a discrete groupoid, $E$ a (directed) graph, $(G,E,\sigma)$ a self-similar groupoid action in the sense of 
\cite{steimil}*{Definition 2.3} and $E_r\times_cG$ the associated $G$-correspondence as defined in \emph{loc. cit.} immediately after 
said definition. Let $E^*$ be another copy of $E$, and let $G_{d}\times_rE^*$ be the dual correspondence, equipped with the obvious left action 
and the right action $\sigma^*:E^*_s\times_c G\to G$, $\sigma^*(e^*,g)=\sigma(g,g^{-1}(e))$. Let $k$ be a commutative unital ring; set 
$R=k[G]$. We have $R$-bimodules $P=k[E_{r}\times_cG]$,
and $Q=k[G_d\times_rE^*]$. One checks that the map $Q\times P\to R$ defined by
\[
\psi(\chi_{(g,e^*)},\chi_{(f,h)})=\delta_{e,f}gh
\]
is $R$-balanced and $R$-bilinear. Thus it induces an $R$-bimodule map $\psi:Q\otimes_RP\to R $ so that $\cX=(P,Q,\psi)$ is an 
$R$-system. One checks further that $\cX$ satisfies properties FS and BL. 
\end{ex}
\begin{lem}\label{lem:yira}
Let $f^1,f^2:\cX\to\cY$ be weak homomorphisms of $R$-systems. Let $H:\rings\to \fC$ be a functor. If $H$ is $M_2$-stable, then 
$H(f^1)=H(f^2)$.
\end{lem}
\begin{proof}
One checks that $\cK(\cX\oplus \cX)=M_2(\cK(\cX))$. Using this identification, and writing $\iota_1$ and $\iota_2$ for the inclusion in upper left and lower right corner respectively, we have a commutative diagram
\[
\xymatrix{M_2\cK(\cX)\ar[rr]^{\cK(f^1\oplus f^2)}&&M_2\cK(\cY)\\
         \cK(\cX)\ar[u]^{\iota_i}\ar[rr]^{\cK(f^i)}&&\cK(\cY)\ar[u]^{\iota_i}}
\]
By \cite{friendly}*{Lemma 2.2.4}, $M_2$-stability of $H$ implies that $H(\iota_1)=H(\iota_2)$. Hence we have
\begin{gather*}
H(\iota_1)\circ H(\cK(f^1))=H(\cK(f^1\oplus f^2))\circ H(\iota_1))=H(\cK(f^1\oplus f^2))\circ H(\iota_2)=\\
H(\iota_2)\circ H(\cK(f_2))=H(\iota_1)\circ H(\cK(f_2)).
\end{gather*}
Because $H$ is $M_2$-stable, $H(\iota_1)$ is an isomorphism, so the above identities imply that $H(\cK(f_1))=H(\cK(f_2))$. 
\end{proof}
\begin{lem}[\cite{lescure}*{Proposition 2.11}]\label{lem:kmorita}
Let $R$ be a ring with local units and $\cX$ an $R$-system satisfying $BL(I)$. Let $\inc_{\cX}:R\to R\oplus\cX$ be the inclusion; set $\iota_{\cX}=\cK(\inc_\cX)$. Let $X$ be an infinite set with $|X|\ge |I|$ and $H$ an $M_X$-stable functor. Then $j(\iota_{\cX})$ is an isomorphism. 
\end{lem}
\begin{proof}
Let $\rho=\id_R\oplus f:R\oplus\cX\to R\oplus R^{(I)}$. We have 
$\rho\circ\inc_{\cX}=\inc_{R^{(I)}}$; hence $H(\cK(\rho))\circ 
H(\iota_{\cX})=H(\iota_{R^{(I)}})$ which is an isomorphism by matricial stability. Thus 
$H(\iota_{\cX})$ is a split mono. To finish, it suffices to prove that $H(\cK(\rho))$ is a 
split epi. Let $\inc'=\inc_{\cX}\oplus\bigoplus_{i\in I}\inc_{\cX}:R\oplus R^{(I)}\to R\oplus\cX\oplus (R\oplus\cX)^{(I)}$ and 
$\inc":R\oplus\cX\to R\oplus\cX\oplus (R\oplus \cX)^{(I)}$ the inclusion into the 
first summand. By Lemma \ref{lem:yira}, $H(\cK(\inc'))\circ H(\cK(\rho))=H(\cK(\inc'\circ\rho))=H(\cK(\inc"))$ which is an isomorphism by matricial stability. 
\end{proof}
 Let  $\cX_i=(P_i,Q_i,\psi_i)$ $(i\ge 1)$ be $R$-systems. Define recursively $\psi_{[1,1]}=\psi_1$ and for $n\ge 1$
 \begin{gather*}
 \psi_{[1,n+1]}:Q_{n+1}\otimes\cdots\otimes Q_1\otimes P_1\otimes\cdots\otimes P_{n+1}\to R\\
 \psi_n(q_{n+1}\otimes\psi_{[1,n]}(q_{n}\otimes\cdots\otimes q_1\otimes p_1\otimes\cdots\otimes p_n)\otimes p_{n+1})
 \end{gather*}
 Set $\cX_1\otimes\cdots\otimes \cX_n=(P_1\otimes\cdots\otimes P_n,Q_n\otimes\cdots\otimes Q_1,\psi_{[1,n]})$. 

\begin{lem}\label{lem:bltenso}
Let $I$ be an infinite set. Let $\{\cX_n\colon n\in \Nz\}$ be a sequence of $R$-systems. Assume that $\cX_m$ satisfies $BL(I)$ for all $m$. Then $\cX_1\otimes\cdots\otimes\cX_n$ and $\bigoplus_{m\ge 1}\cX_m$ satisfy $BL(I)$.
\end{lem}
\begin{proof}
The assertion about the direct sum is clear. Fix $n\ge 2$; for each $1\le \ell\le n$, let $f^\ell:\cX_\ell\to R^{(I)}$ be a weak homomorphism of $R$-systems. Also let $\pi_i:R^{(I)}\to R$ be the projection onto the $i$-th coordinate. Write $f^\ell_i$ for both $\pi_i\circ f_{P_\ell}^{\ell}$ and $\pi_i\circ f_{Q_\ell}^{\ell}$. Put $P=P_1\otimes\cdots\otimes P_n$ and $Q=Q_n\otimes\cdots\otimes Q_1$. We shall define a weak homomorphism $f:\cX_1\otimes\cdots\otimes\cX_n\to R^{(I^n)}$; abusing notation, we write $f$ for both $f_P$ and $f_Q$. We proceed recursively; for $n\ge 1$, we set
\begin{gather*}
f(p_1\otimes\cdots\otimes p_{n+1})=\sum_{(i_1,\dots,i_{n+1})}\epsilon_{(i_1,\dots,i_{n+1})}f_{i_{n+1}}(f_{i_n,\dots,i_1}(p_1\otimes\cdots\otimes p_n)p_{n+1})\\
f(q_{n+1}\otimes\cdots\otimes q_1)=\sum_{(i_1,\dots,i_{n+1})}\epsilon_{(i_1,\dots,i_{n+1})}f_{i_{n+1}}(q_{n+1}f_{i_n,\dots,i_1}(q_n\otimes\cdots\otimes q_1))\\
\end{gather*}
We have
\begin{gather*}
\langle f(q_{n+1}\otimes\cdots\otimes q_1),f(p_1\otimes\cdots\otimes p_{n+1})\rangle\\
=\sum_{(i_1,\dots,i_{n+1})}f_{i_{n+1}}(q_{n+1}f_{i_n,\dots,i_1}(q_n\otimes \cdots\otimes q_1))f_{i_{n+1}}(f_{i_n,\dots,i_1}(p_1\otimes \cdots\otimes p_n)p_{n+1})\\
=\sum_{(i_1,\dots,i_n)}\langle q_{n+1}f_{i_n,\dots,i_1}(q_n\otimes \cdots\otimes q_1),f_{i_n,\dots,i_1}(p_1\otimes \cdots\otimes p_n)p_{n+1} \rangle \\
=\langle q_{n+1}\sum_{(i_1,\dots,i_n)}f_{i_n,\dots,i_1}(q_n\otimes \cdots\otimes q_1)f_{i_n,\dots,i_1}(p_1\otimes \cdots\otimes p_n),p_{n+1}\rangle\\
=\langle q_{n+1}\langle q_n\otimes\cdots\otimes q_1, p_1\otimes\cdots\otimes p_n\rangle,p_{n+1}\rangle\\
=\langle q_{n+1}\otimes\cdots\otimes q_1,p_1\otimes\cdots\otimes q_{n+1}\rangle
\end{gather*}
\end{proof}
\begin{lem}\label{lem:bimo}
Let $f:R\to S$ be a ring homomorphism, $\tilde{S}=S\oplus\Z$ its unitalisation and $M$ an $R$-bimodule. Assume that the canonical bimodule homomorphisms $M\otimes_R R\to M$ and $R\otimes_R S\to S$ are isomorphisms. Then the $R$-bimodule homomorphism induced by the inclusion $\inc:
S\subset \tilde{S}$
\[
\widebar{\inc}:M\otimes_R S\to M\otimes_R\tilde{S}
\]
is an isomorphism.
\end{lem}
\begin{proof}
Let $b:R\times \tilde{S}\to S$, $b(a,s)=f(a)s$; since $b$ is $R$-bilinear and $R$-balanced, it induces a bimodule homomorphism $\bar{b}:R\otimes_R\tilde{S}\to S$. By construction, the following diagram commutes
\[
\xymatrix{R\otimes_R\tilde{S}\ar[r]& S\\
          R\otimes_RS\ar[u]\ar[ur]^{\cong}&}
\]
It follows that the vertical map is split-injective. Hence it remains so upon left-tensoring with $M$, which because we are assuming that $M\otimes_RR\to M$ is an isomorphism, implies that $\widebar{\inc}$ is injective. Moreover, 
$\widebar{\inc}$ is surjective because $M\otimes_RR\to M$ is. 
\end{proof}

Let $\cX$ be an $R$-system. The \emph{Toeplitz ring} of $\cX$ is 
$$
\cT(\cX)=T_R(P\oplus Q)/\langle q\otimes p-\psi(q\otimes p)\colon p\in P,\, q\in Q\rangle
$$
Here $T_R$ is the tensor algebra. By \cite{toke}*{Theorem 1.7}, $\cT(\cX)=\bigoplus_{i,j\ge 0}P^{\otimes i}\otimes Q^{\otimes j}$ with the product induced by $\psi$, as described in \emph{loc. cit.}

Fix a commutative unital ring $\ell$. In what follows by a \emph{homology theory} of $\ell$-algebras we understand a functor $H:\ahal\to \cT$ taking values in a triangulated category 
$\cT$. Let $X$ be an infinte set and $x\in X$. A homology theory $H$ is \emph{excisive} if it sends $\ell$-linearly split algebra extensions to distinguished triangles, 
\emph{homotopy invariant} if it inverts the inclusion $A\subset A[t]$, and \emph{$M_X$-stable} if it inverts the corner inclusion $\iota_x:A\to M_XA$, $\iota_x(a)=\epsilon_{x,x}a$. 
For every infinite set $X$, there is a homology theory $j_X:\ahal\to kk_X$, called \emph{bivariant algebraic $K$-theory}, that is homotopy invariant, excisive and $M_X$-stable and is 
initial among homology theories satisfying those three properties (\cites{ct,ematesis}). We say that a homomorphism of $\ell$-algebras $f$ is a \emph{$kk_X$-equivalence} if $j_X(f)$ is an isomorphism. 
Like $M_X$, $j_X$ and $kk_X$ depend only on the cardinal of $X$; however $kk_X(\ell,A)=KH_0(A)$ for all choices of $\ell$ and $X$ (\cites{ct,ematesis}). Since for $|Y|\ge |X|$ $M_Y$-stability implies 
$M_X$-stability, by universal property of $j_X$ there is a canonical triangulated functor $\bar{j}_Y:kk_X\to kk_Y$ such that $\bar{j}_Y\circ j_X=j_Y$. Hence if an $\ell$-algebra 
homomorphism is a $kk_X$-equivalence then it is a $kk_Y$-equivalence whenever $|Y|\ge |X|$. In what follows, every time we state that a certain algebra map is a $kk$-equivalence, we 
will indicate a set $X$ of cardinal large enough for this to hold, and write $j=j_X$ and $kk=kk_X$.  

\begin{thm}[\cite{lescure}*{Theorem 3.2}]\label{thm:toeplitz}
Let $R$ be a (nonzero) ring with local units and let $\cX$ be an $R$-system. Assume that $\cX$ satisfies FS and $BL(I)$, and that the canonical $R$-bimodule maps $R\otimes P\to P$ and $Q\otimes R\to Q$ are isomorphisms. Consider the natural inclusion $\inc:R\subset \cT(\cX)$. For $X$ as in Lemma \ref{lem:kmorita}, let $j:\rings\to kk$ be the $M_X$-stable bivariant algebraic $K$-theory functor. Then $j(\inc)$ is an isomorphism.
\end{thm}
\begin{proof}
As in \cite{toke}*{Definition 2.4}, we consider the $R$-system 
$T\cX=(T_R(P),T_R(Q),\sum_{n\ge 0}\psi_n)$ and the corresponding ring of 
adjointable operators $\cL(T_R(P))$. By \cite{toke}*{Lemma 3.8}, the $R$-system $(P^{\otimes n},Q^{\otimes n},\psi_n)$ satisfies condition FS for all $n$; hence $T\cX$ also does, by Lemma \ref{lem:sumfs}. In particular $\cK(\cX)\subset \cL(T_R(P))$ by 
\cite{toke}*{Proposition 3.11}. Remark that the projection
$\pi_+:T_R(P)\to T_R(P)$ onto $T^+_R(P)=\bigoplus_{n\ge 1}P^{\otimes n}$ is 
an element of $\cL(T_R(P))$, whose adjoint is the projection of $T_R(Q)$ 
onto $T_R^+(Q)$. Write $\cT=\cT(\cX)$. Let $\rho_0:\cT\to \cL(T_R(P))$ be the Fock 
representation of \cite{toke}*{Proposition 2.5} and let $\rho_1:
\cT\to \cL(T_R(P))$ be the unique ring homomorphism such that 
$\rho_1(x)=\pi_+\circ\rho_0(x)\circ\pi_+$ for all $x\in R\sqcup P\sqcup Q$. 
Remark that $(\rho_0,\rho_1)$ is a quasi-homomorphism taking values in 
$\cK(T_R(P))=\cK(R\oplus T_R^+(P))\triqui\cL(T_R(P))$. By Lemma \ref{lem:bltenso}, $T_{R}^+(P)$ satisfies BL. Hence by Lemma 
\ref{lem:kmorita}, for the map  $\iota=\iota_{T_R^+(P)}$ of said Lemma, $j(\iota)$ is an isomorphism. One checks that 
$\rho_0\circ\inc-\rho_1\circ\inc$ is the composite of $\iota$ followed by the inclusion $\cK(T_R(P))\subset \cL(T_R(P))$. Hence 
$j(\rho_0,\rho_1)\circ j(\inc)=j(\iota)$, so $j(\inc)$ is a split mono and $j(\rho_0,\rho_1)$ is a split epi. It remains to show that 
$j(\rho_0,\rho_1)$ is a split mono. Let $\mu:\cT\otimes R\otimes\cT\to \cT$ be the multiplication map. Consider the $\cT$-system
$\cY=(TP\otimes\cT,\cT\otimes TQ,\mu\circ (\id_{\cT}\otimes\sum_n\psi_n\otimes\id_{\cT}))$. Our hypothesis that $R$ has local units 
and that $R\otimes P\to P$ and $Q\otimes R\to Q$ are isomorphisms imply that also $R\otimes T_RP\to T_RP$ and $T_RQ\otimes R\to T_RQ$ 
are isomorphisms. This together with Remark \ref{rem:unitality} implies that $R\otimes\cT\to\cT$ and $\cT\otimes R\to \cT$ are 
isomorphisms. It follows that any system of local units of $R$ is also a system of local units for $\cT$. In particular $\cT$ has 
local units, and therefore the canonical map
$\cT\otimes_\cT\cT\to \cT$ is an isomorphism. Hence $\cK(\cY)=T_RP\otimes \cT\otimes T_RQ$. Further, because as mentioned above, 
$T\cX$ satisfies FS, the canonical maps $T_RP\otimes R\to T_RP$ and 
$R\otimes T_RQ\to T_RQ$ are isomorphisms by Remark \ref{rem:unitality}. Thus by Lemma \ref{lem:bimo} the inclusion 
$\cT\subset\tilde{\cT}$ induces an 
isomorphism  
\[
T_RP\otimes\cT\otimes T_RQ\isom T_RP\otimes\tilde{\cT}\otimes T_RQ.
\]
Using this, one checks that $\cY$ satisfies FS. (If, for example, $x_1,\dots,x_n\in TP$ and $\Theta=\sum_i\Theta_i^1\otimes \Theta^2_i\in\cK(T\cX)$ satisfies $\Theta\cdot x_i=x_i$ for all $i$, then $\Theta'=\sum_i\Theta^1_i\otimes 1\otimes \Theta^2_i\in\cK(\cY)$ satisfies $\Theta'\cdot x_i\otimes a_i=x_i\otimes a_i$ for all $i$ and all $a_i\in\cT$.) Hence $\cK(\cY)\subset\cL(T_RP\otimes\cT)$. Moreover, the fact, mentioned above, that $T_RP\otimes R\to T_RP$ and 
$R\otimes T_RQ\to T_RQ$ are isomorphisms implies that $R\otimes \cT\to \cT$ and $\cT\otimes R\to \cT$ are isomorphisms. It follows 
from this and the fact that $\cT\cX$ satisfies BL that also $\cY$ satisfies it; thus $j(\iota_{\cY})$ is an isomorphism by Lemma 
\ref{lem:kmorita}. One checks that the assignment
$A\mapsto A\otimes\id_{\cT}$ defines a ring homomorphism $\nu:\cL(T_R(P))\to\cL(T_R(P)\otimes\cT)$ that sends the 
operator of left multiplication by $p\otimes q\in\cK(\cX)$ into that of left multiplication by $p\otimes 1\otimes q\in\cK(\cY)$. Hence 
$\nu(\cK(\cX))\subset\cK(\cY)$. Let $\phi_0=\nu\circ\rho_0$ and $\phi_1=\nu\circ\rho_1$; to finish the proof, it suffices to show that 
$j(\phi_0,\phi_1)=j(\iota_{\cY})$.  Because the images of $\iota_{\cY}$ and $\phi_1$ are orthogonal, the map 
$\psi_{1/2}=\iota_{\cY}+\phi_1$ is a ring homomorphism. It follows from the basic properties of quasi-homomorphisms (listed in 
\cite{ct}*{Theorem 5.2.3}, e.g.) that $j(\phi_0,\phi_1)=j(\phi_0,\phi_{1/2})+j(\phi_{1/2},\phi_1)$ and 
$j(\phi_{1/2},\phi_1)=j(\iota_{\cY})$. Hence we will be done if show that $j(\phi_0,\phi_{1/2})=0$. For $n\ge 0$, let 
$\pi_n:T_RP\to T_RP$ be the projection onto $P^{\otimes n}$. Remark that $\iota(y)=\pi_0\otimes y$ for all $y\in \cT$. Let 
$x\in R\sqcup P\sqcup Q$; consider the following elements of the polynomial ring $\cL(T_RP\otimes\cT)[t]$
\begin{equation}\label{map:prescript}
H(x)=\left\{\begin{matrix} \phi_0(x) & \text{ if }x\in R\\
                            \phi_1(x)+(1-t^2)\rho_0(x)\pi_0\otimes 1+(2t-t^3)\iota(x) & \text{ if } x\in P\\
                            \phi_1(x)+(1-t^2)\rho_1(x)\pi_1\otimes 1+t\iota(x) & \text{ if } x\in Q\end{matrix}\right.
\end{equation}
Let $p\in P$ and $q\in Q$. Then, writing $1=\id_{\cT}$ and identifying an element $x\in\cT$ with left multiplication by the same element, we have
\begin{gather*}
H(q)H(p)=((1-t^2)\rho_0(q)\pi_1\otimes 1+t\pi_0\otimes q+\pi_+\rho_0(q)\pi_+\otimes 1)\cdot\\
((1-t^2)\rho_0(p)\pi_0\otimes 1+(2t-t^3)\pi_0\otimes p+\rho_0(p)\pi_+\otimes 1)\\
=(1-t^2)^2\langle q,p\rangle\pi_0\otimes 1+t(2t-t^3)\pi_0\otimes\langle q,p\rangle+\langle q,p\rangle\pi_+\otimes 1\\
=\langle q,p\rangle \pi_0\otimes 1+\langle q,p\rangle\pi_+\otimes 1=\langle q,p\rangle \otimes 1=H(\langle q,p\rangle).
\end{gather*}
Thus the prescriptions \eqref{map:prescript} determine a ring homomorphism $H:\cT\to \cL(T_RP\otimes\cT)[t]$, and the latter is a homotopy between the ring homomorphisms $\psi_0$ and $\psi_{1/2}$. Furthermore, one checks that $H(x)-\psi_{1/2}(x)\in \cK(\cY)$, and thus $(H,\psi_{1/2})$ is a homotopy between $(\psi_0,\psi_{1/2})$ and the trivial quasi-homomorphism $(\psi_{1/2},\psi_{1/2})$. It follows that $j(\psi_0,\psi_{1/2})=0$. This finishes the proof. 
\end{proof}

\end{document}